\documentclass[11pt,a4paper]{amsart}
\usepackage[utf8]{inputenc}
\usepackage[T1]{fontenc}
\usepackage{xcolor}
\usepackage[latin9]{luainputenc}
\usepackage{geometry}
\usepackage[toc,page]{appendix}
\usepackage[active]{srcltx}

\usepackage{babel}
\usepackage{amsthm}
\usepackage{amssymb}

\usepackage{amsmath,bbm,esint}
\usepackage{amssymb}
\usepackage{accents,calc}
\usepackage{subfig}
\usepackage{caption}
\usepackage{dsfont}
\usepackage{mwe}
\usepackage{comment}
\usepackage{tikz}
\usepackage{amsfonts}
\usepackage{dsfont}
\usepackage{enumitem} 
\usepackage{xcolor}
\usepackage{mathrsfs} %% package of Lebesgue es Hausdorff measures
\usepackage[colorlinks=true,urlcolor=blue,
citecolor=red,linkcolor=blue,linktocpage,pdfpagelabels,bookmarksnumbered,bookmarksopen]{hyperref}
\usepackage[normalem]{ulem}

\usepackage{xcolor}

\usepackage{collectbox}



\def\e{\varepsilon}

\def \p{\partial}

\def \0{\mathbf{0}}

\newcommand{\R}{{\mathbb R}}

\newcommand{\N}{\mathbb{N}}

\usepackage{subfig}

\newcommand{\spt}{{\rm{spt}}}

\newcommand{\PP}{\mathcal{P}}

\newcommand{\be}{\begin{equation}}
\newcommand{\ee}{\end{equation}}

\newcommand{\ba}{\begin{array}}
\newcommand{\ea}{\end{array}}

\usepackage{subfiles}
\newtheorem{theorem}{\textbf{Theorem}}[section]
\newtheorem{lemma}[theorem]{\textbf{Lemma}}
\newtheorem{proposition}[theorem]{\textbf{Proposition}}
\newtheorem{corollary}[theorem]{\textbf{Corollary}}
\newtheorem{example}[theorem]{\textbf{Example}}
\newtheorem{definition}[theorem]{\textbf{Definition}}
\newtheorem{remark}[theorem]{\textbf{Remark}}

\newtheorem{innercustomgeneric}{\customgenericname}
\providecommand{\customgenericname}{}
\newcommand{\newcustomtheorem}[2]{%
  \newenvironment{#1}[1]
  {%
   \renewcommand\customgenericname{#2}%
   \renewcommand\theinnercustomgeneric{##1}%
   \innercustomgeneric
  }
  {\endinnercustomgeneric}
}

\newcustomtheorem{customthm}{Framework}

\newtheorem{assumption}{\textbf{Assumption}}
\numberwithin{equation}{section}

\begin{document}
\title[Accelerated Wasserstein Gradient Flows
]{Guaranteeing Higher Order Convergence Rates for Accelerated Wasserstein gradient Flow Schemes}

\author{Raymond Chu}
\address{Department of Mathematical Sciences, Carnegie Mellon University,
Pittsburgh, PA 15213, USA}
\email{raymondchu@cmu.edu}

\author{Matt Jacobs}
\address{Department of Mathematics, University of California, Santa Barbara,
CA 93106, USA}
\email{majaco@ucsb.edu}

\begin{abstract} 
In this paper, we study higher-order-accurate-in-time minimizing movements schemes for Wasserstein gradient flows.  We introduce a novel accelerated 
second-order scheme, leveraging the differential structure of the Wasserstein space in both Eulerian and Lagrangian coordinates. For sufficiently smooth energy functionals, we show that our scheme provably achieves an optimal quadratic convergence rate. Under the weaker assumptions of Wasserstein differentiability and $\lambda$-displacement convexity (for any $\lambda\in \R$), we show that our scheme still achieves a first-order convergence rate and has strong numerical stability. In particular, we show that the energy is nearly monotone in general, while when the energy is $L$-smooth and $\lambda$-displacement convex (with $\lambda>0$), we prove the energy is non-increasing and the norm of the Wasserstein gradient is exponentially decreasing along the iterates.  Taken together, our work provides the first fully rigorous proof of accelerated second-order convergence rates for smooth functionals and shows that the scheme performs no worse than the classical scheme JKO scheme for functionals that are $\lambda$-displacement convex and Wasserstein differentiable.
\end{abstract}


\maketitle 

%\textcolor{red}{feel free to change the keywords/MSC. it was just something rough I came up with.}

\tableofcontents

\section{Introduction}

Given an energy functional $\phi:\PP_2(\R^d)\to\R\cup\{+\infty\}$, the continuity equation
\begin{equation}
\label{gradient_flow_eqn12}
\begin{cases}
 \partial_t \rho(t,x) - \nabla\cdot \Bigl(\rho(t,x) \nabla\Bigl(\tfrac{\delta \phi}{\delta \mu}\bigl(\rho(t,\cdot),x\bigr)\Bigr)\Bigr) = 0 \quad \text{on } (0,\infty)\times\mathbb{R}^d,\\[6pt]
 \rho(0,\cdot) = \rho_0(\cdot) \quad \text{on } \mathbb{R}^d,
\end{cases}
\end{equation} can be interpreted formally as the gradient flow of $\phi$ on the Wasserstein space $(\PP_2(\R^d), \mathbf{W}_2)$.
This
class of equations models various important physical phenomena such as fluid flow, heat transfer, aggregation-diffusion, and crowd motion \cite{vazquez2007porous, santambrogio2015optimal} to name a few. In general, these equations are both stiff and non-linear making them challenging to solve numerically.  Perhaps the most well-known \emph{stable} numerical method for solving these equations is the celebrated JKO scheme \cite{jordan1998variational}, an unconditionally energy stable variational scheme,  well-known to be first-order-accurate-in-time with respect to the time step $\tau>0$ \cite{ambrosio2005gradient}.   Recently, there has been a great deal of interest in formulating new versions of the JKO scheme, which achieve higher order accuracy in time, while maintaining favorable stability properties \cite{legendre2017second, matthes2019variational, ashworth2020structure,  zaitzeff_hos, carrillo2022primal, han_hos, gallouet2024geodesic, cances2024discretizing}.  However, to the best of our knowledge, none of these methods to date have been able to rigorously prove a second (or higher) order convergence rate in any metric, even for smooth functionals, despite promising numerical evidence.  Indeed, the only rigorous rates that have appeared in the literature for these accelerated methods are $O(\sqrt{\tau})$, worse than the $O(\tau)$ rate of the simpler JKO scheme.  The goal of this manuscript is to remedy this situation.  We provide a new energy stable scheme that we rigorously demonstrate is second-order-in-time for smooth energy functionals and at worst first-order-accurate-in-time for $\lambda$-displacement convex and Wasserstein differentiable functionals (see Section \ref{prelim} for the precise definition of Wasserstein differentiability and \eqref{model_energy_functional_123} for the main model functionals that we consider).

To construct our scheme, we exploit both the Eulerian (Wasserstein) and Lagrangian $L^2$ gradient flow structures of the equation to construct a novel second-order-accurate-in-time numerical scheme for solving \eqref{gradient_flow_eqn12}. To motivate this Lagrangian $L^2$ perspective, let us assume that the velocity field $v=-\nabla \bigl( \frac{\delta \phi}{\delta \mu} \bigr)$ is sufficiently regular, so that the solution $\rho(t,x)$ can be expressed as the push-forward of the initial data $\rho_0\in \PP_2(\R^d)$ by the Lagrangian flow $X(t,x)$, i.e.
\begin{equation}
\rho_t(x) = (X(t,\cdot)_{\#} \rho_0)(x). \label{pushforward_rho_def}
\end{equation}
where
\begin{equation}
\begin{cases} 
\displaystyle \frac{d}{dt}X(t,x)  =-\nabla \left[ \frac{\delta \phi}{\delta \mu}( X(t,\cdot)_{\#} \rho_0,x) \right] \quad \text{on } [0,\infty),  \\
X(0,x) = x, \label{characteristic_eqn}
\end{cases}
\end{equation}

As we will see below, the Lagrangian flow \eqref{characteristic_eqn} is a gradient flow of the lifted energy functional $X\mapsto \phi^{\#}_{\rho_0}(X)$ over the Hilbert space $L^2(\R^d;\rho_0)$ where
\begin{equation} \phi^{\#}_{\rho_0}(X) := \phi(X_{\#} \rho_0). \label{lift_intro_def} \end{equation} That is, if $\nabla \phi^{\#}_{\rho_0}(X)$ denotes the $L^2(\R^d;\rho_0)$ Fréchet derivative of the map $X \mapsto \phi^{\#}_{\rho_0}(X)$, then
\begin{equation}
\begin{cases} 
\displaystyle \frac{d}{dt}X(t,x) = -\nabla \phi^{\#}_{\rho_0}(X(t,x)) \quad \text{on } [0,\infty),  \\
X(0,\cdot) = \textnormal{Id}. \label{L2_Lagrangian_Flow}
\end{cases}.
\end{equation}

 Although the lifted energy $X\mapsto \phi^{\#}_{\rho_0}(X)$ is a good deal more complicated than the original energy $\rho\mapsto \phi(\rho)$, the advantage of this perspective is that we can much more readily generalize high-order-accurate numerical discretizations of Euclidean differential equations to the Hilbert space $L^2(\R^d,\rho_0)$, rather than needing to search for their correct analogue over the Wasserstein space. From a theoretical standpoint, it is often preferable to work with the lifted functional because its higher order derivatives encode higher order geometric information such as geodesic convexity (see \cite[Lemma 3.6]{gangbo2019differentiability} and Theorem \ref{convex_equivalence}). In contrast, naively taking the Wasserstein gradient of an energy functional twice yields only a \emph{partial} Wasserstein Hessian. This partial Hessian does not capture the full geometric information: for instance, the potential energy is linear, so its partial Hessian vanishes identically. Indeed, as we will see in Remark \ref{Hessian_displacement}, the second time derivative of the energy functional along geodesics depend on both the partial Wasserstein Hessian and the spatial gradient of the Wasserstein gradient.

To obtain a second-order-in-time discretization of \eqref{gradient_flow_eqn12}, we apply the second order trapezoidal finite difference scheme to \eqref{L2_Lagrangian_Flow}. In particular, we consider
\begin{equation}
\begin{cases}
X^{\tau}_{n+1}
\in \operatorname*{arg\,min}_{\xi \in L^2(\mathbb{R}^d;\rho_0)}
\left[
\frac{1}{2}\bigl(\phi^{\#}_{\rho_0}(\xi)
+ \langle
\nabla \phi^{\#}_{\rho_0}(X^{\tau}_n),
\xi
\rangle_{L^2(\mathbb{R}^d;\rho_0)}
\bigr)
+ \frac{1}{2 \tau}\|\xi - X^{\tau}_n\|_{L^2(\mathbb{R}^d;\rho_0)}^2
\right], \\
\rho^{\tau}_{n+1} := (X^{\tau}_{n+1})_{\#} \rho_0,
\end{cases}
\label{trapezoid_variational_form}
\end{equation}
where the initial condition $X^{\tau}_0$ is prescribed.
 The optimizers $X^{\tau}_{n+1}$ of \eqref{trapezoid_variational_form} admits the following trapezoidal finite difference scheme:
\begin{equation}
X^{\tau}_{n+1}
= X^{\tau}_n
- \frac{\tau}{2} \bigl(
\nabla \phi^{\#}_{\rho_0}(X^{\tau}_{n+1})
+ \nabla \phi^{\#}_{\rho_0}(X^{\tau}_{n})
\bigr). \label{trapezoid_rule123}
\end{equation}
We note that, unlike the JKO scheme, the trapezoidal update maps
$T^{\tau}_{n+1} := X^{\tau}_{n+1} \circ (X^{\tau}_n)^{-1}$
are not, in general, optimal transport maps between
$\rho_{n}^{\tau}$ and $\rho_{n+1}^{\tau}$, since they need not be gradients of convex functions.

\subsection{Background and previous results}
We mentioned earlier, the Lagrangian flow~\eqref{characteristic_eqn} can also be realized as a gradient flow in the Hilbert space $L^2(\mathbb{R}^d;\rho_0)$ of the \textit{lifted} energy functional \eqref{lift_intro_def}, which was introduced in mean-field games. 
%This lifted energy functional has the following probabilistic interpretation: $\phi^{\#}_{\rho_0}(\xi)$ is the energy functional $\phi$ evaluated along the law of the random variable $\xi$ under the probability measure $\rho_0$. 
As shown in mean-field games \cite{gangbo2019differentiability,gangbo2022global,carmona2018probabilistic,cardaliaguet2019master}, the convexity and differentiability properties of the lift $\phi^{\#}_{\rho_0}$ on $L^2$ are closely related to notions of convexity and differentiability of $\phi$ over $\mathbf{W}_2$. In particular, for a large class of probability measures $\rho_0$, the differentiability of $\phi$ in the Wasserstein sense at the measure $\xi_{\#} \rho_0$ is equivalent to the Fréchet differentiability of $\phi^{\#}_{\rho_0}$ at the function $\xi$ (see \cite[Corollary 3.22]{gangbo2019differentiability}). In this case, we have
\[
\nabla \phi^{\#}_{\rho_0}(\xi) = \nabla_{\mathbf{W}} \phi(\xi_{\#} \rho_0,\xi),
\]
where $\nabla_{\mathbf{W}}$ denotes the Wasserstein gradient, %The above equality can be extended to an equality of the sub-differentials of $\phi$ and $\phi^{\#}_{\rho_0}$ \cite{gangbo2019differentiability}. Here the Wasserstein gradient 
which  when $\phi$ and $\rho$ are sufficiently smooth (see \cite[Lemma 4.12]{ambrosio2007gradient}), can be characterized as
\[
\nabla_{\mathbf{W}} \phi(\rho)(x) = \nabla \frac{\delta \phi}{\delta \mu} (\rho,x).
\]
Thus, formally, the Lagrangian flow \eqref{characteristic_eqn} can be viewed as the gradient flow in the Hilbert space $L^2(\R^d;\rho_0)$ described by \eqref{L2_Lagrangian_Flow}. The implicit Euler discretization of \eqref{L2_Lagrangian_Flow} and its connection to Wasserstein gradient flows \eqref{gradient_flow_eqn12} has been studied in
\cite{junge2017fully,carrillo2021lagrangian,evans2005diffeomorphisms,ambrosio2006stability}. However, to obtain quantatitive error rates, \cite{junge2017fully} had to assume uniform $C^5$ bounds on the iterates.

\begin{comment}
The Lagrangian $L^2$-gradient flow formulation of \eqref{gradient_flow_eqn12} has been primarily studied in the setting where $\phi$ is discontinuous. This perspective was first introduced in \cite{evans2005diffeomorphisms,ambrosio2006stability}. In particular, \cite{ambrosio2006stability} proved convergence of the implicit Euler scheme under strong assumptions on the $\mathbb{R}^d$ ODE \eqref{characteristic_eqn} for a class of discontinuous $\phi$. To the authors’ knowledge, the only reference that obtains convergence rates in dimension $d > 1$ for these Lagrangian $L^2$-gradient flow schemes is \cite{junge2017fully}. For a class of discontinuous functionals $\phi$, they recursively defined the Lagrangian flow using the implicit Euler scheme for \eqref{L2_Lagrangian_Flow}:
\begin{equation}
X^{\tau}_{n+1} = X^{\tau}_n - \tau \,\nabla \phi^{\#}_{\rho_0}\bigl(X^{\tau}_{n+1}\bigr), \label{implicit_Euler_L21}
\end{equation}
starting from initial data $X^{\tau}_0$. Defining the associated measures as
\begin{equation}
\rho^{\tau}_{n+1} := \bigl(X^{\tau}_{n+1}\bigr)_{\#}\,\rho_0, \label{numerical_measures_L21}
\end{equation}
they showed that if the $C^5$ norm of $\rho^{\tau}_{n+1}$ is uniformly bounded in $\tau$ and $n$, then
\[
W_2(\rho^{\tau}_n, \rho_{\tau n}) = O(\tau),
\]
where $\rho_{\tau n}$ is the solution to \eqref{gradient_flow_eqn12} evaluated at time $\tau n$. For further references, see the survey article \cite{carrillo2021lagrangian} and the works cited therein.
\end{comment}

Gradient flows on Hilbert spaces, such as \eqref{L2_Lagrangian_Flow}, are well understood when $\phi^{\#}_{\rho_0}$ is $\lambda$-convex. The general theory of gradient flows for $\lambda$-convex functionals in Hilbert spaces (see \cite[Chapters 11 and 12]{ambrosio2021lectures}) ensures existence and uniqueness of strong Hilbert-Space-valued  solutions to \eqref{L2_Lagrangian_Flow}. In this setting, solutions can be constructed via the implicit Euler scheme,
\begin{equation}
\begin{cases}
X^{\tau}_{n+1} = X^{\tau}_n - \tau \,\nabla \phi^{\#}_{\rho_0}\bigl(X^{\tau}_{n+1}\bigr),  \\
\rho^{\tau}_{n+1} := \bigl(X^{\tau}_{n+1}\bigr)_{\#}\,\rho_0, \label{numerical_measures_L21}
\end{cases}
\end{equation}
which achieves the optimal convergence rate of $O(\tau)$, where $\tau$ is the time step size, when the initial data lies in the subdifferential of $\phi^{\#}_{\rho_0}$ \cite{ambrosio2005gradient}.
The convexity of $\phi^{\#}_{\rho_0}$ over $L^2(\R^d;\rho_0)$ is closely related to the (displacement) convexity of $\phi$ over $\mathbf{W}_2$. In particular, for the class of \textit{continuous} energy functionals $\phi$, \cite[Lemma 3.6]{gangbo2022global} shows that $\lambda$-displacement convexity of $\phi$ in $\mathbf{W}_2$ is equivalent to $\lambda$-convexity of its lift $\phi^{\#}_{\rho_0}$ in $L^2(\R^d;\rho_0)$. When one further assumes that $\phi$ is differentiable in the sense of \cite{gangbo2019differentiability}, 
this $\lambda$-convexity of the lifted functional $\phi^{\#}_{\rho_0}$ is also equivalent to 
$\phi$ being $\lambda$-convex along generalized geodesics in $\mathbf{W}_2(\R^d)$, as shown in \cite[Theorem 1.1]{parker2024some}. In more general settings, where $\phi$ may be discontinuous, one can still obtain poly-convexity of $\phi^{\#}_{\rho_0}$ from the displacement convexity of $\phi$ \cite{carrillo2021lagrangian}.

By the theory described above, one can construct solutions to \eqref{L2_Lagrangian_Flow} for continuous and $\lambda$-displacement convex energy functionals using the implicit Euler scheme. We now examine its connection to the JKO scheme. In \cite{ambrosio2006stability}, under suitable conditions, it was shown that the iterates defined by \eqref{numerical_measures_L21} coincide with those of the JKO scheme. Moreover, under some assumptions, the update map
\[
T^{\tau}_{n+1} := X^{\tau}_{n+1} \circ \bigl(X^{\tau}_n\bigr)^{-1}
\]
is the optimal transport map from $\rho^{\tau}_n$ to $\rho^{\tau}_{n+1}$. In particular, composing both sides of the Lagrangian coordinates of \eqref{numerical_measures_L21} with $\bigl(X^{\tau}_n\bigr)^{-1}$ yields
\[
T^{\tau}_{n+1} = \textnormal{Id} - \tau \nabla \phi^{\#}_{\rho^{\tau}_n}(T^{\tau}_{n+1}),
\]
showing that the optimal transport map corresponds to a single implicit Euler step for the $L^2$-gradient flow
\[
\frac{d}{dt} \tilde{X}(t,x) = -\nabla \phi^{\#}_{\rho^{\tau}_n}(\tilde{X}(t,x)),
\]
with initial condition $\tilde{X}(0,\cdot) = \textnormal{Id}$. 

We now focus on higher-order schemes, i.e., schemes that formally converge faster than $O(\tau)$. A key difficulty in extending higher-order schemes to dimensions $d > 1$ is that in Eulerian coordinates it can be difficult to correctly approximate the Wasserstein space by a Hilbert space to greater than first order accuracy, while in Lagrangian coordinates, the composition of flow map iterates $X_{n+1}^{\tau}\circ (X_{n}^{\tau})^{-1}$ is in general not an optimal transport map between $\rho_n^{\tau}$ and $\rho_{n+1}^{\tau}$, creating new difficulties. Note that the failure of $X_{n+1}^{\tau}\circ (X_{n}^{\tau})^{-1}$ to be an optimal transport map is due to Brenier’s theorem, which states that optimal transport maps must be gradients of convex functions, a structure not naturally preserved by higher-order Lagrangian schemes in dimensions $d>1$.
Nonetheless, many higher-order stable schemes for Wasserstein gradient flows have been proposed for \eqref{gradient_flow_eqn12} (see, e.g., \cite{legendre2017second, matthes2019variational, ashworth2020structure,  zaitzeff_hos, carrillo2022primal, han_hos, gallouet2024geodesic, cances2024discretizing}). However, as we noted earlier, these approaches have so far only established convergence without quantitative rates, or at best, have achieved a suboptimal $O(\sqrt{\tau})$ convergence rate by extending the methods in \cite{ambrosio2005gradient}.

\smallskip

Our scheme \eqref{trapezoid_variational_form} appears to be the first to identify smoothness conditions on the energy functional that guarantee an $O(\tau^2)$ convergence rate for both the Lagrangian coordinates and the probability measure iterates. Moreover, to our knowledge, it is the first to leverage the differential structure of $\mathbf{W}_2$ to derive error estimates for Wasserstein gradient flows. We expect this calculus-based framework to extend naturally to higher-order schemes including explicit schemes. Although we only verify the case $p=2$ below, we expect that an $O(\tau^p)$ accurate (explicit or implicit) scheme for \eqref{gradient_flow_eqn12} can be obtained by applying an $O(\tau^p)$ accurate finite difference scheme to \eqref{L2_Lagrangian_Flow}. In addition, even for less regular functionals, where such calculus based techniques do not apply, we  prove our scheme \eqref{trapezoid_variational_form} still achieves the $O(\tau)$ convergence rate of the JKO scheme.

\smallskip

% By leveraging the Hilbert space structure of $L^2(\mathbb{R}^d;\rho_0)$, we extend the classical finite-dimensional proof of the $O(\tau^2)$ convergence of the trapezoid scheme (originally for ODEs in $\mathbb{R}^d$) to the infinite-dimensional setting. In particular, we show that when $\phi$ is sufficiently smooth, $X^{\tau}_n$ converges to the strong solution $X(\tau n,\cdot)$ of \eqref{L2_Lagrangian_Flow} at a rate of $O(\tau^2)$ in the $L^2(\mathbb{R}^d;\rho_0)$ norm. \smallskip

% By defining
% \[
% \rho_t := (X(t,\cdot))_{\#} \rho_0
% \quad\text{and}\quad
% \rho^{\tau}_n := (X^{\tau}_n)_{\#} \rho_0,
% \]
% the inequality
% \[
% W_2(\rho_{\tau n}, \rho^{\tau}_n)
% \leq  \|X(\tau n,\cdot)-X_{n}^{\tau}\|_{L^2(\mathbb{R}^d;\rho_0)}=O(\tau^2),
% \]
% then shows that these measures also converge with a rate $O(\tau^2)$ in the Wasserstein-2 metric. Then we characterized $\rho_t$ as the weak solution \eqref{gradient_flow_eqn1} with initial data $\rho_0$. \smallskip 

\subsection{Main results}

Our first result establishes an $O(\tau^2)$ error rate for the scheme \eqref{trapezoid_variational_form} under suitable smoothness conditions on the lifted energy functional. Our strategy is to extend the classical finite-dimensional proof of the $O(\tau^2)$ convergence of the trapezoid scheme for ODEs on $\mathbb{R}^d$ to the infinite-dimensional Hilbert space $L^2(\R^d;\rho_0)$. Consequently, our assumptions on convergence mirror those in the finite-dimensional setting.

\begin{assumption}[Smoothness Assumption] \label{smoothness_assumption}
Let $\rho_0 \in \mathcal{P}_2(\mathbb{R}^d)$ be given, and define the Hilbert space $\mathbb{H} := L^2(\mathbb{R}^d; \rho_0)$.  
We assume that the lifted energy functional
\[
    X \mapsto \phi^{\#}_{\rho_0}(X)
\]
belongs to $C^{1,1}(\mathbb{H}; \mathbb{H})$.  
In addition, we assume that the unique strong solution $X : [0, \infty) \to \mathbb{H}$ of \eqref{L2_Lagrangian_Flow} with $X_0 = \mathrm{Id}$ satisfies
\[
    X \in C^{2,1}_{\mathrm{loc}}([0, \infty); \mathbb{H}).
\] 
\end{assumption}

As our numerical Lagrangian flow, \( X^{\tau}_{n+1} \) arises from a finite difference scheme, we are able to extend the classical finite-dimensional analysis of such schemes to our infinite-dimensional setting, thereby obtaining:

\begin{theorem}[$O(\tau^2)$ Convergence: Theorem \ref{higher_order_convergence_theorem}, Theorem \ref{strong_gf_solution}, Theorem \ref{limiting_PDE}] \label{higher_order_convg_intro}
Fix a terminal time $T>0$, $\rho_0 \in \mathcal{P}_2(\mathbb{R}^d)$, and assume that the smoothness Assumption \ref{smoothness_assumption} holds. 
Define the Lipschitz constants
\[
L(\phi) := \sup_{\xi_1 \neq \xi_2 \in \mathbb{H}} \frac{ \| \nabla \phi^{\#}_{\rho_0}(\xi_1) - \nabla \phi^{\#}_{\rho_0}(\xi_2) \|_{\mathbb{H}}}{\|\xi_1 - \xi_2\|_{\mathbb{H}}}, \quad
L(T, \ddot{X}) := \sup_{t \neq s \in [0,T]} \frac{ \| \dot{X}_t - \dot{X}_s - (t - s)\ddot{X}_s \|_{\mathbb{H}} }{|t - s|}.
\]
Then, for any time step $\tau \leq 1/L(\phi)$ and integer $n \in \mathbb{N}$ such that $n\tau \leq T$, one has
\[
W_2(\rho_{n\tau }, \rho_n^{\tau}) \leq \| X(n\tau) - X_n^{\tau} \|_{\mathbb{H}} \leq e^{2 L(\phi) T} \| X_0^{\tau} - \textnormal{Id} \|_{\mathbb{H}} + 2 \frac{L(T, \ddot{X})}{L(\phi)} (e^{2 L(\phi) T} - 1) \cdot \tau^2,
\] where $\rho^{\tau}_n$ and $X^{\tau}_n$ are defined by \eqref{trapezoid_variational_form}. Here, $\rho_t := (X(t))_{\#} \rho_0$ is a weak solution of 
\begin{equation}
\begin{cases}
    \p_t \rho - \nabla \cdot(\rho \nabla_{\mathbf{W}} \phi(\rho,x)) = 0  \text{ on } (0,\infty) \times \R^d  \\
    \rho(0,\cdot)=\rho_0(\cdot) \text{ on } \R^d
\end{cases} \label{gradient_flow_eqn1}.
\end{equation}
If instead one has $X \in C^{2,\alpha}_{\textnormal{loc}}([0,\infty);\mathbb{H})$ the convergence rates becomes $O(\tau^{1+\alpha})$.
\end{theorem}

We now provide an example of an energy functional satisfying our assumptions. Consider the functional (see Example~\ref{C2_functionals}, Example~\ref{C11_Lifted_Functional}~, and Example~\ref{C2alpha_velocity})
\begin{equation}
U(\rho)
:= \int_{\mathbb{R}^d} f\bigl((\rho \star \chi)(x)\bigr)\,d\rho (x)
+ \int_{\mathbb{R}^d} V(x)\,d\rho(x)
+ \frac{1}{2} \int_{\mathbb{R}^d} \int_{\mathbb{R}^d} W(x-y)\,d\rho(x)\,d\rho(y),  \label{model_energy_functional_123}
\end{equation}
where $V,W \in C^{2,1}_b(\mathbb{R}^d)$, $f \in C^{2,1}_b(\mathbb{R})$, and $\chi \in C^{2,1}_c(\mathbb{R}^d)$, with $W$ and $\eta$ being even functions. Here, $\star$ denotes convolution. This functional satisfies the assumptions of Theorem \ref{higher_order_convg_intro} when $\rho_0$ has no atoms or when $\rho = \frac{1}{N} \sum_{j=1}^{N} \delta_{X_j}$. The latter case is useful for numerical simulations as the energy and its gradient will be fully discrete.  Numerical simulations of this scheme for this energy functional can be found in Section \ref{sec:experiments}.  These discretizations have proven very useful for blob-type methods for simulating Wasserstein gradient flows c.f. \cite{Carrillo2019, Daneri2022-fr, doi:10.1142/S0219891623500212, BURGER2023113347, Craig2023-hj, CRAIG2025782, Carrillo2024-we}.

Even under weaker regularity assumptions on $\phi$, namely, that $\phi$ is Wasserstein differentiable and $\lambda$-displacement convex, we obtain an intermediate $O(\tau)$ convergence rate by adapting the \textit{Discrete Evolution Variational Inequality} (EVI) method from \cite{ambrosio2005gradient}. For this convergence, we will need the following weaker assumptions:

\begin{assumption} \label{assume_rough}
    We fix a measure $\rho_0 \in \mathcal{P}_2(\mathbb{R}^d)$ and let $\mathbb{H} := L^2(\mathbb{R}^d;\rho_0)$.  \smallskip 

\noindent Given a functional $\phi : \mathcal{P}_2(\mathbb{R}^d) \to \mathbb{R}$, we assume that its lifted functional $\phi_{\rho_0}^{\#} : \mathbb{H} \to \mathbb{R}$, defined by \eqref{lift_intro_def}
is Fréchet differentiable over $\mathbb{H}$ and is $\lambda$-convex over $\mathbb{H}$ for some $\lambda \in \mathbb{R}$ Additionally, we assume that the lifted energy functional is proper, i.e.,
\[
\inf_{\xi \in L^2(\mathbb{R}^d;\rho_0)} \phi^{\#}_{\rho_0}(\xi) > -\infty.
\] 

\noindent We further assume that in \eqref{trapezoid_variational_form} the time step parameter satisfies
\[
    \lambda/2 + 1/\tau > 0.
\]
Finally, we suppose that the initial iterates satisfy
\[
    \lim_{\tau \downarrow 0} \|X_0^{\tau} - \mathrm{Id}\|_{\mathbb{H}} = 0,
    \qquad
    \sup_{\tau > 0} \|\nabla \phi^{\#}_{\rho_0}(X_0^{\tau})\|_{\mathbb{H}} < \infty.
\]
\end{assumption}

\begin{remark}
    We expect that the Fréchet differentiablity assumption can be weakened to just lower semi-continuity of the lift, but we impose this condition to simplify some of our arguments.
\end{remark}

\begin{remark}
    Note that the condition of the lift being $\lambda$-convex is rather strong and different from a functional being $\lambda$-displacement convex in the Wasserstein sense.  In particular, the $\lambda$-convexity of the lift does not hold for the internal energy functional $\rho\mapsto \int f(\rho)$ even when $f$ is convex.  Nonetheless, if the energy is continuous with respect to the Wasserstein metric, then $\lambda$-convexity of the lifted energy is equivalent to $\lambda$-convexity in the Wasserstein sense (see the next remark).
\end{remark}

\begin{remark}[Convexity and Differentiability Assumptions] 
\label{sufficient_conditions_rough_assumptions}
Let $\rho_0 \in \mathcal{P}_{2}(\mathbb{R}^d)$ be atomless and $\phi : \mathcal{P}_2(\mathbb{R}^d) \to \mathbb{R}$ be continuous. Then $\lambda$-displacement convexity (see Definition~\ref{geodesic_convex}) is equivalent to the $\lambda$-convexity of the lifted functional $\phi^{\#}_{\rho_0}$ (see Theorem~\ref{convex_equivalence} and \cite{gangbo2022global}).
\smallskip

\noindent
Furthermore, if $\phi \in C^1(\mathbf{W}_2(\mathbb{R}^d); \mathbb{R})$ (see Section \ref{prelim} for a precise definition) satisfies suitable gradient growth conditions (see Proposition~\ref{differentiability_equivalence}), then the lifted functional $\phi^{\#}_{\rho_0}$ is Fréchet differentiable. In this case, the gradient of the lifted functional is given by
\[
\nabla \phi^{\#}_{\rho_0}(\xi)(x) = \nabla_{\mathbf{W}} \phi\big(\xi_{\#} \rho_0,\, \xi(x)\big).
\]
These results can be useful for verifying the convexity and smoothness assumptions we require.
\end{remark}

We will also frequently use the quantity
\begin{equation} \label{lambda_tau_def}
    \lambda_{\tau} := \frac{1}{2\tau} \log\!\left( \frac{1 + \lambda \tau}{1 - \lambda \tau} \right),
\end{equation}
which satisfies $\lambda_{\tau} \to \lambda$ as $\tau \downarrow 0$.

\begin{theorem}[$O(\tau)$ Convergence: Theorem \ref{linear_convergence_rate}, Theorem \ref{refined_convergence}, Theorem \ref{strong_gf_solution}, Theorem \ref{limiting_PDE}] \label{lower_order_intro}
Under the notation and assumptions of Assumption \ref{assume_rough} one has
\begin{equation}
W_2(\rho^{\tau}_t,\rho_t) \leq \|X(t) - \underline{X}_t^{\tau}\|_{\mathbb{H}} \leq
\begin{cases}
 \sqrt{3} \|X_0^{\tau} - \textnormal{Id}\|_{\mathbb{H}} 
    + \frac{\sqrt{33}}{2}\tau \| \nabla \phi^{\#}_{\rho_0}(X^{\tau}_0)\|_{\mathbb{H}} 
    , & \text{if } \lambda \geq 0,\\[1em]
\sqrt{3} e^{|\lambda_{\tau}| t} \|X_0^{\tau} - \textnormal{Id}\|_{\mathbb{H}}  
+ \sqrt{3} C(\lambda_{\tau}, t, \tau) \cdot \tau e^{|\lambda_{\tau}| t} \|\nabla \phi^{\#}_{\rho_0}(X_0^{\tau})\|_{\mathbb{H}}^2, & \text{if } \lambda < 0.
\end{cases} \label{rough_convg_intro}
\end{equation}
Here, the lower-order term $C(\lambda_{\tau}, t, \tau)$ is defined in \eqref{C_definition}, $\lambda_{\tau}$ is from \eqref{lambda_tau_def}, and $\underline{X}_t^{\tau} := X_n^{\tau}$ for $t \in [n\tau, (n+1)\tau)$ and $\rho_t^{\tau} := (\underline{X}_t^{\tau})_{\#}\rho_0$ and $\rho_t := (X(t))_{\#} \rho_0$, where $X(t)$ is the unique strong solution of \eqref{L2_Lagrangian_Flow}.
\smallskip

Furthermore, if $\nabla \phi^{\#}_{\rho_0}$ is $L$-smooth on $\mathbb{H}$, i.e.,
\[
\sup_{\xi_1 \neq \xi_2 \in \mathbb{H}} \frac{ \| \nabla \phi^{\#}_{\rho_0}(\xi_1) - \nabla \phi^{\#}_{\rho_0}(\xi_2) \|_{\mathbb{H}}}{\|\xi_1 - \xi_2\|_{\mathbb{H}}} \leq L,
\]
and if $\tau \le 1/L$, then the error decays exponentially in time whenever $\lambda > 0$:
\begin{equation}
W_2(\rho^{\tau}_t,\rho_t) \leq \|X(t) - \underline{X}_t^{\tau}\|_{\mathbb{H}} \leq \sqrt{3} e^{-|\lambda_{\tau,L}| t} \left( \|\textnormal{Id} - X_0^{\tau}\|_{\mathbb{H}} + \tilde{C}(\lambda, t, \tau) \cdot \tau \|\nabla \phi^{\#}_{\rho_0}(X_0^{\tau})\|_{\mathbb{H}} \right), \label{refined_decay_rates123}
\end{equation}
where $\tilde{C}(\lambda, t, \tau)$ is defined in Theorem~\ref{refined_convergence} and $\lambda_{\tau,L}$ is given in Lemma~\ref{refined_decay}. \smallskip
\end{theorem}

Observe that both $\lambda_{\tau}$ and $\lambda_{\tau,L}$ converge to the optimal rate $\lambda$ as $\tau \to 0$.
\smallskip

Our approach to Theorem \ref{lower_order_intro} is based on the EVI characterization of $\lambda$-convex gradient flows. For example, in Hilbert spaces \cite{ambrosio2021lectures}, a curve $x(t)$ solves the gradient flow of a $\lambda$-convex functional $f$ if and only if it satisfies the evolution variational inequality
\[
\frac{d}{dt} \left( \frac{1}{2} \|x(t) - y\|^2 \right) + \frac{\lambda}{2} \|x(t) - y\|^2 + f(x(t)) \leq f(y) \quad \text{for a.e. } t \in (0,\infty),
\]
for all $y$ in the domain of $f$. The key idea is to show that suitable interpolations of the numerical Lagrangian flow $X^{\tau}_n$ satisfy an approximate EVI with an error term of order $O(\tau^2)$. Then the $O(\tau)$ convergence rate is obtained by using a version of Grönwall's inequality from \cite{ambrosio2005gradient}.

\smallskip

The approximate EVI we derived for our trapezoid-rule method is in Theorem \ref{differential_inequality_thm} and follows from our Discrete EVI derived in Lemma \ref{discrete_EVI}. Deriving this discrete EVI with the correct $O(\tau^2)$ error term was a key technical challenge in our analysis. Compared to the implicit Euler and JKO schemes in \cite{ambrosio2005gradient}, our scheme's variational problem \eqref{trapezoid_variational_form} included the term
\[ \int_{\R^d} \langle \nabla \phi^{\#}_{\rho^{\tau}_0}(X^{\tau}_n),\xi \rangle d\rho_0(x). \] To control this term we had to carefully use both the convexity properties of $\phi^{\#}_{\rho_0}$ and the trapezoid scheme \eqref{trapezoid_rule123} (see Lemma \ref{EVI_Helper}). \smallskip

Another difficulty in the trapezoidal rule setting, compared to the implicit Euler scheme, is the lower bounds of $\|X^{\tau}_{n+1} - X^{\tau}_n\|^2_{\mathbb{H}}$ in terms of the gradient. In the implicit Euler case, this quantity simplifies to $\tau^2 \|\nabla \phi^{\#}_{\rho_0}(X^{\tau}_{n+1})\|^2_{\mathbb{H}}$. However, in the trapezoidal rule case, it becomes
\[
\frac{\tau^2}{4} \left\| \nabla \phi^{\#}_{\rho_0}(X^{\tau}_{n+1}) + \nabla \phi^{\#}_{\rho_0}(X^{\tau}_n) \right\|^2_{\mathbb{H}},
\]
which prevents us from establishing a lower bound of the form $\|X^{\tau}_{n+1} - X^{\tau}_n\|^2_{\mathbb{H}} \geq C \tau^2 \|\nabla \phi^{\#}_{\rho_0}(X^{\tau}_{n+1})\|^2_{\mathbb{H}}$ without assuming that $\nabla \phi^{\#}_{\rho_0}$ is sufficiently smooth. As a result, we were unable to recover exponentially decaying constants in the error estimate \eqref{rough_convg_intro} when $\lambda > 0$ without a smoothness assumption.

\smallskip

Beyond convergence rates, our scheme exhibits the following numerical stability properties:

\begin{theorem}[Numerical Stability: Lemma \ref{trapezoid_well_defined}, Lemma \ref{energy_almost_decreasing}, Lemma \ref{refined_decay}]
Under the notation and assumptions of Assumption \ref{assume_rough}. Then the energy is \emph{almost decreasing}, in the sense that
\[
\phi^{\#}_{\rho_0}(X^{\tau}_{n+1}) - \phi^{\#}_{\rho_0}(X^{\tau}_n)
\leq \frac{\tau}{4} \left( \|\nabla \phi^{\#}_{\rho_0}(X^{\tau}_n)\|^2_{\mathbb{H}} - \|\nabla \phi^{\#}_{\rho_0}(X^{\tau}_{n+1})\|^2_{\mathbb{H}} \right).
\]
Moreover, the gradient norm satisfies:
\[
\|\nabla \phi^{\#}_{\rho_0}(X^{\tau}_{n+1})\|^2_{\mathbb{H}} \leq \mathcal{C}(\lambda) \cdot \|\nabla \phi^{\#}_{\rho_0}(X^{\tau}_n)\|^2_{\mathbb{H}},
\]
where
\[
\mathcal{C}(\lambda) :=
\begin{cases}
1, & \text{if } \lambda \geq 0, \\
e^{2 |\lambda_{\tau}| \tau}, & \text{if } \lambda < 0,
\end{cases}
\]
and $\lambda_{\tau}$ is defined in \eqref{lambda_tau_def}. In particular, when $\lambda \geq 0$, the gradient norm is non-increasing across iterations.

\smallskip

Furthermore, if $\lambda \geq 0$ and $\nabla \phi^{\#}_{\rho_0}$ is $L$-\textnormal{Lipschitz}, then for $\tau \leq 1/L$, the energy is \emph{non-increasing}:
\[
\phi^{\#}_{\rho_0}(X^{\tau}_{n+1}) + \frac{\lambda}{2} ||X^{\tau}_{n+1}-X^{\tau}_n||_{\mathbb{H}} + \tau \left( 1 - \frac{L \tau}{2} \right) \|\nabla \phi^{\#}_{\rho_0}(X^{\tau}_{n+1})\|^2_{\mathbb{H}} \leq \phi^{\#}_{\rho_0}(X^{\tau}_n),
\]
and the \emph{gradient norm decays exponentially}:
\[
\|\nabla \phi^{\#}_{\rho_0}(X^{\tau}_{n+1})\|^2_{\mathbb{H}} \leq e^{-2 \lambda_{\tau,L} \tau} \|\nabla \phi^{\#}_{\rho_0}(X^{\tau}_n)\|^2_{\mathbb{H}},
\]
where $\lambda_{\tau,L}$ is defined in Lemma~\ref{refined_decay}.
\end{theorem}

We also have obtained a classical stability result of our scheme in Lemma \ref{classical_stability_lemma} and refined it in the $L$-Lipschitz setting in Lemma \ref{refined_stability}. \smallskip

The rest of the paper is structured as follows. In Section \ref{prelim}, we review the differential structure of $\mathbf{W}_2(\R^d)$ and the differential and convexity properties of the lifted energy. In Section \ref{stability_section}, we introduce our second-order scheme and derive its stability properties. Section \ref{higher_order_convg} uses calculus-based arguments to show that the scheme converges at a second-order rate when the energy functional is sufficiently smooth. In Section \ref{Discrete_EVI_Section}, we prove that for less regular energy functionals, the scheme still converges at least linearly using the discrete EVI method introduced in \cite{ambrosio2005gradient}. Then, in Section \ref{further_stability}, we show that when the energy is $\lambda$-displacement convex with $\lambda > 0$ and $L$-smooth, the scheme achieves asymptotically sharp exponential decay rates in time for the gradient norm and the error derived in Section \ref{Discrete_EVI_Section}. Finally in Section \ref{sec:experiments} we present numerical experiments of our scheme \eqref{trapezoid_variational_form} using the energy functional \eqref{model_energy_functional_123}.

\smallskip

\noindent \textbf{Acknowledgements.} R.C. was partially supported by NSF grant DMS-2153254, DMS-2342349, and the Dissertation Year Fellowship from the University of California, Los Angeles during this research project.  M.J. is partially supported supported by NSF grant DMS-2400641.  R.C. would also like to thank Professors Inwon Kim and Wilfrid Gangbo for helpful discussions, and also Professor Dejan Slepčev for insightful conversations that led to Section \ref{sec:experiments}. The authors thank Alp\'{a}r M\'esz\'aros for helpful discussions concerning the Wasserstein Hessian. A preliminary version of some of these results appeared in the Ph.D. thesis of the first author \cite{chu2025optimal}.

\section{Preliminaries on the Differential Structure of \texorpdfstring{$\mathbf{W}_2$}{W2}}
 \label{prelim}

\subsection{Notation} 

Given any measure $\mu$, we define the Hilbert space
\[
L^2(\R^d;\mu) := \left \{ \xi : \R^d \rightarrow \R^d : \int_{\R^d}|\xi(x)|^2 d\mu(x) < \infty \right \},
\]
where the inner product is given by
\[
\langle \xi_1,\xi_2 \rangle_{L^2(\R^d;\mu)} := \int_{\R^d} \langle \xi_1, \xi_2 \rangle d\mu(x),
\]
where $\langle x,y \rangle$ denotes the dot product in $\R^d$, and we let $||\xi||_{L^2(\R^d;\mu)}$ represent the norm on $L^2(\R^d;\mu)$ induced by the $L^2(\R^d;\mu)$ inner product.

\smallskip

% Similarly, for any $T \in [0,+\infty]$, we define the space-time Hilbert space
% \[
% L^2( [0,T] \times \R^d;\mu) := \left \{ \vec{\xi} : [0,T] \times \R^d \rightarrow \R^d : \int_0^T \int_{\R^d} |\vec{\xi}(t,x)|^2 d\mu(x) dt < \infty  \right \},
% \]
% which we equip with the inner product
% \[
% \langle \vec{\xi}_1,\vec{\xi}_2 \rangle_{L^2([0,T] \times \R^d;\mu)} := \int_0^{T} \int_{\R^d} \langle \vec{\xi}_1, \vec{\xi}_2 \rangle d\mu(x) dt,
% \]
% and its associated norm as $||\vec{\xi}||_{L^2([0,T] \times \R^d;\mu)}$.

\smallskip

Let $\mathcal{P}_2(\R^d)$ be the space of probability measures on $\R^d$ with finite second moments.

\smallskip

We frequently consider the \textit{push-forward} of probability measures under (Borel) measurable maps. Given a measurable map $\xi : \R^d \rightarrow \R^d$ and a probability measure $\mu$, the push-forward of $\mu$ by $\xi$ is defined as
\[
\xi_{\#} \mu(A) := \mu(\xi^{-1}(A)) \quad \text{for all Borel sets } A \subset \R^d.
\]
\noindent If $\mu \in \mathcal{P}_2(\R^d)$ and $\xi \in L^2(\R^d;\mu)$, then
\[
\int_{\R^d} |x|^2 d(\vec{\xi}_{\#} \mu) = \int_{\R^d} |\xi(x)|^2 d\mu(x) < \infty.
\]
Thus, we conclude that $\xi_{\#} \mu \in \mathcal{P}_2(\R^d)$ whenever $\xi \in L^2(\R^d;\mu)$ and $\mu \in \mathcal{P}_2(\R^d)$.

\smallskip

We also frequently use the following composition rule for push-forwards: given any Borel measurable maps $\xi_1,\xi_2 : \R^d \rightarrow \R^d$,
\[
(\xi_1 \circ \xi_2)_{\#} \mu = (\xi_1)_{\#} \big((\xi_2)_{\#} \mu\big).
\]

\smallskip

The Wasserstein-2 metric space $\mathbf{W}_2(\R^d)$ is given by $(\mathcal{P}_2(\R^d),W_2)$, where the Wasserstein-2 metric is 
\begin{equation}
W_2(\mu,\nu) := \inf_{\pi \in \Pi[\mu,\nu] } \iint_{\R^d \times \R^d} |x-y|^2 d\pi(x,y),
\label{W2_Def}
\end{equation}
where $\Pi[\mu,\nu]$ denotes the set of probability measures on $\R^d \times \R^d$ with left marginal $\mu$ and right marginal $\nu$ (see \cite{villani2021topics,santambrogio2015optimal}).

\smallskip

We now establish some useful and well-known inequalities for this metric.

\begin{lemma}  \label{lipschitz_control}
Let $\xi : \mathbb{R}^d \rightarrow \mathbb{R}^d$ be $L$-Lipschitz, meaning that
\[
|\xi(x) - \xi(y)| \leq L|x - y|, \quad \forall x, y \in \mathbb{R}^d.
\]
Then, for any $(\mu, \nu) \in (\mathcal{P}_{2}(\mathbb{R}^d))^2$, we have
\[
W_2(\xi_{\#} \mu, \xi_{\#} \nu) \leq L W_2(\mu, \nu).
\]
\end{lemma}

\begin{proof}
This follows from the observation that if $\pi \in \Pi[\mu, \nu]$, then $(\xi, \xi)_{\#} \pi \in \Pi[\xi_{\#} \mu, \xi_{\#} \nu]$.
\end{proof}

\begin{lemma} \label{same_measure_bound}
Fix $\mu \in \mathcal{P}_2(\mathbb{R}^d)$ and let $\xi_1, \xi_2 \in L^2(\mathbb{R}^d; \mu)$. Then,
\[
W_2^2((\xi_1)_{\#} \mu, (\xi_2)_{\#} \mu) 
\leq \int_{\mathbb{R}^d} |\xi_1(x) - \xi_2(x)|^2 \, d\mu(x) 
= \|\xi_1 - \xi_2\|^2_{L^2(\mathbb{R}^d; \mu)}.
\]
\end{lemma}

\begin{proof}
By assumption, both measures $\xi_1{}_{\#} \mu$ and $\xi_2{}_{\#} \mu$ belong to $\mathcal{P}_2(\mathbb{R}^d)$. The result then follows by observing that the coupling $(\xi_1, \xi_2)_{\#} \mu$ belongs to $\Pi[\xi_1{}_{\#} \mu, \xi_2{}_{\#} \mu]$, and applying the definition of the Wasserstein-2 distance via \eqref{W2_Def}.
\end{proof}

\subsection{Derivatives on Wasserstein Space} \label{calculus_Wasserstein}

We present two notions of derivatives for $\phi$, following the definitions in \cite[Section 2]{cardaliaguet2019master}. The concept of the Wasserstein derivative used here coincides with their notion of the intrinsic derivative.

\smallskip

\begin{definition}[First Variation] 
Let $\phi : \mathcal{P}_2(\mathbb{R}^d) \to \mathbb{R}$. We say that $\phi \in C^1(\mathcal{P}_2(\mathbb{R}^d); \mathbb{R})$ if there exists a function 
\[
\frac{\delta \phi}{\delta \mu} : \mathcal{P}_2(\mathbb{R}^d) \times \mathbb{R}^d \to \mathbb{R}
\]
satisfying the following conditions:
\begin{enumerate}
    \item  The function $\frac{\delta \phi}{\delta \mu}$ is jointly continuous on $\mathbf{W}_2(\R^d) \times \mathbb{R}^d$
    \item For any $\nu \in \mathcal{P}_2(\mathbb{R}^d)$, the function $x \mapsto \frac{\delta \phi}{\delta \mu}(\nu,x)$ has at most quadratic growth.
    \item  For any $\rho, \rho' \in \mathcal{P}_2(\mathbb{R}^d)$, we have
    \begin{equation} \label{eq:first_variation}
        \lim_{\varepsilon \to 0} \frac{\phi((1-\varepsilon) \rho + \varepsilon \rho') - \phi(\rho)}{\varepsilon} = \int_{\mathbb{R}^d} \frac{\delta \phi}{\delta \mu}(\rho,x) \, d(\rho' - \rho)(x).
    \end{equation}
\end{enumerate}

\noindent The function $\frac{\delta \phi}{\delta \mu}$ is called the \textit{first variation} of $\phi$.
\end{definition}

\noindent Since the first variation is defined only up to an additive constant, we impose the normalization
\[
\int_{\R^d} \frac{\delta \phi}{\delta \mu}(\rho,y) d\rho(y) = 0, \quad \forall \rho \in \mathcal{P}_2(\R^d).
\]
\noindent The first variation also satisfies the fundamental theorem of calculus property (see \cite{cardaliaguet2019master}):
\begin{equation}
    \phi(\rho') - \phi(\rho) = \int_0^1 \int_{\R^d} \frac{\delta \phi}{\delta \mu}((1-s)\rho + s \rho',x) d(\rho'-\rho)(x)ds.
    \label{first_variation}
\end{equation}

\smallskip

Next, we define the Wasserstein gradient, which extends the classical gradient notion to the metric space $\mathbf{W}_2(\R^d)$.

\smallskip

    \begin{definition}[Wasserstein gradient]  \label{Wasserstein_Grad_Def}
We say that $\phi \in C^1(\mathbf{W}_2(\mathbb{R}^d); \mathbb{R})$ if:
\begin{enumerate}
    \item $\phi \in C^1(\mathcal{P}_2(\mathbb{R}^d); \mathbb{R})$.
    \item The first variation $\frac{\delta \phi}{\delta \mu}(\rho, x)$ is differentiable in $x$ for all $\rho \in \mathcal{P}_2(\mathbb{R}^d)$.
    \item The gradient $\nabla_x \frac{\delta \phi}{\delta \mu}(\rho, x)$ is jointly continuous on $\mathbf{W}_2(\mathbb{R}^d) \times \mathbb{R}^d$.
    \item For any $\rho \in \mathcal{P}_2(\R^d)$, the mapping $y \mapsto \nabla_x \frac{\delta \phi}{\delta \mu}(\rho, y)$ has at most quadratic growth.
\end{enumerate}

\noindent The \textit{Wasserstein gradient} of $\phi$ is defined as
\[
\nabla_{\mathbf{W}} \phi(\rho, x) := \nabla_x \frac{\delta \phi}{\delta \mu}(\rho, x).
\]
\end{definition}

Bounds on the Wasserstein gradient provide Lipschitz control for $\phi$. Indeed, in \cite[Section 2]{cardaliaguet2019master}, the dual formulation of the $W_1$ metric was used to establish the following Lipschitz property.

\smallskip

\begin{lemma}[\cite{cardaliaguet2019master}] \label{Lipschitz}
Let $\phi \in C^1(\mathbf{W}_2(\R^d);\R^d)$ and assume that
\[
L := \sup_{(\rho,x) \in \mathcal{P}_2(\R^d) \times \R^d} |\nabla_{\mathbf{W}} \phi(\rho,x)| < \infty.
\]
\noindent Then, for any $\mu, \nu \in \mathcal{P}_2(\R^d)$ and $p \geq 1$, we have
\begin{equation}
    |\phi(\mu) - \phi(\nu)| \leq L W_p(\mu,\nu).
    \label{lipschitz_w1}
\end{equation}
\end{lemma}

The definition of Wasserstein differentiability at a point $\mu \in \mathcal{P}_2(\mathbb{R}^d)$ can be found in \cite[Definition 3.11]{gangbo2019differentiability}. When $\phi \in C^1(\mathbf{W}_2(\mathbb{R}^d); \mathbb{R})$ and its Wasserstein gradient $\nabla_{\mathbf{W}} \phi$ satisfies certain mild growth conditions, then $\phi$ is Wasserstein differentiable on $\mathcal{P}_2(\mathbb{R}^d)$ (see Proposition~\ref{differentiability_equivalence} and \cite[Corollary 3.22]{gangbo2019differentiability}). The growth condition bounds on $\nabla_{\mathbf{W}} \phi$ are used to control the associated error terms arising in the differentiation. \smallskip

Next, we introduce the notion of partial Hessians in Wasserstein space that correspond to taking the Wasserstein gradient twice:

\smallskip

\begin{definition} [Partial Wasserstein Hessian] 
We say that $\phi \in C^2(\mathbf{W}_2; \mathbb{R})$ if $\phi \in C^1(\mathbf{W}_2; \mathbb{R})$ and, for all $x \in \mathbb{R}^d$, the mapping
\[
\rho \mapsto \nabla_{\mathbf{W}} \phi(\rho, x)
\]
belongs to $C^1(\mathbf{W}_2; \mathbb{R}^d)$, componentwise. Moreover, we assume that the partial Wasserstein Hessian
\[
\nabla^2_{\mathbf{W}} \phi : \mathcal{P}_2(\mathbb{R}^d) \times \mathbb{R}^d \times \mathbb{R}^d \to \mathbb{R}^{d \times d},
\]
defined as the Wasserstein gradient of $\nabla_{\mathbf{W}} \phi(\rho, x)$ with respect to $\rho$, is jointly continuous on $\mathcal{P}_2(\mathbb{R}^d) \times \mathbb{R}^d \times \mathbb{R}^d$. \smallskip

In particular, one has that
\[ \nabla^2_{\mathbf{W}} \phi(\rho,x,x') = \left[ \nabla_{x'} \frac{\delta}{\delta \mu} \nabla_x \frac{\delta}{\delta \mu} \left( \phi \right) \right](\rho,x,x')  \]

\end{definition}

\smallskip

By Lemma \ref{Lipschitz}, if $\phi \in C^2(\mathbf{W}_2(\R^d);\R)$ and $\nabla^2_{\mathbf{W}} \phi$ is uniformly bounded in $\mathcal{P}_2(\R^d) \times \R^d \times \R^d$, then the Wasserstein gradient satisfies the Lipschitz property
\begin{equation}
\sup_{\mu,\nu \in \mathcal{P}_2(\R^d)} \frac{|\nabla_{\mathbf{W}} \phi(\mu,x)-\nabla_{\mathbf{W}} \phi(\nu,x)|}{W_p(\mu,\nu)} \leq C, \label{gradient_Lipschitz_Wp}
\end{equation}
\noindent for all $p \geq 1$ where $C > 0$ is independent of $x$ and $p$. We will see later that the partial Wasserstein Hessian and spatial gradient of the Wasserstein gradient will allow us to control the gradient of the lifted energy functional. See in particular Example \ref{C11_Lifted_Functional}.\smallskip

\begin{remark}
A positive semidefinite \emph{partial} Wasserstein Hessian does \emph{not} guarantee displacement 
convexity of an energy functional. The partial Hessian is obtained by applying the Wasserstein 
gradient twice, and therefore does not coincide with the (full) Wasserstein Hessian considered in 
\cite[Definition~3.1]{chow2019partial}. By contrast, a positive semidefinite full Wasserstein Hessian 
does imply displacement convexity (see \cite[Lemma~4.1]{parker2024some}).
More concretely, the second time derivative of the energy along Wasserstein geodesics depends not 
only on the partial Hessian $\nabla^2_{\mathbf{W}}\phi$, but also on
$\nabla_x \nabla_{\mathbf{W}}\phi$, as shown in \eqref{second_deriv_phi}.
\end{remark}

Now we show an energy example functional that is of class $C^2(\mathbf{W}_2(\R^d);\R$). The functional is well suited for numerical approximation, since both the energy and its gradients can be computed in a fully discrete form when the probability measure is a finite sum of Dirac masses.

\begin{example} 
Fix $ f \in C_b^2(\mathbb{R}) $, $ V \in C_b^2(\mathbb{R}^d) $, $ W \in C_b^2(\mathbb{R}^d) $, and $ \chi \in C_c^2(\mathbb{R}^d) $. Further assume that $ W $ and $ \eta $ are even. Define $ U : \mathcal{P}_2(\mathbb{R}^d) \to \mathbb{R} $ by  
\[
U(\rho) := \int_{\mathbb{R}^d} f((\rho \star \chi)(x)) \, d\rho(x) + \int_{\mathbb{R}^d} V(x) \, d\rho(x) + \frac{1}{2} \int_{\mathbb{R}^d} \int_{\mathbb{R}^d} W(x-y) \, d\rho(x) \, d\rho(y),
\]
where $ \star $ denotes convolution. Then, the following properties hold:
\begin{enumerate}
    \item $ U \in C^2(\mathbf{W}_2(\mathbb{R}^d); \mathbb{R}) $.
    
    \item The Wasserstein gradient $ \nabla_{\mathbf{W}} U(\rho, x) $ is uniformly Lipschitz in $ x $, i.e.,
    \[
    \sup_{\rho \in \mathcal{P}_2(\mathbb{R}^d)} \sup_{x \neq y} \frac{|\nabla_{\mathbf{W}} U(\rho, x) - \nabla_{\mathbf{W}} U(\rho, y)|}{|x - y|} < \infty.
    \]

    \item The partial Wasserstein Hessian $ \nabla^2_{\mathbf{W}} U $ is uniformly bounded, i.e., 
    \[
    \sup_{(\rho, x) \in \mathcal{P}_2(\mathbb{R}^d) \times \mathbb{R}^d} |\nabla^2_{\mathbf{W}} U(\rho, x)| < \infty.
    \]

    % \item If $ \mu \in \mathcal{P}_2(\mathbb{R}^d) $ has no atoms, then the functional on $ L^2(\mathbb{R}^d;\mu)$ defined by
    % \[ U_{\mu}^{\#}(\xi) := U(\xi_{\#} \mu) \]
    % is of class $ C^{1,1}( L^2(\mathbb{R}^d;\mu); L^2(\mathbb{R}^d;\mu)) $.
\end{enumerate}

\label{C2_functionals}

\end{example}

\begin{proof}

We first establish that $U \in C^1\left(\mathbf{W}_2(\mathbb{R}^d); \mathbb{R}\right)$ by computing its Wasserstein gradient. First we compute its first variation. For the convolved internal energy, we let
\[ \mathcal{F}(\rho) := \int_{\R^d} f( (\rho \star \chi)(x) ) d\rho(x). \]
Fix $\rho, \rho' \in \mathcal{P}_2(\mathbb{R}^d)$, and define $\sigma := \rho' - \rho$. For $\varepsilon \neq 0$, set $\rho_\varepsilon := \rho + \varepsilon \sigma$. Then
\[
\lim_{\varepsilon \to 0} \frac{\mathcal{F}(\rho_{\e}) - \mathcal{F}(\rho)}{\e}
= \int_{\R^d}  f( (\rho \star \chi)(x) )d\sigma(x) + \int_{\R^d} \left( f'(\rho \star \chi)\rho \star \chi \right) d\sigma(x) 
\] Here we are using the notation for any $g \in L^{\infty}(\R^d)$
\[ \left((g \rho) \star \chi\right)(x) := \int_{\R^d} \chi(x-y) g(y) d\rho(y) \]

Hence, by standard identities of the first variations in Wasserstein space for the potential and interaction energy (see, e.g., \cite{santambrogio2017euclidean}), we conclude that
\[
\frac{\delta U}{\delta \rho}(\rho, x)
=
f((\rho \star \chi)(x)) + \left[f'\left(\rho \star \chi\right)\rho \star \chi\right](x)
+ V(x)
+ (W \star \rho)(x).
\]

Now to obtain the Wasserstein gradient, we differentiate in $x$ to see that
\begin{equation}
\nabla_{\mathbf{W}} U(\rho, x) = \int_{\mathbb{R}^d} \!\!\bigl[ f'\bigl((\rho * \chi)(x)\bigr) + f'\bigl((\rho * \chi)(z)\bigr) \bigr] \nabla \chi(x - z)\, d\rho(z) + \nabla V(x) + (\nabla W \star \rho)(x). \label{U_wasserstein_gradient}
\end{equation}

\noindent The joint continuity of $ \nabla_{\mathbf{W}} U $ follows from the smoothness of $ f' $, $ \nabla \chi $, $ \nabla V $, and $ \nabla W $ along with the observation that $\rho \mapsto  (\chi \star \rho)$  and $\rho \mapsto (\nabla \chi \star \rho)$ are Lipschitz continuous on the $\mathbf{W}_1(\mathbb{R}^d)$ metric uniformly in $x$ due the dual problem formulation of the $\mathbf{W}_1$ metric.  \smallskip

Property $(2)$ follows from the assumptions $ \nabla^2 V, \nabla^2 W, \nabla^2 \chi \in L^\infty(\mathbb{R}^d) $, $ f' \in L^\infty(\mathbb{R}) $, and the compact support of $ \chi $. 

\smallskip

\noindent For the partial Wasserstein Hessian, we first compute the first variation of \eqref{U_wasserstein_gradient} with $x$ fixed. For the internal energy term, we set $\rho_{\chi} := \rho \star \chi$. Then
\begin{align*}
\lim_{\varepsilon \to 0} 
\frac{\nabla_{\mathbf{W}} \mathcal{F}(\rho_{\varepsilon},x) 
      - \nabla_{\mathbf{W}} \mathcal{F}(\rho,x)}{\varepsilon}
&= \int_{\mathbb{R}^d} \!\!\bigl[
      f'(\rho_\chi(x)) + f'(\rho_\chi(z))
    \bigr]\nabla \chi(x - z)\, d\sigma(z) \\
&\quad + 
  \iint_{\mathbb{R}^d \times \mathbb{R}^d}
    \!\!\bigl[
      f''(\rho_\chi(x))\,\chi(x - v)
      + f''(\rho_\chi(z))\,\chi(z - v)
    \bigr]\nabla \chi(x - z)\, d\rho(z)\, d\sigma(v).
\end{align*}

\noindent For the interaction energy term,
\[
\lim_{\varepsilon \to 0}\frac{(\nabla W \star \rho_{\varepsilon})(x)
      - (\nabla W \star \rho)(x)}{\varepsilon}
= \int_{\mathbb{R}^d} \nabla W(x - x')\, d\sigma(x').
\]
Hence, the first variation of the Wasserstein gradient is
\[
\begin{aligned}
\frac{\delta\, \nabla_{\mathbf{W}} U}{\delta \mu}(\rho; x, x')
&= \nabla W(x - x') 
 + \bigl[f'(\rho_\chi(x)) + f'(\rho_\chi(x'))\bigr]\, \nabla \chi(x - x') \\
&\quad + \int_{\mathbb{R}^d} \!\!\bigl[
      f''(\rho_\chi(x))\, \chi(x - x')
      + f''(\rho_\chi(z))\, \chi(z-x')
    \bigr] \nabla \chi(x - z)\, d\rho(z).
\end{aligned}
\]

Differentiating in $x'$ gives the $(i,j)$-th entry of the partial Wasserstein Hessian:
\[
\begin{aligned}
-[\nabla^2_{\mathbf{W}} U]_{i,j}(\rho, x, x')
&= [\nabla^2 W(x - x')]_{i,j} \\[0.3em]
&\quad + \int_{\mathbb{R}^d}
      \!\!\bigl[
        f''(\rho_\chi(z))\, \partial_j \chi(z - x')
        + f''(\rho_\chi(x))\, \partial_j \chi(x - x')
      \bigr] \partial_i \chi(x - z)\, d\rho(z) \\[0.3em]
&\quad + \bigl[
      f'(\rho_\chi(x)) + f'(\rho_\chi(x'))
    \bigr] \partial^2_{i,j}\chi(x - x')
  - f''(\rho_\chi(x'))\, \partial_i \chi(x - x')\, (\rho * \partial_j \chi)(x').
\end{aligned}
\]

% \int_{\mathbb{R}^d} \partial_i \chi(z - x') \partial_j \chi(x-z) f''(\rho \star \chi(z)) \, dz + \partial^2_{i,j} W(x - x').

\noindent So to check Property (2), it suffices to check joint continuity of $ \nabla^2_{\mathbf{W}} U$. This follows from the smoothness of $ \chi $, $ W $, and $ f'' $ Property (3) follows from the boundedness of all the first and second derivative terms.
\end{proof}

% For the fourth property, we first apply Theorem~\ref{differentiability_equivalence} to deduce that when $\mu$ has no atoms
% \[
% \nabla U_{\mu}^{\#}(\xi) = \nabla_{\mathbf{W}} U((\xi)_{\#}\mu,\xi).
% \] Then $ \nabla U_{\mu}^{\#}(\xi)$ being Lipschitz on $L^2(\R^d;\mu)$ follows from
%  the triangle inequality, the second property of 
% Example \ref{C2_functionals}, and \eqref{gradient_Lipschitz_Wp} 
% along with the inequality
% \[
% W_{2}((\xi_1)_{\#}\mu,(\xi_2)_{\#}\mu) \leq 
% \|\xi_1 \;-\;\xi_2\|_{L^2(\mathbb{R}^d;\,\mu)}.\] \end

\subsection{Convexity and Differentiability of the Lifted Functional}

In this section, we study the differential and convexity properties of the lifted functionals $\phi_{\mu}^{\#} : L^2(\mathbb{R}^d; \mu) \to \mathbb{R}$ arising in Mean Field Games. The results of \cite{gangbo2022global,gangbo2019differentiability,carmona2018probabilistic} establish that the convexity and differentiability of $ \phi $ in the Wasserstein space $ \mathbf{W}_2 $ are closely linked to the convexity and differentiability of $ \phi_{\mu}^{\#} $ on $ L^2(\mathbb{R}^d; \mu) $.

\smallskip

\begin{definition}[Lifted Energy Functional]  \label{lift_def}
Let $\phi : \mathcal{P}_2(\mathbb{R}^d) \to \mathbb{R}$ and $\mu \in \mathcal{P}_2(\mathbb{R}^d)$. The \textnormal{lift} of $\phi$ \textnormal{by} $\mu$ is the functional $\phi^{\#}_{\mu} : L^2(\mathbb{R}^d; \mu) \to \mathbb{R}$ defined by
\[
\phi^{\#}_{\mu}(\xi) := \phi(\xi_{\#} \mu).
\]
\noindent Furthermore, if $\phi^{\#}_{\mu}$ is Fréchet differentiable at $\xi \in L^2(\mathbb{R}^d;\mu)$, we denote its Fréchet derivative in $L^2(\mathbb{R}^d; \mu)$ by $\nabla \phi^{\#}_{\mu}$.
\end{definition}

We begin by stating the connection between convexity $\phi^{\#}_{\mu}$ and $\phi$. We first recall the standard definition of $\lambda$-convexity for functions over Hilbert spaces. 

\smallskip
\begin{definition}[Flat Convexity]
Given any $\mu \in \mathcal{P}_2(\mathbb{R}^d)$, we say that $\phi^{\#}_{\mu} : L^2(\mathbb{R}^d; \mu) \to \mathbb{R}$, as defined in Definition~\ref{lift_def}, is $\lambda$-convex if
\[
\phi^{\#}_{\mu}((1 - t)\xi_1 + t \xi_2)
\leq (1 - t) \phi^{\#}_{\mu}(\xi_1) + t \phi^{\#}_{\mu}(\xi_2)
- \frac{\lambda}{2} t(1 - t) \|\xi_1 - \xi_2\|^2_{L^2(\mathbb{R}^d; \mu)}
\]
for all $\xi_1, \xi_2 \in L^2(\mathbb{R}^d; \mu)$ and $t \in [0,1]$.
\end{definition}

Let us also recall the notion of displacement convexity.

\begin{definition}[Displacement Convexity] \label{geodesic_convex}
A functional $\phi : \mathcal{P}_2(\mathbb{R}^d) \to \mathbb{R}$ is said to be $\lambda$-\textnormal{displacement convex} if, for any $\mu, \nu \in \mathcal{P}_2(\mathbb{R}^d)$ and any \textnormal{constant-speed geodesic} $(\gamma_t)_{t \in [0,1]}$ with $\gamma_0 = \mu$ and $\gamma_1 = \nu$, we have
\[
\phi(\gamma_t) \leq (1 - t) \phi(\mu) + t \phi(\nu) - \frac{\lambda}{2} t(1 - t) W_2^2(\mu, \nu), \quad \forall t \in [0,1].
\]
\end{definition}

Now we formally derive the connection between the partial Wasserstein Hessian and displacement convexity:

\begin{remark}[Displacement Convexity and Wasserstein Derivatives] \label{Hessian_displacement}
Let $\mu \in \mathcal{P}_2(\R^d)$ be absolutely continuous with respect to the Lebesgue measure.  
Fix $\nu \in \mathcal{P}_2(\R^d)$, and let $T$ denote the optimal transport map from $\mu$ to $\nu$.  Then the constant-speed geodesic $(\gamma_t)_{t \in [0,1]}$ from $\mu$ to $\nu$ is given by
\[
\gamma_t := \big((1-t)\mathrm{Id} + tT\big)_{\#} \mu.
\]

Assume the energy functional $\phi : \mathcal{P}_2(\R^d) \to \R$ is sufficiently smooth.  Then, denoting the velocity field by $v(x) := T(x) - x$ and $T_t(x) = (1-t)x+ tT(x)$, we have the following expressions for the first and second derivatives of $\phi$ along the geodesic:
\begin{equation}
\frac{d}{dt} \phi(\gamma_t) 
= \int_{\R^d} \left\langle \nabla_{\mathbf{W}} \phi(\gamma_t, T_t(x)),\, v(x) \right\rangle  d\mu(x), \label{first_deriv_phi}
\end{equation}
\begin{align}
\frac{d^2}{dt^2} \phi(\gamma_t)
&= \int_{\R^d}
\Big\langle
\nabla_x \nabla_{\mathbf{W}} \phi(\gamma_t, T_t(x))\, v(x),
\, v(x)
\Big\rangle
\, d\mu(x) \notag \\
&\quad
+ \iint_{\R^d \times \R^d}
\Big\langle
\nabla_{\mathbf{W}}^2 \phi(\gamma_t, T_t(x), T_t(z))\, v(z),
\, v(x)
\Big\rangle
\, d\mu(z)\, d\mu(x).
\label{second_deriv_phi}
\end{align}

Observe from \eqref{second_deriv_phi} that if the operator norm of $\nabla_x \nabla_{\mathbf{W}} \phi(\rho,x,z)$ and $\nabla^2_{\mathbf{W}} \phi(\rho,x,z)$ are uniformly bounded by $C$, then $\phi$ is $-2C$ displacement convex and $2C$ displacement concave when the initial measure is absolutely continuous. Indeed, by Cauchy-Schwarz
\[ \left| \frac{d^2}{dt^2} \phi(\gamma_t) \right| \leq  C \int_{\R^d} |v(x)|^2 d\mu(x) + C \iint_{\R^d \times \R^d} |v(z)| \cdot |v(x)| d\mu(z)d\mu(x).  \] So as $W_2(\mu,\nu) = ||v(x)||_{L^2(\R^d;\mu)}$, we conclude from Jensen's inequality that
\[  \left| \frac{d^2}{dt^2} \phi(\gamma_t) \right| \leq  C W_2^2(\mu,\nu) + C ||v(x)||_{L^1(\R^d;\mu)}^2 \leq 2C W_2^2(\mu,\nu).  \] 
We now focus on formally computing these derivatives. To compute the first time derivative, we use \eqref{first_variation} to see that
\[ \phi(\gamma_{t+h}) - \phi(\gamma_t) \approx \int_{\R^d} \frac{\delta \phi}{\delta \mu}( \gamma_t,x) d(\gamma_{t+h}-\gamma_t)(x) = \int_{\R^d} \left( \frac{\delta \phi}{\delta \mu}(\gamma_t, T_{t+h}) - \frac{\delta \phi}{\delta \mu}(\gamma_t,T_t) \right) d\mu.  \] Then by Taylor expanding in the spatial coordinate, we obtain that
\[ \approx h\int_{\R^d} \langle \nabla_{\mathbf{W}}  \phi(\gamma_t,T_t(x)), v(x) \rangle d\mu(x),   \] which shows the equality of the first time derivative. \smallskip

To see the second time derivative equality, let $F(\gamma_t) := \frac{d}{dt} \phi(\gamma_t),$ then by \eqref{first_deriv_phi}
\[ F(\gamma_{t+h})-F(\gamma_t) = (I) + (II) \] 
\small{ \[ = \int_{\R^d} \langle \nabla_{\mathbf{W}} \phi(\gamma_{t+h},T_{t+h}(x)) - \nabla_{\mathbf{W}} \phi(\gamma_{t+h},T_t(x)),v(x)\rangle d\mu(x)   + \int_{\R^d} \langle \nabla_{\mathbf{W}} \phi(\gamma_{t+h},T_{t}(x)) - \nabla_{\mathbf{W}} \phi(\gamma_{t},T_t(x)),v(x)\rangle d\mu(x).  \] } Then observe that by Taylor expanding in the spatial coordinate
\[ (I) \approx h\int_{\R^d} \langle  \nabla_x \nabla_{\mathbf{W}} \phi(\gamma_{t}, T_t(x)) v(x),v(x) \rangle d\mu(x)  \] and by using \eqref{first_variation}
\[ (II) \approx \iint_{\R^d \times \R^d} \langle \frac{\delta \nabla_{\mathbf{W}} \phi}{\delta \mu}(\gamma_t,T_t(x),z), v(x)      \rangle d(\gamma_{t+h}-\gamma_t)(z) d\mu(x)  \]
\[ = \iint_{\R^d \times \R^d} \langle \frac{\delta \nabla_{\mathbf{W}} \phi}{\delta \mu}(\gamma_t,T_t(x),T_{t+h}(z)) - \frac{\delta \nabla_{\mathbf{W}} \phi}{\delta \mu}(\gamma_t,T_t(x),T_{t}(z)) , v(x)      \rangle d\mu(z) d\mu(x), \]  so by Taylor expanding in the last spatial argument, we see that
\[ \approx h\iint_{\R^d \times \R^d} \langle \nabla^2_{\mathbf{W}} \phi(\gamma_t,T_t(x),T_t(z)) v(z),v(x) \rangle d\mu(z)d\mu(x). \] This implies the formula for the second time derivative.
\end{remark}

% \begin{remark}[Second Time Derivative of $U$] \label{d2U_dt2}
% \textcolor{red}{We now apply Remark~\ref{Hessian_displacement} to compute the formal second time derivative of the functional $U$ from Example~\ref{C2_functionals} along a constant-speed geodesic $(\gamma_t)$.  With the same notation from Remark~\ref{Hessian_displacement}, we formally have that}

% \begin{align*}
% \frac{d^2}{dt^2} U(\gamma_t) 
% &= \int_{\R^d} \left\langle \nabla^2 V(T_t(x))\, v(x),\, v(x) \right\rangle \, d\mu(x) \\
% &\quad + \frac{1}{2} \iint_{\R^d \times \R^d} 
% \left\langle \nabla^2 W(T_t(x) - T_t(z))\, (v(x) - v(z)),\, (v(x) - v(z)) \right\rangle 
% \, d\mu(x)\, d\mu(z) \\
% &\quad + \int_{\R^d} 
% \left\langle \left(f'(\gamma_t \star \chi) \star \nabla^2 \chi\right)(T_t(x))\, v(x),\, v(x) \right\rangle 
% \, d\mu(x) \\
% &\quad - \sum_{i,j=1}^{d} \iiint_{\R^d \times \R^d \times \R^d} 
% \p_i \chi(u - T_t(z))\, \p_j \chi(T_t(x) - u)\, f''(\gamma_t \star \chi)(u)\, v_j(z)\, v_i(x)\, 
% du\, d\mu(x)\, d\mu(z).
% \end{align*}
% \end{remark}

\smallskip

The notion of geodesic convexity of $\phi$ and convexity of the lift $\phi^{\#}_{\mu}$ are equivalent when $ \phi $ is continuous, as established in \cite{gangbo2022global}.

\begin{theorem} \label{convex_equivalence} \cite[Lemma 3.6]{gangbo2022global}
Let $ \phi : \mathcal{P}_2(\mathbb{R}^d) \to \mathbb{R} $ be \textnormal{continuous} and assume that $\mu \in \mathcal{P}_2(\R^d)$ has no atoms. Then the following statements are equivalent:
\begin{center}
\begin{enumerate}
    \item $\phi$ is $\lambda$-displacement convex on $\mathcal{P}_2(\mathbb{R}^d)$.
    \item $\phi_{\mu}^{\#}$ is $\lambda$-convex on $L^2(\mathbb{R}^d; \mu)$.
\end{enumerate}
\end{center}

\end{theorem}
\begin{proof}
Let $m$ be the uniform probability measure on $[0,1]^d$. The result was proven in \cite[Lemma 3.6]{gangbo2022global} for $m$. To extend this to a general $\mu \in \mathcal{P}_2(\mathbb{R}^d)$ with no atoms, we claim that $\phi^{\#}_m$ is $\lambda$-convex on $L^2(\mathbb{R}^d; m)$ if and only if $\phi^{\#}_\mu$ is $\lambda$-convex on $L^2(\mathbb{R}^d; \mu)$.
\smallskip

We first show that $\lambda$-convexity of $\phi^{\#}_\mu$ implies $\lambda$-convexity of $\phi^{\#}_m$. Since $\mu \in \mathcal{P}_2(\mathbb{R}^d)$ is atomless and $\mathbb{R}^d$ is a Polish space, there exists a measurable map $T$ such that $T_{\#} \mu = m$ (see \cite[Page 379]{carmona2018probabilistic}). Because $m \in \mathcal{P}_2(\R^d)$, we have that $T \in L^2(\mathbb{R}^d; \mu)$. For $i \in \{1,2\}$, fix $\xi_i \in L^2(\mathbb{R}^d; m)$ and note that $\xi_i \circ T \in L^2(\mathbb{R}^d; \mu)$. Since $\phi^{\#}_\mu$ is $\lambda$-convex, we obtain
\[
\phi^{\#}_\mu\big((1-t) \xi_1 \circ T + t \xi_2 \circ T\big) \leq (1-t) \phi^{\#}_\mu(\xi_1 \circ T) + t \phi^{\#}_\mu(\xi_2 \circ T) - \frac{\lambda}{2} t(1-t) \|\xi_1 \circ T - \xi_2 \circ T\|^2_{L^2(\mathbb{R}^d; \mu)}.
\]
Since $T_{\#} \mu = m$, this simplifies to
\[
\phi^{\#}_m\big((1-t) \xi_1 + t \xi_2\big) \leq (1-t) \phi^{\#}_m(\xi_1) + t \phi^{\#}_m(\xi_2) - \frac{\lambda}{2} t(1-t) \|\xi_1 - \xi_2\|^2_{L^2(\mathbb{R}^d; m)},
\]
which is precisely the definition of $\lambda$-convexity for $\phi^{\#}_m$.

\smallskip

The reverse implication follows by the same argument with the roles of $\mu$ and $m$ interchanged.
\end{proof}

Now we focus on derivatives of $\phi^{\#}_{\mu}$. Under mild growth conditions on $\nabla_{\mathbf{W}} \phi$, the lift is differentiable on $L^2(\R^d;\mu)$:

\begin{proposition} \cite[Proposition 5.48]{carmona2018probabilistic}  \label{differentiability_equivalence} 
Assume that $\mu \in \mathcal{P}_2(\mathbb{R}^d)$ have no atoms. Suppose $\phi \in C^1(\mathbf{W}_2(\mathbb{R}^d);\mathbb{R})$ is such that for any bounded $\mathcal{K} \subset \mathcal{P}_2(\R^d)$
\begin{enumerate}
    \item $x \mapsto \nabla_{\mathbf{W}} \phi(\rho, x)$ has at most linear growth, uniformly for $\rho \in \mathcal{K}$.
    \item $x \mapsto \frac{\delta \phi}{\delta \mu}(\rho, x)$ has at most quadratic growth, uniformly for $\rho \in \mathcal{K}$.
\end{enumerate}
Then the lift $\phi^{\#}_{\mu}$ is Fréchet differentiable on $L^2(\mathbb{R}^d; \mu)$, with
\[
\nabla \phi^{\#}_{\mu}(\xi)(x) = \nabla_{\mathbf{W}} \phi(\xi_{\#} \mu, \xi(x)) \quad \mu \textnormal{ almost everywhere}.
\]
\end{proposition}

\begin{proof}
We apply \cite[Proposition 5.48]{carmona2018probabilistic} to the atom-less probability space $(\mathbb{R}^d, \mathcal{B}(\mathbb{R}^d), \mu)$, using the fact that for this probability space, the law of a random variable $X$ is given by $X_{\#} \mu$. \end{proof}

In \cite[Corollary 3.22]{gangbo2019differentiability} the Fréchet differentiability of the lift at $\xi$ is shown to be equivalent to the Wasserstein differentiability of $\phi$ at $\xi_{\#} \mu$. In this case, we have that
\[ \nabla \phi^{\#}_{\mu}(\xi)(x) = \nabla_{\mathbf{W}} \phi(\xi_{\#} \mu,\xi(x)) \quad \mu \text{ almost everywhere.} \] We refer to \cite{gangbo2019differentiability} for the general definition of the Wasserstein gradient in terms of the sub and super differential of $\phi$. \smallskip

Now we link notions of $C^{1,1}$ of $\phi$ with its lift.

 \begin{definition}[$C^{1,\alpha}(X;X)$ Functions]  
 Let $X$ be a Hilbert space and $\alpha \in (0,1]$. A functional $u: X \to \mathbb{R}$ is said to belong to the class $C^{1,\alpha}(X;X)$ if it is Fréchet differentiable everywhere, and there exists a constant $C > 0$ such that for all $x, y \in X$,
 \[
 \big\| \nabla u(x) - \nabla u(y) \big\|_X 
 \leq C \big\| x - y \big\|_X^{\alpha}.
 \]

\noindent We will say $u \in C^{1,\alpha}_b(X;X)$ if $u \in C^{1,\alpha}(X;X)$ and
\[ \sup_{x \in X} |u(x)|+|\nabla u(x)|  < \infty. \]

 \end{definition}

\begin{example} \label{C11_Lifted_Functional}
Under the notation and assumptions of Example~\ref{C2_functionals}, if $\mu \in \mathcal{P}_2(\mathbb{R}^d)$ is atomless, then the lifted functional $U^{\#}_{\mu}$ belongs to $C^{1,1}(\mathbb{H}_{\mu}; \mathbb{H}_{\mu})$ for $\mathbb{H}_{\mu} := L^2(\R^d;\mu)$.
\end{example}

\begin{proof} By Proposition \ref{differentiability_equivalence}, we have that
\[ \nabla U^{\#}_{\mu}(\xi) = \nabla_{\mathbf{W}} U(\xi_{\#} \mu,\xi). \] Hence, for $\xi_i \in \mathbb{H}_{\mu}$, the triangle inequality implies
    \[ ||\nabla U^{\#}_{\mu}(\xi_1) - \nabla U^{\#}_{\mu}(\xi_2)||_{\mathbb{H}_{\mu}} \leq \] \[ || \nabla_{\mathbf{W}} U( (\xi_1)_{\#} \mu,\xi_1) - \nabla_{\mathbf{W}} U( (\xi_1)_{\#} \mu,\xi_2)||_{\mathbb{H}_{\mu}} +  || \nabla_{\mathbf{W}} U( (\xi_1)_{\#} \mu,\xi_2) - \nabla_{\mathbf{W}} U( (\xi_2)_{\#} \mu,\xi_2) ||_{\mathbb{H}_{\mu}}. \] Now by properties (2) and (3) of Example \ref{C2_functionals}, there exists a $C>0$ such that
\[ \leq C (||\xi_1-\xi_2||_{\mathbb{H}_{\mu}} + W_2( (\xi_1)_{\#} \mu, (\xi_2)_{\#} \mu)) \leq 2C ||\xi_1-\xi_2||_{\mathbb{H}_{\mu}}. \]
\end{proof}

\section{The Lagrangian Trapezoidal Scheme: Definition and Numerical Stability} \label{stability_section}

In this section, we define a higher-order implicit method for the gradient flow of $\phi$ and establish stability properties of the energy functional and its Wasserstein gradient. Our scheme corresponds to the implicit trapezoidal rule applied to the Lagrangian Flow \eqref{L2_Lagrangian_Flow}. We begin by defining the variational problem associated with the scheme and then prove its well-posedness.

\begin{definition}[Lagrangian Trapezoidal 
Scheme] \label{trapezoid_def} Fix a reference measure $\rho_0 \in \mathcal{P}_{2}(\mathbb{R}^d)$ and assume that $\phi^{\#}_{\rho_0}$ is Fr\'echet differentiable on $L^2(\R^d;\rho_0)$. Fix a time step $\tau > 0$ and an initial condition $X_0^{\tau} \in L^2(\mathbb{R}^d; \rho_0)$. Then, for $n \geq 0$, define the discrete Lagrangian flows $X_{n+1}^\tau$ via:
\begin{equation}
X_{n+1}^\tau \in \operatorname*{arg\,min}_{\xi \in L^2(\mathbb{R}^d; \rho_0)}
\left\{
\frac{1}{2} \phi^{\#}_{\rho_0}(\xi)
+ \frac{1}{2} \int_{\mathbb{R}^d} \langle \nabla \phi^{\#}_{\rho_0}(X^{\tau}_n)(x), \xi(x) \rangle \, d\rho_0(x)
+ \frac{1}{2\tau} \|\xi - X_n^\tau\|_{L^2(\mathbb{R}^d; \rho_0)}^2 
\right\}.
\label{trapezoid_var_problem}
\end{equation}

\noindent Finally, define the discrete trapezoidal rule solutions as the measures
\[
\rho_n^\tau := (X_n^\tau)_{\#} \rho_0.
\]
\end{definition}

Now, under appropriate assumptions on $\phi$, we show that the scheme is well-defined for sufficiently small $\tau$.

\begin{lemma}[Existence and Uniqueness of the Minimizer] \label{min_exist} Given a $\mu \in \mathcal{P}_{2}(\R^d)$, we assume that $\phi^{\#}_{\mu}$ is Fr\'echet differentiable on $L^2(\R^d;\mu)$. We further assume that $\phi^{\#}_{\mu}$ is $\lambda$-convex on $L^2(\R^d;\mu)$ for some $\lambda \in \R$. Further, let $\tau > 0$ satisfy $(\lambda / 2 + 1 / \tau) > 0$ and fix an $v \in L^2(\mathbb{R}^d; \mu)$. Define the operator $\Phi_{\tau, \mu, v}: L^2(\mathbb{R}^d; \mu) \to \mathbb{R}$ as
\[
\Phi_{\tau, \mu, v}(\xi) := \frac{1}{2} \left(\phi^{\#}_{\mu}(\xi) + \int_{\mathbb{R}^d} \langle \nabla \phi^{\#}_{\mu}( v(x)), \xi(x) \rangle \, d\mu(x) \right) + \frac{1}{2 \tau} \|\xi - v\|^2_{L^2(\mathbb{R}^d; \mu)}.
\]
\noindent Additionally, assume that
\begin{equation}
\| \nabla \phi^{\#}_{\mu}(v)(x) \|^2_{L^2(\mathbb{R}^d; \mu)} < \infty, \label{L2_gradient_bound_assume}
\end{equation}
and that the energy functional $\phi$ is proper, i.e., $\inf_{\mu \in \mathcal{P}_2(\mathbb{R}^d)} \phi(\mu) > -\infty$. Then, there exists a unique minimizer of $\Phi_{\tau, \mu, \eta}$ in $L^2(\mathbb{R}^d; \mu)$.
\end{lemma}

\begin{proof} 
From our assumptions it follows that $\Phi_{\tau, \mu,v}(\xi)$ is $(\lambda/2 + 1/\tau)$-convex in $L^2(\mathbb{R}^d; \mu)$, ensuring at most one minimizer. 

The existence of a minimizer follows from the continuity and convexity of $\Phi_{\tau, \mu,v}(\xi)$. Define the constant 
\[
\beta := \inf_{\xi \in L^2(\mathbb{R}^d; \mu)} \Phi_{\tau, \mu,v}(\xi) > -\infty,
\]
where the lower bound holds due to \eqref{L2_gradient_bound_assume} and the properness of $\phi$.

Now, consider a minimizing sequence $\{\xi_n\}$ for $\Phi_{\tau,\mu,v}$. We claim that the $(\lambda/2 + 1/\tau)$-convexity of $\Phi_{\tau,\mu,v}$ implies that $\{\xi_n\}$ is a Cauchy sequence in $L^2(\mathbb{R}^d; \mu)$. Defining the geodesic $\gamma(t) := (1-t) \xi_n + t \xi_m$ implies from convexity that
\[
\beta \leq \Phi_{\tau, \mu, v}\left(\gamma\left(\frac{1}{2}\right)\right) 
\leq \frac{1}{2} \Phi_{\tau, \mu, v}(\xi_n) + \frac{1}{2} \Phi_{\tau, \mu, v}(\xi_m) - \frac{1}{8} \underbrace{(\lambda/2 + 1/\tau)}_{> 0} \|\xi_n - \xi_m\|^2_{L^2(\mathbb{R}^d; \mu)}.
\]
Since $\{\xi_n\}$ is a minimizing sequence, we have $\Phi_{\tau, \mu, v}(\xi_n) \to \beta$. This implies $\{\xi_n\}$ is a Cauchy sequence, as $\|\xi_n - \xi_m\|^2_{L^2(\mathbb{R}^d; \mu)}$ must tend to 0 as $n, m \to \infty$. The continuity of $\Phi_{\tau, \mu, v}$ ensures that the limit of $\xi_n$ in $L^2(\mathbb{R}^d; \mu)$ is the unique minimizer of $\Phi_{\tau, \mu,v}$.
\end{proof}

Now, we characterize the minimizer of $\Phi_{\tau,\mu, v}$.

\begin{lemma}  \label{critical_points}
Under the same notation and assumptions as Lemma \ref{min_exist}, the unique minimizer $\xi^*$ of $\Phi_{\tau,\mu,v}$ is the unique solution over $L^2(\R^d;\mu)$ to the implicit equation:
\[
\xi^*(x) = \eta(x) - \frac{\tau}{2} \left( \nabla \phi^{\#}_{\mu}(\xi^*)(x) + \nabla \phi^{\#}_{\mu}(v)(x)  \right).
\] 
\end{lemma}

\begin{proof} 
Since $\Phi_{\tau,\mu,v}$ is strictly convex and Fréchet differentiable, its unique minimizer is also its unique critical point. The implicit equation relationship then follows from
\[
\nabla \Phi_{\tau,\mu,v}(\xi) = \frac{1}{2} \left( \nabla \phi^{\#}_{\mu}(\xi)  + \nabla \phi^{\#}_{\mu}(v)  \right) + \frac{1}{\tau} (\xi - v).
\]
\end{proof}

Next, we will derive a useful estimate on the gradient of the minimizer of $\Phi_{\tau,\mu,\eta}$ that ensures the gradient norm is bounded along our trapezoid iterates.

\begin{lemma}[Gradient Estimate]  \label{gradient_along_flow}
Under the assumptions and notation of Lemma \ref{min_exist}, define
\[
C(\lambda,\tau) :=
\begin{cases}
1, & \lambda \geq 0, \\
e^{2|\lambda_{\tau}| \tau}, & \lambda < 0,
\end{cases}
\quad \textnormal{where} \quad 
\lambda_{\tau} := \frac{1}{2\tau} \log \left( \frac{1+\lambda \tau}{1-\lambda \tau} \right).
\]

\noindent Then the unique minimizer $\xi^* \in L^2(\mathbb{R}^d;\mu)$ of $\Phi_{\tau,\mu,v}$ satisfies
\[
\|\nabla \phi^{\#}_{\mu}(\xi^*) \|^2_{L^2(\mathbb{R}^d;\mu)} \leq C(\lambda,\tau) \|\nabla \phi^{\#}_{\mu}(v) \|^2_{L^2(\mathbb{R}^d;\mu)}.
\]
\end{lemma}

\begin{proof}
By $\lambda$-convexity, we have that for any $\eta \in L^2(\mathbb{R}^d;\mu)$
\[
\langle \nabla \phi^{\#}_{\mu}( \xi^*) - \nabla \phi^{\#}_{\mu}( v), \xi^* - v \rangle_{L^2(\mathbb{R}^d;\mu)} \geq \lambda \|\xi^* - v\|_{L^2(\mathbb{R}^d;\mu)}^2.
\]
Using Lemma \ref{critical_points}, we have that the above expression is
\begin{equation}
\frac{\tau}{2} \left( \|\nabla  \phi^{\#}_{\mu} (v)\|^2_{L^2(\mathbb{R}^d;\mu)} - \|\nabla \phi^{\#}_{\mu}(\xi^*) \|^2_{L^2(\mathbb{R}^d;\mu)} \right) \geq \lambda \|\xi^*-v\|^2_{L^2(\mathbb{R}^d;\mu)}.
\label{monotone_gradient}
\end{equation} \noindent This implies the claim for $\lambda \geq 0$. \smallskip

\noindent Now, consider the case when $\lambda < 0$. Applying Lemma \ref{critical_points} to \eqref{monotone_gradient}, we express $\lambda \|\xi^*-v\|^2_{L^2(\mathbb{R}^d;\mu)}$ as
\begin{equation}
\lambda \frac{\tau^2}{4} \left( \|\nabla \phi^{\#}_{\mu}(v)\|^2_{L^2(\mathbb{R}^d;\mu)} 
+ \|\nabla \phi^{\#}_{\mu}(\xi^*)\|^2_{L^2(\mathbb{R}^d;\mu)}
+ 2\langle \nabla \phi^{\#}_{\mu}(v), \nabla \phi^{\#}_{\mu}(\xi^*)  \rangle_{L^2(\mathbb{R}^d;\mu)}  
\right).
\label{squared_distance_equality}
\end{equation}
Since $\lambda < 0$, we apply Cauchy-Schwarz and the inequality $ab \leq \frac{a^2}{2}  + \frac{b^2}{2}$ to obtain
\begin{equation}
\lambda \|\xi^* - v\|^2_{L^2(\mathbb{R}^d;\mu)} \geq  \lambda \frac{\tau^2}{2} \left( \|\nabla \phi^{\#}_{\mu}(v)\|^2_{L^2(\mathbb{R}^d;\mu)}
+ \|\nabla \phi^{\#}_{\mu}(\xi)\|^2_{L^2(\mathbb{R}^d;\mu)} \right).
\end{equation}
Thus, from \eqref{monotone_gradient}, we obtain the desired inequality of
\[
\|\nabla\phi^{\#}_{\mu}(v)\|^2_{L^2(\mathbb{R}^d;\mu)} \geq  \frac{(1+\lambda \tau)}{(1-\lambda \tau)} \|\nabla  \phi^{\#}_{\mu}(\xi^*)\|^2_{L^2(\mathbb{R}^d;\mu)}.\] \end{proof}

Now we focus on properties of our scheme (Definition \ref{trapezoid_var_problem}). We recall the following quantity, which will be frequently used
\begin{equation}
\mathcal{C}(\lambda,\tau) :=
\begin{cases}
1, & \lambda \geq 0, \\
e^{2|\lambda_{\tau}| \tau}, & \lambda < 0,
\end{cases} = \max\{1,e^{2 \lambda_{\tau} \tau} \}
\quad \textnormal{where} \quad 
\lambda_{\tau} := \frac{1}{2\tau} \log \left( \frac{1+\lambda \tau}{1-\lambda \tau} \right).
\end{equation}

\noindent Now Assumption \ref{assume_rough}, Lemmas   \ref{min_exist}, \ref{critical_points}, and \ref{gradient_along_flow} implies the following for our trapezoid scheme:

\begin{theorem}[Trapezoid Rule and gradient Estimate] \label{trapezoid_well_defined} 
Under Assumption \ref{assume_rough} and using its notation, there exists a unique minimizer in $\mathbb{H}$ of \eqref{trapezoid_var_problem} satisfying, for $\rho_0$-a.e. $x$,
\begin{equation}
    X^{\tau}_{n+1}(x) = X^{\tau}_{n}(x) - \frac{\tau}{2} \left( \nabla \phi^{\#}_{\rho_0}(X^{\tau}_{n+1})(x) + \nabla \phi^{\#}_{\rho_0}(X^{\tau}_n)(x) \right).
    \label{trapezoid_implicit}
\end{equation}
Additionally, the lifted gradient along the iterates satisfies the estimate:
\begin{equation}
    \|\nabla \phi^{\#}_{\rho_0}(X_{n+1}^{\tau})\|^2_{\mathbb{H}} 
    \leq \mathcal{C}(\lambda,\tau) \cdot \|\nabla \phi^{\#}_{\rho_0}(X^{\tau}_n)\|^2_{\mathbb{H}}.
    \label{gradient_bounds}
\end{equation}
\end{theorem}

% \begin{proof} 
% We proceed by induction. For $n=0$, the claim follows from Assumption \ref{assume_rough}, which allows us to apply Lemmas \ref{min_exist}, \ref{critical_points}, and \ref{gradient_along_flow}. \smallskip

% \noindent Now, assuming the statement holds for some $n \in \mathbb{N}$, we have by iteratinon of \ref{gradient_bounds}:
% \[
%     \|\nabla \phi^{\#}_{\rho_0}(X_n^{\tau})\|^2_{\mathbb{H}} 
%     \leq (\mathcal{C}(\lambda))^n \cdot \|\nabla \phi^{\#}_{\rho_0}(X_0^{\tau})\|^2_{\mathbb{H}} < \infty,
% \]
% where the last inequality follows from Assumption \ref{assume_rough}. This allows us to again apply Lemmas \ref{min_exist}, \ref{critical_points}, and \ref{gradient_along_flow}, concluding the theorem for $n+1$.
% \end{proof}

Beyond the gradient estimate \eqref{gradient_bounds}, we establish additional stability properties.

\begin{lemma}[Energy is Almost Decreasing] \label{energy_almost_decreasing}
Under Assumption \ref{assume_rough} and using its notation, with $X_n^\tau$ and $\rho_n^\tau$ defined as in Definition~\ref{trapezoid_def}, we have the following estimate:
\begin{equation}
    \phi^{\#}_{\rho_0}(X^{\tau}_{n+1}) - \phi^{\#}_{\rho_0}(X^{\tau}_{n})
    \leq \frac{\tau}{4} \left( \|\nabla \phi^{\#}_{\rho_0}(X_n^\tau)\|^2_{\mathbb{H}} 
    - \|\nabla \phi^{\#}_{\rho_0}(X_{n+1}^\tau)\|^2_{\mathbb{H}} \right).
    \label{one_step_almost_decrease}
\end{equation}
% In particular, for any $n > m$,
% \[
%    \phi^{\#}_{\rho_0}(X^{\tau}_{n}) - \phi^{\#}_{\rho_0}(X^{\tau}_{m}) 
%     \leq \frac{\tau}{4} \cdot (\mathcal{C}(\lambda))^{n-m} \cdot \|\nabla \phi^{\#}_{\rho_0}(X_0^{\tau})\|_{\mathbb{H}}.
% \]
\end{lemma}

\begin{proof}
Using the competitor $X_n^\tau$ in \eqref{trapezoid_var_problem}, we obtain:
\begin{equation}
   \phi^{\#}_{\rho_0}(X^{\tau}_{n+1}) - \phi^{\#}_{\rho_0}(X^{\tau}_{n})
    \leq \langle \nabla \phi^{\#}_{\rho_0}(X_n^\tau), X_n^\tau - X_{n+1}^\tau \rangle_{\mathbb{H}} 
    - \frac{1}{\tau} \|X_{n+1}^\tau - X_n^\tau\|^2_{\mathbb{H}}. \label{competitor_bound}
\end{equation}
Applying equation \eqref{trapezoid_implicit}, the right-hand side simplifies to:
\small\[
    \langle \nabla \phi^{\#}_{\rho_0}(X_n^\tau), X_n^\tau - X_{n+1}^\tau \rangle_{\mathbb{H}} 
  - \frac{1}{\tau} \|X_{n+1}^\tau - X_n^\tau\|^2_{\mathbb{H}} =\frac{\tau}{2} \left( \|\nabla \phi^{\#}_{\rho_0}(X_n^\tau)\|^2_{\mathbb{H}}  
    + \langle \nabla \phi^{\#}_{\rho_0}(X_n^\tau), \nabla \phi^{\#}_{\rho_0}(X_{n+1}^\tau) \rangle_{\mathbb{H}} \right) - \frac{1}{\tau} \|X_{n+1}^\tau - X_n^\tau\|^2_{\mathbb{H}}.
\]

\normalsize
\noindent Using Theorem~\ref{trapezoid_well_defined} on the last term, we find that the above expression further simplifies to:
\[
     \langle \nabla \phi^{\#}_{\rho_0}(X_n^\tau), X_n^\tau - X_{n+1}^\tau \rangle_{\mathbb{H}} 
  - \frac{1}{\tau} \|X_{n+1}^\tau - X_n^\tau\|^2_{\mathbb{H}} =\frac{\tau}{4} \left( \|\nabla \phi^{\#}_{\rho_0}(X_n^\tau)\|^2_{\mathbb{H}} 
    - \|\nabla \phi^{\#}_{\rho_0}(X_{n+1}^\tau)\|^2_{\mathbb{H}} \right).
\] Using the above display in \eqref{competitor_bound} implies  \eqref{one_step_almost_decrease}. 
\end{proof}

% In Section \ref{further_stability}, we will show that when $\lambda \geq 0$, $\nabla \phi^{\#}_{\rho_0}$ is $L$-Lipschitz, and $\tau \leq L^{-1}$, the above estimates can be refined to demonstrate that the energy decreases and the gradient decays exponentially fast. These refined estimates are based on the Discrete Evolution Variational Inequality developed in Section \ref{Discrete_EVI_Section}. \smallskip

To conclude this section, we establish a stability estimate similar to the classical one for the JKO scheme, along with a uniform Lipschitz time estimate for the numerical solutions.

\begin{lemma}[Stability Estimate] \label{classical_stability_lemma}  
Under Assumption \ref{assume_rough} and using its notation, with $X_n^\tau$ and $\rho_n^\tau$ defined as in Definition~\ref{trapezoid_def}, we have the following estimate for any $N$ with $N\tau \leq T$:
\begin{equation} 
    \sum_{j=0}^{N} \frac{W_2^2(\rho^{\tau}_j,\rho^{\tau}_{j+1}) }{\tau} 
    \leq \sum_{j=0}^{N} \frac{\|X_{j+1}^{\tau}-X_j^{\tau}\|^2_{\mathbb{H}} }{\tau} 
    \leq T \cdot \tilde{\mathcal{C}}(\lambda,\tau,T) \cdot \| \nabla \phi^{\#}_{\rho_0}(X_0^{\tau})\|^2_{\mathbb{H}}.
    \label{classical_stability}
\end{equation}
Also for any $n,m \in \mathbb{N}$ with $n,m \leq N$, we have the Lipschitz in time estimate:
\begin{equation}
    W_2(\rho_n^{\tau},\rho_m^{\tau}) 
    \leq \|X^{\tau}_{n}-X^{\tau}_m\|_{\mathbb{H}} 
    \leq    \tau \cdot |n-m| \cdot \sqrt{\tilde{\mathcal{C}}(\lambda,\tau,T)}  \cdot \|\nabla \phi^{\#}_{\rho_0}(X^{\tau}_0)\|_{\mathbb{H}}.
    \label{lipschitz_bound}
\end{equation}
where 
\[
    \tilde{\mathcal{C}}(\lambda,\tau,T) = \max \{ 1,e^{- 2\lambda_{\tau} T} \}.
\]
\end{lemma}

\begin{proof}
First, observe that for any $n,m \in \mathbb{N}$:
\[
    W_2^2(\rho_n^{\tau},\rho_{m}^{\tau}) = W_2^2( (X_n^{\tau})_{\#} \rho_0, (X_{m}^{\tau})_{\#} \rho_0) 
    \leq \|X^{\tau}_n-X^{\tau}_{m}\|^2_{\mathbb{H}},
\]
where the final inequality follows from Lemma \ref{same_measure_bound}. Thus, it suffices to show the bound on $X^{\tau}_n$. \smallskip

For \eqref{classical_stability}, we apply Theorem \ref{trapezoid_well_defined} to obtain:
\[
    \|X_{j+1}^{\tau}- X_j^{\tau}\|^2_{\mathbb{H}} 
    = \frac{\tau^2}{4} \|\nabla \phi^{\#}_{\rho_0}(X^{\tau}_{j+1}) +  \nabla \phi^{\#}_{\rho_0}(X^{\tau}_{j})\|^2_{\mathbb{H}}.
\]
By expanding out the square and applying Cauchy-Schwarz, we see that
\[
    \|X_{j+1}^{\tau}- X_j^{\tau}\|^2_{\mathbb{H}} 
    \leq \frac{\tau^2}{2} \left( \|\nabla \phi^{\#}_{\rho_0}(X^{\tau}_{j+1})\|^2_{\mathbb{H}} 
    + \|\nabla \phi^{\#}_{\rho_0}(X^{\tau}_{j})\|^2_{\mathbb{H}} \right).
\]
Then, applying the gradient estimate in Theorem \ref{trapezoid_well_defined}, we obtain:
\[
    \|X_{j+1}^{\tau}- X_j^{\tau}\|^2_{\mathbb{H}} 
    \leq \tau^2  \tilde{\mathcal{C}}(\lambda,\tau,T)  \|\nabla \phi^{\#}_{\rho_0}(X_0^{\tau})\|^2_{\mathbb{H}}.
\]
By summing, we obtain the bound \eqref{classical_stability}. \smallskip

To obtain \eqref{lipschitz_bound}, observe that for $n \geq m$ the implicit equation in Theorem \ref{trapezoid_well_defined} implies:
\[
    X^{\tau}_n - X^{\tau}_m = \sum_{j=m}^{n-1} (X^{\tau}_{j+1} - X^{\tau}_j) 
    = -\frac{\tau}{2} \sum_{j=m}^{n-1}  \left( \nabla \phi^{\#}_{\rho_0}(X^{\tau}_{j+1}) +  \nabla \phi^{\#}_{\rho_0}(X^{\tau}_{j}) \right).
\]
Using Theorem \ref{trapezoid_well_defined}, we conclude:
\[
    \|X^{\tau}_n - X^{\tau}_m\|_{\mathbb{H}} 
    \leq |n-m| \cdot \tau \cdot \sqrt{\tilde{\mathcal{C}}(\lambda,\tau,T)} \cdot \|\nabla \phi^{\#}_{\rho_0}(X^{\tau}_0)\|_{\mathbb{H}}.
\]
\end{proof}

\section{\texorpdfstring{$O(\tau^2)$}{O(tau2)} Convergence via Differential Methods}
\label{higher_order_convg}

 We now establish the higher-order convergence of the Lagrangian Trapezoidal Scheme under the assumption that both $\phi$ and the limiting velocity field are sufficiently smooth by extending the finite dimensional proofs to $\mathbb{H}$. First we begin by establishing existence of \eqref{L2_Lagrangian_Flow}.

\begin{lemma}
Assume that $\phi^{\#}_{\rho_0} \in C_b^{1,1}(\mathbb{H};\mathbb{H})$. Then for any $\eta \in \mathbb{H}$, there exists a unique map $X \in L^2_{\textnormal{loc}}([0,\infty) \times \R^d;\rho_0)$ such that
\[
\begin{cases} 
\dot{X}(t) = -\nabla \phi^{\#}_{\rho_0}(X(t)) \quad \text{on } [0,\infty), \\ 
X(0) = \eta.
\end{cases}
\]
\end{lemma}

\begin{proof} This follows from a Banach's Fixed Point theorem argument because $\phi^{\#}_{\rho_0} \in C_b^{1,1}(\mathbb{H};\mathbb{H})$.
\end{proof}

To ensure the initial measure is $\rho_0$, we will often take as initial data $\eta = \text{Id} \in \mathbb{H}$. In addition, to ensure higher order convergence of our trapezoid scheme, we will also need to make a $C^{2,\alpha}([0,T];\mathbb{H})$ assumption on the Lagrangian flow.

\begin{definition} \label{C11_X}
Let $T > 0$, $\alpha \in (0,1]$, and $X$ be a Hilbert space. A function $u: [0,T] \to X$ belongs to $C^{1,\alpha}([0,T];X)$ if it is Fréchet differentiable and there exists a constant $C > 0$ such that its Fréchet derivative $\dot{u}: [0,T] \to X$ satisfies  
\begin{equation}
\|u(t+h) - u(t) - h \dot{u}(t)\|_{X} \leq C |h|^{1+\alpha}, \quad \forall \, 0 \leq t, t+h \leq T. \label{time_taylor_expand}
\end{equation}
We say $u \in C^{1,\alpha}_{\textnormal{loc}}([0,\infty);X)$ if $u \in C^{1,\alpha}([0,T];X)$ for every $T > 0$. 
\end{definition}

\begin{assumption}[Higher-Order Convergence]  \label{higher_order_assumption}
Fix $\rho_0 \in \mathcal{P}_2(\R^d)$ and define $\mathbb{H} := L^2(\R^d;\rho_0)$. Assume $\phi^{\#}_{\rho_0} \in C_b^{1,1}(\mathbb{H};\mathbb{H})$, and define the Lipschitz constant of its gradient by
\[ 
L(\phi) := \sup_{\xi_1 \neq \xi_2 \in \mathbb{H}} \frac{\|\nabla \phi^{\#}_{\rho_0}(\xi_1) - \nabla \phi^{\#}_{\rho_0}(\xi_2)\|_{\mathbb{H}}}{\|\xi_1-\xi_2\|_{\mathbb{H}}} < \infty. 
\]

Let $X \in L^2_{\textnormal{loc}}([0,\infty) \times \mathbb{R}^d; \rho_0)$ be the unique solution to \eqref{L2_Lagrangian_Flow}. We further assume there exists an $\alpha \in (0,1]$ such that $\dot{X}(t,\cdot) \in C_{\textnormal{loc}}^{1,\alpha}([0,\infty); \mathbb{H})$. For each $T>0$, define
\begin{equation} 
L_{\alpha}(T,\ddot{X}) := \sup_{t\neq s \in [0,T]} \frac{\|\dot{X}_t - \dot{X}_s - (t-s)\ddot{X}_s\|_{\mathbb{H}}}{|t-s|^{\alpha}} < \infty. 
\label{characteristic_Lipschitz}  
\end{equation}
\end{assumption}

We begin by extending the trapezoidal quadrature rule to the Hilbert space $\mathbb{H}$.

\begin{lemma} \label{trapezoid_integral_rule}
Fix $T > 0$, $\rho_0 \in \mathcal{P}_{2}(\R^d)$, and suppose $u \in C^{1,\alpha}([0,T];\mathbb{H})$ for some $\alpha \in (0,1]$. Let $L$ be the smallest constant for which $u$ satisfies \eqref{time_taylor_expand}. Then, for any $0 \leq a \leq b \leq T$, we have
\begin{equation}
\left\| \frac{b-a}{2} \left( u(b,x) + u(a,x) \right) - \int_a^b u(s,x) \, ds \right\|_{\mathbb{H}} 
\leq 2 L |b-a|^{2+\alpha}.   
\label{trapezoid_rule_bound_integral}
\end{equation} 
\end{lemma}

\begin{proof}
We first observe that Jensen's inequality and Fubini's theorem imply for any $f$
\begin{equation} 
\left\| \int_a^b f(t,x)\,dt \right\|_{\mathbb{H}}^2 \leq (b-a)\int_a^b \|f(t,x)\|_{\mathbb{H}}^2\,dt. 
\label{integral_inequality_Jensen}
\end{equation}

We begin by decomposing the integrand in \eqref{trapezoid_rule_bound_integral}. Observe that the integrand decomposes as
\begin{align} \label{decompose_trapezoid}
\int_a^b \frac{1}{2}\left[u(s,x)-u(b,x)\right] + \frac{1}{2}\left[u(s,x)-u(a,x)\right]\,ds &= \frac{1}{2}\int_a^b (s-b)\dot{u}(b,x)\,ds + \frac{1}{2}\int_a^b (s-a)\dot{u}(a,x)\,ds + \mathcal{R}, 
\end{align}
where the remainder term is
\[
\mathcal{R} = \frac{1}{2}\underbrace{\int_a^b [u(s,x)-u(b,x)-(s-b)\dot{u}(b,x)]\,ds}_{\mathcal{R}_1} 
+ \frac{1}{2}\underbrace{\int_a^b [u(s,x)-u(a,x)-(s-a)\dot{u}(a,x)]\,ds}_{\mathcal{R}_2}.
\]

We first bound the $\mathbb{H}$ norm of the remainder term. Specifically, we will show that
\[
\|\mathcal{R}_i\|_{\mathbb{H}} \leq \frac{L}{\sqrt{2\alpha + 3}} (b - a)^{2 + \alpha}, \quad \text{for } i \in \{1, 2\}.
\]
Since the argument is analogous for $\mathcal{R}_2$, we only provide the proof for $\mathcal{R}_1$.
 Using \eqref{integral_inequality_Jensen} and Definition \ref{C11_X}, we have
\[
\|\mathcal{R}_1\|_{\mathbb{H}}^2 
\leq (b-a)\int_a^b \|u(s,x)-u(b,x)-(s-b)\dot{u}(b,x)\|_{\mathbb{H}}^2\,ds 
\leq L^2(b-a)\int_a^b|s-b|^{2(1+\alpha)}\,ds 
= \frac{L^2}{2\alpha+3}(b-a)^{2\alpha+4}.
\]
This gives the desired bound on $\mathcal{R}_1$. \smallskip

Next, we control the remaining terms in \eqref{decompose_trapezoid}. Notice that
\[
\int_a^b (s-b)\dot{u}(b,x)+(s-a)\dot{u}(a,x)\,ds = \frac{(a-b)^2}{2}(\dot{u}(a,x)-\dot{u}(b,x)).
\] 
Thus, 
\[
\left\|\int_a^b (s-b)\dot{u}(b,x)+(s-a)\dot{u}(a,x)\,ds\right\|_{\mathbb{H}} 
= \frac{(b-a)^2}{2}\|\dot{u}(a,x)-\dot{u}(b,x)\|_{\mathbb{H}}.
\]
Since $\dot{u}(t,x)$ is $\alpha$-Hölder continuous on $\mathbb{H}$ with Hölder constant $2L$, we have
\[
\left\|\int_a^b (s-b)\dot{u}(b,x)+(s-a)\dot{u}(a,x)\,ds\right\|_{\mathbb{H}} 
\leq L(b-a)^{2+\alpha}.
\]

Combining this with our bound on $\mathcal{R}$, we conclude from \eqref{decompose_trapezoid}:
\[
\left\|\int_a^b \frac{1}{2}\left[u(s,x)-u(b,x)\right] + \frac{1}{2}\left[u(s,x)-u(a,x)\right]\,ds\right\|_{\mathbb{H}} \leq 2L|b-a|^{2+\alpha}.
\]
\end{proof}

We now have the tools needed to establish the higher-order convergence of our scheme.

\begin{theorem}[Higher Order Convergence] \label{higher_order_convergence_theorem}
Suppose Assumption \ref{higher_order_assumption} holds and the time step satisfies $0 < \tau \leq \frac{1}{L(\phi)}$. Given a terminal time $T > 0$, let $n \in \mathbb{N}$ be such that $(n+1)\tau \leq T$. Then, the iterates $X_n^{\tau}$ of the Lagrangian trapezoid scheme (Definition \ref{trapezoid_def}) satisfy
\[
W_2(\rho_{n\tau},\rho_n^{\tau}) 
\leq \|X_n^{\tau} - X(n\tau)\|_{\mathbb{H}} 
\leq e^{2 L(\phi) T}\|X_0^{\tau} - \textnormal{Id}\|_{\mathbb{H}} 
+ 2\frac{L(T,\ddot{X})}{L(\phi)}(e^{2 L(\phi) T}-1)\,\tau^{1+\alpha},
\]where $X(t)$ is the unique solution to \eqref{L2_Lagrangian_Flow} in $L^2_{\textnormal{loc}}([0,\infty)\times\R^d;\rho_0)$, and we define $\rho_t := (X_t)_{\#}\rho_0$ and $\rho_j^{\tau} := (X_j^{\tau})_{\#}\rho_0$.

\end{theorem}

\begin{proof}
For each $n \in \mathbb{N}$, define the error iterates
\[
e_{n+1} := \|X_{n+1}^{\tau}-X((n+1)\tau)\|_{\mathbb{H}},
\]
along with the trapezoid approximation of $X((n+1)\tau)$:
\[
Y_{n+1}^{\tau} := X(n\tau) - \frac{\tau}{2}\left(\nabla \phi^{\#}_{\rho_0}(X(n\tau)) + \nabla \phi^{\#}_{\rho_0}(X((n+1)\tau))\right).
\] 
The triangle inequality yields
\[
e_{n+1} \leq \underbrace{\|Y_{n+1}^{\tau} - X((n+1)\tau)\|_{\mathbb{H}}}_{\mathcal{I}_{n+1}^{(1)}} 
+ \underbrace{\|Y_{n+1}^{\tau} - X^{\tau}_{n+1}\|_{\mathbb{H}}}_{\mathcal{I}_{n+1}^{(2)}}.
\]

To control $\mathcal{I}^{(1)}_{n+1}$, recall from \eqref{L2_Lagrangian_Flow} that for any $s < t$,
\begin{equation} 
X(t) = X(s) - \int_s^t \nabla \phi^{\#}_{\rho_0}(X(a))\,da, \quad \rho_0\text{-a.e.}
\label{integral_identity_of_X}
\end{equation}
Then by choosing $s=n \tau$ and $t=(n+1)\tau$, we see by Lemma \ref{trapezoid_integral_rule}  and Assumption \ref{higher_order_assumption} that
\[
\mathcal{I}^{(1)}_{n+1} \leq 2 L_{\alpha}(T,\ddot{X})\,\tau^{2+\alpha}.
\]

Next, to control $\mathcal{I}^{(2)}_{n+1}$, use \eqref{integral_identity_of_X} to write
\[
Y_{n+1}^{\tau} - X^{\tau}_{n+1} 
= (X(n\tau)-X_n^{\tau}) - \frac{\tau}{2}\left(\nabla \phi^{\#}_{\rho_0}(X(n\tau))-\nabla \phi^{\#}_{\rho_0}(X_n^{\tau})\right) 
- \frac{\tau}{2}\left(\nabla \phi^{\#}_{\rho_0}(X((n+1)\tau))-\nabla \phi^{\#}_{\rho_0}(X_{n+1}^{\tau})\right).
\]
Thus, the triangle inequality and Assumption \ref{higher_order_assumption} yield
\[
\mathcal{I}^{(2)}_{n+1} \leq e_n + \tau\frac{L(\phi)}{2}(e_n+e_{n+1}).
\]
Combining these estimates, we have
\[
e_{n+1} \leq \frac{2L_{\alpha}(T,\ddot{X})\tau^{2+\alpha}}{1-\frac{L(\phi)\tau}{2}} + \frac{ (1+\frac{L(\phi)\tau}{2})e_n}{1-\frac{L(\phi) \tau}{2}}.
\]

Since we assumed $\tau \leq \frac{1}{L(\phi)}$, it follows that $L(\phi)\tau \leq 1$. Hence, we have that $1-\frac{L(\phi)\tau}{2} \geq \frac{1}{2}$. Hence,
\[ e_{n+1} \leq 4L_{\alpha}(T,\ddot{X}) \tau^{2+\alpha} + \frac{ (1+\frac{L(\phi)\tau}{2})e_n}{1-\frac{L(\phi) \tau}{2}}.  \] Now by using the inequality $\frac{1+x}{1-x}\leq 1+4x \text{ for } x\in[0,\tfrac{1}{2}]$ with $x=L(\phi)\tau/2$ we obtain
\begin{equation}
e_{n+1}\leq 4L_{\alpha}(T,\ddot{X})\tau^{2+\alpha}+(1+2L(\phi)\tau)e_n. \label{iteration_bound}
\end{equation}

By iterating \eqref{iteration_bound}, we arrive at
\[
e_{n}\leq(1+2L(\phi)\tau)^n e_0 + 4L_{\alpha}(T,\ddot{X})\tau^{2+\alpha}\sum_{j=0}^{n-1}(1+2L(\phi)\tau)^j.
\]

Thus, for $(n+1)\tau\leq T$, we have
\[
e_n\leq e^{2L(\phi)T}e_0+2\frac{L_{\alpha}(T,\ddot{X})}{L(\phi)}(e^{2L(\phi)T}-1)\tau^{1+\alpha}.
\]
\end{proof}

In particular, if $\dot{X} \in C^{1,1}_{\textnormal{loc}}([0,T];\mathbb{H})$, we obtain an $O(\tau^2)$ convergence rate. \smallskip

\subsection{Examples of Functionals Satisfying Higher-Order Convergence}

Now we focus on deriving examples that satisfy our higher order convergence theorem.

% \begin{lemma} 
% Fix $\rho_0 \in \mathcal{P}_2(\mathbb{R}^d)$ and assume $\phi^{\#}_{\rho_0} \in C^{1,1}(L^2(\mathbb{R}^d;\rho_0);L^2(\mathbb{R}^d;\rho_0))$ with
% \[
% \sup_{\xi \in \mathbb{H}} ||\nabla \phi^{\#}_{\rho_0}(\xi)||_{\mathbb{H}} \leq C.
% \]
% Then the unique solution $X \in L^2_{\textnormal{loc}}([0,\infty) \times \mathbb{R}^d; \rho_0)$ satisfies $X \in C^{1,1}([0,\infty);\mathbb{H})$.
% \label{Lipschitz_Streamlines}
% \end{lemma}

% \begin{proof} 

% We first show that $X(t)$ is Lipschitz. From \eqref{characteristic_ODE}, for $\rho_0$-a.e. $x$,
% \[
% X(t+h) - X(t) = - \int_t^{t+h} \nabla \phi^{\#}_{\rho_0}(X(s)) \, ds.
% \]
% Applying Jensen's inequality,
% \[
% \|X_{t+h}-X_t\|^2_{\mathbb{H}} \leq h \int_t^{t+h} \|\nabla \phi^{\#}_{\rho_0}(X(s))\|^2_{\mathbb{H}} \, ds \leq C^2h^2.
% \]
% Thus, $X$ is Lipschitz. \smallskip

% To establish $C^{1,1}$ regularity, we use \eqref{characteristic_ODE} to see that
% \[
% \|\dot{X}(t+h)-\dot{X}(t)\|_{\mathbb{H}} = \|\nabla \phi^{\#}_{\rho_0}(X(t+h)) - \nabla \phi^{\#}_{\rho_0}(X(t))\|_{\mathbb{H}}.
% \]
% Since $\nabla \phi^{\#}_{\rho_0}$ is Lipschitz with constant $L_{\phi}$,
% \[
% \|\dot{X}(t+h)-\dot{X}(t)\|_{\mathbb{H}} \leq L_{\phi} \|X(t+h)-X(t)\|_{\mathbb{H}} \leq L_{\phi} C |h|.
% \]
% Thus, $X \in C^{1,1}([0,\infty);\mathbb{H})$.
% \end{proof}

\begin{theorem}  \label{C11_Velocity_Sufficient}
 Assume that $\rho_0$ is a non-atomic measure such that $\phi^{\#}_{\rho_0} \in C^{1,1}([0,\infty);\R^d)$ with $\nabla_{\mathbf{W}} \phi$ satisfying the growth conditions in Proposition \ref{differentiability_equivalence}. In addition, suppose there exists a constants $\alpha \in (0,1]$ and $C > 0$ such that for all $\rho \in \mathcal{P}_2(\R^d)$ and $x,y \in \R^d$
\[
|  \nabla_{\mathbf{W}} \phi(\rho,x)|  + |\nabla_x \nabla_{\mathbf{W}} \phi(\rho,x)|  
+ |\nabla^2_{\mathbf{W}} \phi(\rho,x,y)|    
\leq C,
\]
\[ [\nabla_{x} \nabla_{\mathbf{W}} \phi(\rho,\cdot)]_{C^{\alpha}(\R^d)}+[\nabla^2_{\mathbf{W}} \phi(\rho,\cdot,\cdot)] _{C^{\alpha}(\R^d \times \R^d)}  \leq C \]
Moreover, assume that the partial Wasserstein Hessian of $\phi$ satisfies the Hölder-condition:
\begin{equation} 
\sup_{\mu \neq \nu \in \mathcal{P}_2(\mathbb{R}^d)} 
\frac{\|\nabla^2_{\mathbf{W}} \phi(\mu,x,y) - \nabla^2_{\mathbf{W}} \phi(\nu,x,y) \|}{W^{\alpha}_1(\mu,\nu)} 
\leq C.  
\label{Lipschitz_W2_Hessian}  
\end{equation}
Then, $\dot{X} \in C^{1,\alpha}([0,\infty);\mathbb{H})$, and its Fréchet derivative is given by:
\[
\resizebox{1.0\textwidth}{!}{
$
\ddot{X}(t,x) =  
\bigl[\nabla_x \nabla_{\mathbf{W}} \phi\bigl((X(t))_{\#} \rho_0, X(t,x)\bigr) \bigr] \dot{X}(t,x)
+ \int_{\mathbb{R}^d} \bigl( \nabla^2_{\mathbf{W}} \phi\bigl((X(t))_{\#} \rho_0, X(t,x), X(t,z)\bigr) \bigr) \dot{X}(t,z) \, d\rho_0(z).
$
}
\]

\end{theorem}

\begin{proof} Because of our assumptions, we have that $\nabla \phi^{\#}_{\rho_0}(\xi)(x) = \nabla_{\mathbf{W}} \phi(\xi_{\#} \rho_0,\xi(x))$ for $\rho_0$-a.e. $x$. \smallskip

\textbf{Step 1: $X(t) \in C^{1,1}([0,\infty);\mathbb{H})$ and Lipschitz Regularity in Time} 
\smallskip

Because $\nabla_{\mathbf{W}} \phi(\rho,x)$ is uniformly bounded on $\spt(\rho_0)$, we immediately obtain that $\nabla \phi^{\#}_{\rho_0}(\xi)$ is uniformly bounded on $\mathbb{H}$. This implies from \eqref{L2_Lagrangian_Flow} that $X(t)$ is Lipschitz continuous on $\mathbb{H}$. Then as $\nabla \phi^{\#}_{\rho_0}(\xi)$ is Lipschitz on $\mathbb{H}$, \eqref{L2_Lagrangian_Flow}  then implies $X(t) \in C^{1,1}([0,\infty);\mathbb{H})$. \smallskip

Now using \eqref{L2_Lagrangian_Flow}, we have that for a.e. $x \in \text{spt}(\rho_0)$
\[ X(t+h,x) - X(t,x) = -\int_{t}^{t+h} \nabla_{\mathbf{W}} \phi( X(s)_{\#} \rho_0,X(s,x)) ds.   \] From our assumption that the Wasserstein gradient is uniformly bounded, we obtain for a.e. $x \in \text{spt}(\rho_0)$
\begin{equation} |X(t+h,x) -X(t,x)| \leq C|h| \label{Lipschitz_Space} \end{equation}

\textbf{Step 2 : $\dot{X}_t \in C^{1,\alpha}([0,\infty);\mathbb{H})$} \smallskip

For notational simplicity, we will let $F(\rho,x) := \nabla_{\mathbf{W}} \phi(\rho,x)$ and $\rho_t := (X_t)_{\#} \rho_0$.   \smallskip

 Using \eqref{L2_Lagrangian_Flow}, we obtain for $\rho_0$-a.e. $x$:
\[
\dot{X}(t,x)-\dot{X}(t+h,x) = F(\rho_t,X(t,x))- F(\rho_{t+h},X(t+h,x)).
\]
\[ = \underbrace{\left[ F(\rho_t,X(t,x)) - F(\rho_t,X(t+h,x)) \right]}_{(I)} + \underbrace{[ F(\rho_t,X(t+h,x))-F(\rho_{t+h},X(t+h,x)) ]}_{(II)} \]

\textbf{Bounding Term (I).} As $x\mapsto \nabla_x F(\rho,x)$ is Hölder continuous, we obtain from Taylor's Theorem
\[
(I) = -h\nabla_x F(\rho_t,X(t,x)) \cdot(\frac{X(t+h,x)-X(t,x)}{h}) + \tilde{\mathcal{R}}
\] where $|\tilde{\mathcal{R}}(t,x)| \leq C' |X(t+h,x)-X(t,x)|^{1+\alpha}$, which implies from \eqref{Lipschitz_Space} that $||\tilde{\mathcal{R}}(t,\cdot)||_{\mathbb{H}} \leq C'' |h|^{1+\alpha}$. \smallskip

Simplifying further, we have that
\[ (I) =-\left( \nabla_x F(\rho_t,X(t,x)) \right)  \dot{X}(t,x) + \mathcal{R}_1,\] where
\[ \mathcal{R}_1 = \tilde{\mathcal{R}}  -\underbrace{h\left( \nabla_x F(\rho_t,X(t,x) \right) \cdot \left( \frac{X(t+h,x)-X(t,x)}{h}  - \dot{X}(t,x)
 \right)}_{\tilde{\mathcal{R}}_2}.   \] From $X(t) \in C^{1,1}(\mathbb{H};\mathbb{H})$ and $\nabla_x F$ being uniformly bounded, we see that $||\tilde{\mathcal{R}}_2||_{\mathbb{H}} = O(|h|^2)$. In particular, we have that $||\mathcal{R}_1||_{\mathbb{H}} = O(|h|^{1+\alpha})$. \smallskip

\textbf{Bounding Term (II).} Applying \eqref{first_variation}, we obtain:
\[
(II) = -\int_0^1 \int_{\mathbb{R}^d} \frac{\delta F}{\delta \mu}\bigl( (1-s) \rho_t + s\rho_{t+h},X(t+h,x),z \bigr) d(\rho_{t+h} - \rho_t)(z) ds.
\]
Using Taylor's Theorem with $y \mapsto \nabla^2_{\mathbf{W}} \phi(\rho,x,y)$ is uniformly $C^{\alpha}$
\[
= -h\int_0^1 \int_{\mathbb{R}^d} \nabla_{\mathbf{W}} F( (1-s)\rho_t+s\rho_{t+h},X(t+h,x),X(t,z)) \cdot \frac{(X(t,z)-X(t+h,z))}{h} d\rho_0(z)ds + \mathcal{R}_2,
\] where $||\mathcal{R}_2(t,\cdot)||_{\mathbb{H}} = O(h^{1+\alpha})$. \smallskip

Because $X \in C^{1,1}(\mathbb{H};\mathbb{H})$ and the boundedness of $\nabla_{\mathbf{W}} F$, we obtain:
\[
(II) = h\int_0^1 \int_{\mathbb{R}^d} \nabla_{\mathbf{W}} F( (1-s)\rho_t + s\rho_{t+h},X(t+h,x),X(t,z)) \cdot \dot{X}(t,z) d\rho_0(z) ds + \mathcal{R}_3,
\]
where $\mathcal{R}_3 = \mathcal{R}_2 + \tilde{\mathcal{R}}_3$, where  $\|\tilde{\mathcal{R}}_3\|_{\mathbb{H}} = O(h^2)$. \smallskip

Now we use that $x \mapsto \nabla_{\mathbf{W}} F(\rho,x,y)$ is uniformly $C^{\alpha}$ and \eqref{Lipschitz_Space}, to see that
\[
(II) = h\int_0^1 \int_{\mathbb{R}^d} \nabla_{\mathbf{W}} F( (1-s)\rho_t + s\rho_{t+h},X(t,x),X(t,z)) \cdot \dot{X}(t,z) d\rho_0(z) ds + \mathcal{R}_4,
\] where  $\mathcal{R}_4 = \mathcal{R}_3 + \tilde{\mathcal{R}}_4$, where $||\tilde{\mathcal{R}}_4||_{\mathbb{H}}= O(h^{1+\alpha})$ \smallskip

Finally, using the Hölder condition on the measure position of $\nabla_{\mathbf{W}} F$ and the dual formula for the $ W_1 $ metric,
\[ ||\nabla_{\mathbf{W}} \phi( (1-s)\rho_t + s\rho_{t+h}, X(t,x),X(t,z)) - 
 \nabla_{\mathbf{W}} \phi( \rho_t, X(t,x),X(t,z))||_{\mathbb{H}} \] \[ \leq 
CW^{\alpha}_1(\rho_t, (1-s)\rho_t + s\rho_{t+h} ) 
 \leq Cs^{\alpha} W^{\alpha}_1(\rho_t,\rho_{t+h}) \leq \tilde{C}s^{\alpha} h^{\alpha},
\] where in the final inequality we used Lemma \ref{same_measure_bound}.
Thus, we obtain:
\[
(II) = h  \int_{\mathbb{R}^d} \nabla_{\mathbf{W}} F(\rho_t,X(t,x),X(t,z)) \cdot \dot{X}(t,z) d\rho_0(z) + \mathcal{R}_5,
\]
where $\|\mathcal{R}_5\|_{\mathbb{H}} = O(h^{1+\alpha})$, completing the proof.
\end{proof}

\begin{example} 
 Suppose $U(\rho)$ is the functional from Example \ref{C2_functionals}, and additionally assume that $V, W \in C^{2,1}_b(\mathbb{R}^d)$, $\chi \in C^{2,1}_c(\mathbb{R}^d)$, and $f \in C^{2,1}_b(\mathbb{R}^d)$. Then, $U(\rho)$ satisfies the assumptions of Theorem \ref{C11_Velocity_Sufficient} with parameter $\alpha=1$. \label{C2alpha_velocity}
\end{example}

\begin{proof}
By the gradient formulas in Example \ref{C2_functionals}, it suffices to verify \eqref{Lipschitz_W2_Hessian}. We now check this term by term for the partial Wasserstein Hessian formula in Example \ref{C2_functionals}. Throughout the proof, we set $\mu_\chi := \mu \star \chi$ and write $f \lesssim g$ if $f \le Cg$ for some constant $C>0$.  \smallskip

The difference in first term of the partial Wasserstein Hessian formula at the $(i,j)$th component is
\[ \int_{\R^d} f''(\mu_{\chi}(z)) \p_j \chi(z-x') \p_i \chi(x-z) d\mu(z) - \int_{\R^d} f''(\nu_{\chi}(z)) \p_j \chi(z-x') \p_i \chi(x-z) d\nu(z)  \]
\begin{align*}
&= \int_{\R^d}
\bigl[ f''(\mu_{\chi}(z)) - f''(\nu_{\chi}(z)) \bigr]
\, \p_j \chi(z-x') \, \p_i \chi(x-z) \, d\mu(z) \\
&\quad + \int_{\R^d}
f''(\nu_{\chi}(z))
\, \p_j \chi(z-x') \, \p_i \chi(x-z)
\, d(\mu-\nu)(z) 
=: (I) + (II).
\end{align*}

To control $(I)$, we use that $f''$ is Lipschitz, $\nabla \eta \in L^{\infty}$, the dual representation of the $W_1$ metric, and $\mu$ is a probability measure to obtain that
\[ |(I)| \lesssim \int_{\R^d} |\mu_{\chi}(z)-\nu_{\chi}(z)| d\mu(z) = \int_{\R^d} \left| \int_{\R^d}  \chi(z-x) d(\mu-\nu)(x) \right| d\mu(z) \lesssim \int_{\R^d} W_1(\mu,\nu) d\mu(z) = W_1(\mu,\nu). \]
Now we use that the integrand for $(II)$ is uniformly Lipschitz due to our assumptions to obtain from the dual representation of the $W_1$ metric that
\[ |(II)| \lesssim W_1(\mu,\nu). \] 

The difference in the second term of the partial Wasserstein Hessian at the $(i,j)$th component is $(III)$, which is defined as
\[\int_{\R^d} \left[ f''(\mu_{\chi}(x))-f''(\nu_{\chi}(x) \right] \p_j \chi(x-x') \p_i \chi(x-z) d\mu(z) + \int_{\R^d} f''(\nu_{\chi}(x)) \p_j \chi(x-x') \p_i \chi(x-z) d(\mu-\nu)(z) \] can be similarly bounded to show that
\[ |(III)| \lesssim W_1(\mu,\nu). \]

The difference in the third term of the partial Wasserstein Hessian at the $(i,j)$th component is
\[ (IV) :=(f'(\mu_{\chi}(x)) - f'(\nu_{\chi}(x)) + f'(\mu_{\chi}(x')) - f'(\nu_{\chi}(x'))) \p^2_{i,j} \chi(x-x').   \] Observe for any $z \in \R^d$ one has from $f'$ and $\chi$ being Lipschitz
 that \[  |f'(\mu_{\chi}(z)) - f'(\nu_{\chi}(z))| \lesssim \left| \int_{\R^d} \chi(z-x) d(\mu-\nu)(x) \right| \lesssim W_1(\mu,\nu).  \] Hence, one sees that
\[ |(IV)| \lesssim W_1(\mu,\nu).  \]

The difference in the fourth and final term of the partial Wasserstein Hessian at the $(i,j)$th component is
\[ (f''(\nu_{\chi}(x'))-f''(\mu_{\chi}(x'))) \cdot  \p_i \chi(x-x') (\nu \star \p_j \chi)(x') + f''(\mu_{\chi}(x')) \cdot \p_i \chi(x-x') \cdot ( (\nu-\mu) \star \p_j \chi)(x') = (V) + (VI)  \]

Now because $\nabla \chi \in L^{\infty}$ and $f''$ is Lipschitz, we have from the dual formula of the $W_1$ metric
\[ |(V)| \lesssim |\nu_{\chi}(x')-\mu_{\chi}(x')| = \left| \int_{\R^d} \chi(x'-z) d(\nu-\mu)(z) \right| \lesssim W_1(\mu,\nu). \] Finally, one has that from $f'',\nabla \chi \in L^{\infty}$ and $\p_j \chi$ is Lipschitz that from the dual formula of the $W_1$ metric.
\[ |(VI)|  \lesssim | (\nu-\mu) \star \p_j \chi(x')| = \left| \int_{\R^d} \p_j \chi(x'-z) d (\nu-\mu)(z)  \right| \lesssim W_1(\mu,\nu). \]

Therefore, one has that
\[ |\nabla_{\mathbf{W}} U(\mu,x,x')-\nabla_{\mathbf{W}} U(\nu,x,x')| \lesssim W_1(\mu,\nu), \] which is \eqref{Lipschitz_W2_Hessian}.
\end{proof}

\section{\texorpdfstring{$O(\tau)$}{O(tau)} convergence via Discrete EVI} \label{Discrete_EVI_Section}

\subsection{Discrete EVI}

In this section, we show that the maps $X^{\tau}_n$ for less regular functionals converge in $\mathbb{H}$ at a rate of $O(\tau)$ and characterize the limit in terms of an Evolution Variational Inequality (EVI). 
To achieve these goals, we utilize the \textit{discrete evolution variational inequality}. Our approach is inspired by the methods in \cite{ambrosio2007gradient}. Before deriving the discrete EVI, we first present a useful inequality.

\begin{lemma} \label{EVI_Helper} 
Under the same notation and assumptions as in Theorem \ref{trapezoid_well_defined}, we have for any $\xi \in \mathbb{H}$:
\small
\[
\langle \nabla \phi^{\#}_{\rho_0}(X_n^\tau), \xi - X_{n+1}^\tau \rangle_{\mathbb{H}}
\leq  \left( \phi^{\#}_{\rho_0}(\xi) - \phi^{\#}_{\rho_0}(X_{n+1}^\tau) \right) 
    + \frac{\tau}{2} \big( 
        \| \nabla \phi^{\#}_{\rho_0}(X_n^\tau) \|^2_{\mathbb{H}} 
        - \| \nabla \phi^{\#}_{\rho_0}(X_{n+1}^\tau) \|^2_{\mathbb{H}} 
    \big) - \frac{\lambda}{2} \|X^{\tau}_n-X^{\tau}_{n+1}\|^2_{\mathbb{H}} - \frac{\lambda}{2} \|\xi - X^{\tau}_n\|^2_{\mathbb{H}}.
\]
\normalsize
\end{lemma}

\begin{proof} 
We first decompose the inner product as follows:
\begin{equation}
\langle \nabla \phi^{\#}_{\rho_0}(X_n^\tau), \xi - X_{n+1}^\tau \rangle_{\mathbb{H}}
= \langle \nabla \phi^{\#}_{\rho_0}(X_n^\tau), \xi - X_{n}^\tau \rangle_{\mathbb{H}} + \langle \nabla \phi^{\#}_{\rho_0}(X_n^\tau), X_{n}^\tau - X_{n+1}^\tau \rangle_{\mathbb{H}} = (I) + (II). \label{split_inner_product}
\end{equation}

\noindent To bound $(I)$, we use the tangent line inequality:
\[
(I)  \leq \phi^{\#}_{\rho_0}(\xi) - \phi^{\#}_{\rho_0}(X^{\tau}_n) -  \frac{\lambda}{2} \|\xi - X^{\tau}_n\|^2_{\mathbb{H}}.
\] 
To bound $(II)$, we use the implicit equation in Theorem \ref{trapezoid_well_defined} to obtain:

\[
(II) = \frac{\tau}{2} \left( \|\nabla \phi^{\#}_{\rho_0}(X^{\tau}_n)\|^2_{\mathbb{H}} 
+ \langle\nabla \phi^{\#}_{\rho_0}(X^{\tau}_n),\nabla \phi^{\#}_{\rho_0}(X^{\tau}_{n+1}) \rangle_{\mathbb{H}} \right).
\]

\noindent The implicit equation in Theorem \ref{trapezoid_well_defined} also implies that
\[
\langle\nabla \phi^{\#}_{\rho_0}(X^{\tau}_n),\nabla \phi^{\#}_{\rho_0}(X^{\tau}_{n+1}) \rangle_{\mathbb{H}} = 
    - \|\nabla \phi^{\#}_{\rho_0}(X^{\tau}_{n+1})\|^2_{\mathbb{H}} 
 + \frac{2}{\tau}\langle 
    \nabla \phi^{\#}_{\rho_0}(X^{\tau}_{n+1}), X^{\tau}_n - X^{\tau}_{n+1} 
\rangle_{\mathbb{H}}.
\]

\noindent Using the tangent line inequality on the inner product term implies
\[
(II) \leq \frac{\tau}{2} \big( 
    \|\nabla \phi^{\#}_{\rho_0}(X^{\tau}_n)\|^2_{\mathbb{H}} 
    - \|\nabla \phi^{\#}_{\rho_0}(X^{\tau}_{n+1})\|^2_{\mathbb{H}} 
\big) + \phi^{\#}_{\rho_0}(X^{\tau}_n) - \phi^{\#}_{\rho_0}(X^{\tau}_{n+1})  
- \frac{\lambda}{2} ||X^{\tau}_n- X^{\tau}_{n+1}||^2_{\mathbb{H}}.
\]

\noindent Combining the bounds on $(I)$ and $(II)$ allows us to conclude.
\end{proof}

\begin{lemma}[Discrete EVI] \label{discrete_EVI}
Under Assumption \ref{assume_rough}, for any $\xi \in \mathbb{H}$ and $n \in \mathbb{N}$, we have:
\[
\frac{1}{2 \tau} \left( \|\xi - X^{\tau}_{n+1}\|^2_{\mathbb{H}} - \|\xi - X^{\tau}_n\|^2_{\mathbb{H}} \right) 
+ \frac{\lambda}{4} \left(\|\xi - X^{\tau}_n\|^2_{\mathbb{H}} + \|\xi - X^{\tau}_{n+1}\|^2_{\mathbb{H}} \right)
\]
\[
\leq \phi^{\#}_{\rho_0}(\xi) - \phi^{\#}_{\rho_0}(X^{\tau}_{n+1}) 
+ \frac{\tau}{4} \left( \|\nabla \phi^{\#}_{\rho_0}(X_n^{\tau})\|^2_{\mathbb{H}} - \|\nabla \phi^{\#}_{\rho_0}(X^{\tau}_{n+1})\|^2_{\mathbb{H}} \right) 
- \left( \frac{1}{2 \tau} + \frac{\lambda}{4} \right) \|X^{\tau}_n - X^{\tau}_{n+1}\|^2_{\mathbb{H}}.
\]
\end{lemma}

\begin{proof} 
Following the arguments of \cite{ambrosio2005gradient} for the JKO scheme, for $n \in \mathbb{N}$, we define the functional
\[
\Phi(\xi) := \frac{1}{2} \left( \phi^{\#}_{\rho_0}(\xi) + \langle \nabla \phi^{\#}_{\rho_0}(X_n^{\tau}), \xi \rangle_{\mathbb{H}} \right) + \frac{1}{2 \tau} \|\xi - X_n^{\tau}\|^2_{\mathbb{H}}.
\] 
By Assumption \ref{assume_rough}, $\Phi$ is $(\lambda/2 + 1/\tau)$-convex. Since $X^{\tau}_{n+1}$ minimizes $\Phi$ due to Definition \ref{trapezoid_def}, we obtain, for the linear interpolation $\gamma(t) := (1-t) X^{\tau}_{n+1} + t \xi$ that
\[
\Phi(X^{\tau}_{n+1}) \leq \Phi(\gamma(t)) \leq (1-t) \Phi(X^{\tau}_{n+1}) + t \Phi(\xi) - \frac{t(1-t)}{2} \left( \frac{\lambda}{2} + \frac{1}{\tau} \right) \|\xi - X^{\tau}_{n+1}\|^2_{\mathbb{H}} \quad \forall t \in(0,1].
\] 
In particular, this implies by dividing by $t$ and then letting $t \rightarrow 0$ gives
\begin{equation} \label{original_EVI}
0 \leq \Phi(\xi) - \Phi(X^{\tau}_{n+1}) - \frac{1}{2} \left( \frac{\lambda}{2} + \frac{1}{\tau} \right) \|\xi - X^{\tau}_{n+1}\|^2_{\mathbb{H}}. 
\end{equation}
% Letting $t \to 0$, we conclude that
% \begin{equation}
% \scalebox{0.875}{$
% \frac{1}{2 \tau} \left( \|\xi - X^{\tau}_{n+1}\|^2_{\mathbb{H}} - \|\xi - X^{\tau}_n\|^2_{\mathbb{H}} \right)  
% \leq \frac{1}{2} \left( \phi^{\#}_{\rho_0}(\xi) - \phi^{\#}_{\rho_0}(X^{\tau}_{n+1}) + \langle \nabla \phi^{\#}_{\rho_0}(X^{\tau}_n), \xi - X^{\tau}_{n+1} \rangle_{\mathbb{H}} \right) 
% - \frac{\lambda}{4} \|\xi - X^{\tau}_{n+1}\|^2_{\mathbb{H}}
% - \frac{1}{2 \tau} \|X^{\tau}_n - X^{\tau}_{n+1}\|^2_{\mathbb{H}}.
% $}
% \label{original_EVI}
% \end{equation}
Applying Lemma \ref{EVI_Helper} to \eqref{original_EVI} implies the claim.
\end{proof}

\subsection{Interpolations and Differential Inequalities}

We recall that gradient flows of $\lambda$-convex functionals are uniquely characterized by the Evolution Variational Inequality (EVI) (see \cite{ambrosio2005gradient}). Our approach to obtaining an $O(\tau)$ error rate is to consider linear interpolations of the discrete streamline and energy functional, and to show that the discrete EVI (Lemma~\ref{discrete_EVI}) implies that these interpolations satisfy an approximate EVI with a controlled error term. This combined with Grönwall's inequality will allow us to obtain our error rate. Our methods are similar to those in \cite{ambrosio2005gradient} with adaptations to handle our more complicated case.

\begin{definition}[Interpolations] \label{interpolations_usual}
Fix a time step size $\tau > 0$. Then we define
\[
\ell_{\tau}(t) := \frac{t - n\tau}{\tau}, \quad
\overline{X}^{\tau}_t := X^{\tau}_{n+1},  \text{ and }
\underline{X}^{\tau}_t := X^{\tau}_n, \quad \text{for } t \in [n\tau,(n+1)\tau).
\]
The linearly interpolated numerical Lagrangian flow and energy functional are given by
\[
\tilde{X}^{\tau}_t := (1 - \ell_{\tau}(t)) \underline{X}^{\tau}_t + \ell_{\tau}(t) \overline{X}^{\tau}_t, \quad
\phi_{\tau}(t) := (1 - \ell_{\tau}(t)) \phi^{\#}_{\rho_0}(\underline{X}^{\tau}_t) 
+ \ell_{\tau}(t) \phi^{\#}_{\rho_0}(\overline{X}^{\tau}_t).
\]
For $\xi_1,\xi_2 \in \mathbb{H}$, we define the metric
\[
d(\xi_1, \xi_2) := \|\xi_1 - \xi_2\|_{\mathbb{H}},
\]
and for any $\xi \in \mathbb{H}$, the linearly interpolated distance functional is
\[
d^2_{\tau}(t;\xi) := 
(1 - \ell_{\tau}(t)) d^2(\xi, \underline{X}^{\tau}_t) 
+ \ell_{\tau}(t) d^2(\xi, \overline{X}^{\tau}_t).
\] 
Define the constant interpolation gradient error term as
\[
G^{\tau}_t := \|\nabla \phi^{\#}_{\rho_0}(X^{\tau}_n)\|^2_{\mathbb{H}} - \|\nabla \phi^{\#}_{\rho_0}(X^{\tau}_{n+1})\|^2_{\mathbb{H}}, \quad t \in [n\tau,(n+1)\tau).
\] and the distance between iterates as
\[
\mathcal{D}_{t}^{\tau} := d(\underline{X}^{\tau}_t,\overline{X}^{\tau}_t).
\]
\end{definition}

Now with the notation of Definition \ref{interpolations_usual}, we see that the discrete EVI (Lemma \ref{discrete_EVI}) implies the following differential inequality for a.e. $t \geq 0$
\begin{equation} 
\frac{1}{2} \frac{d}{dt} d^2_{\tau}(t;\xi) + \frac{\lambda}{4} \left( d^2(\xi,\underline{X}^{\tau}_t ) + d^2(\xi,\overline{X}^{\tau}_t) ) \right) 
\leq \phi^{\#}_{\rho_0}(\xi) - \phi_{\tau}(t) +  \frac{\tau}{4} G^{\tau}_t - \frac{\lambda}{4} (\mathcal{D}_{t}^{\tau})^2 + \mathcal{R}_t^{\tau}. 
\label{EVI_Interpolated_Inequality}
\end{equation} 

\noindent Here, the remainder term $\mathcal{R}_t^{\tau}$ is given by:
\[
\mathcal{R}_t^{\tau} := \left( \phi_{\tau}(t) - \phi^{\#}_{\rho_0}(\overline{X}^{\tau}_t) - \frac{1}{2 \tau} (\mathcal{D}_{t}^{\tau})^2 \right).
\]
Expanding this term, we obtain:
\begin{equation}
\mathcal{R}_t^{\tau} = (1-\ell_{\tau}(t)) \left( \phi^{\#}_{\rho_0}(\underline{X}^{\tau}_t) - \phi^{\#}_{\rho_0}(\overline{X}^{\tau}_t) - \frac{1}{\tau} (\mathcal{D}_{t}^{\tau})^2 \right) 
- \frac{1}{\tau}\left( \ell_{\tau}(t)- \frac{1}{2} \right) (\mathcal{D}_{t}^{\tau})^2. \label{remainder_term}
\end{equation}

% \noindent The second term satisfies the integral property:
% \begin{equation} 
% \int_{n \tau}^{(n+1) \tau} (\ell_{\tau}(t) - \frac{1}{2}) \cdot (\mathcal{D}_{t}^{\tau})^2) dt 
% = d^2(X_{n}^{\tau},X_{n+1}^{\tau}) \int_{n \tau}^{(n+1) \tau} (\ell_{\tau}(t)-\frac{1}{2}) dt = 0.
% \label{integrates_zero} 
% \end{equation} 

When $\lambda \neq 0$, as in \cite{ambrosio2005gradient}, it is more convenient to modify the differential inequality \eqref{EVI_Interpolated_Inequality} so that the left-hand side contains terms of the form $d^2_{\tau}(t;\xi)$ instead of $d^2(\xi,\overline{X}^{\tau}_t)$ and $d^2(\xi,\underline{X}^{\tau}_t)$. To achieve this, we use the following inequality:

\begin{lemma} \label{replacement_bounds} With the notation of Definition \ref{interpolations_usual}, we have the bounds
\begin{equation} 
    d^2(\xi,\underline{X}_{t}^{\tau}) + d^2(\xi,\overline{X}^{\tau}_t) 
    \leq 2 d^2_{\tau}(t;\xi) + (\mathcal{D}_t^{\tau})^2 + 2(\mathcal{D}_t^{\tau})\cdot d_{\tau}(t;,\xi).  
    \label{upper_bound_on_d} 
\end{equation} 
and
\begin{equation}     \label{lower_bound_on_d}  
    2d^2_{\tau}(t;\xi) -2(\mathcal{D}_t^{\tau})d_{\tau}(t;\xi) \leq d^2(\xi,\underline{X}_t^{\tau}) + d^2(\xi,\overline{X}^{\tau}_t) + 3 (\mathcal{D}_t^{\tau})^2.\end{equation} 
\end{lemma}

\begin{proof} First observe that the concavity of the square root function implies that 
\begin{equation}
    (1-\ell_{\tau}(t))\cdot d(\underline{X}_{t}^{\tau}, \xi) 
    + \ell_{\tau}(t) \cdot d(\overline{X}^{\tau}_t, \xi) 
    \leq d_{\tau}(t;\xi). \label{concave_d_interpolated}
\end{equation} 
Using triangle inequality and \eqref{concave_d_interpolated} we see that
\begin{equation} d(\underline{X}_t^{\tau},\xi) = (1-\ell_{\tau}(t)) d(\underline{X}_t^{\tau},\xi) + \ell_{\tau}(t) d(\underline{X}_t^{\tau},\xi)  \leq d_{\tau}(t;\xi) + \ell_{\tau}(t) (\mathcal{D}_t^{\tau}). \label{d_upper_bound_1_2}  \end{equation} Similarly, we also have that
\begin{equation} d(\overline{X}_t^{\tau},\xi) \leq d_{\tau}(t;\xi) + (1-\ell_{\tau}(t)) (\mathcal{D}_t^{\tau}). \label{d_upper_bound_2_2} \end{equation} Squaring these two inequalities and adding them imply \eqref{upper_bound_on_d}. \smallskip

Now for \eqref{lower_bound_on_d}, observe that from the triangle inequality
\[ d^2_{\tau}(t;\xi) = d^2(\underline{X}_t^{\tau},\xi) + \ell_{\tau}(t) \left( d^2(\overline{X}_t^{\tau},\xi) - d^2(\underline{X}_t^{\tau},\xi)  \right) \leq d^2(\underline{X}_t^{\tau},\xi) + \ell_{\tau}(t)( (\mathcal{D}_t^{\tau})^2 + 2 d(\underline{X}_t^{\tau},\xi)(\mathcal{D}_t^{\tau})).  \] So now using \eqref{d_upper_bound_1_2}, we obtain that
\[d^2_{\tau}(t;\xi) \leq  d^2(\underline{X}_t^{\tau},\xi) + 2\ell_{\tau}(t)(  d_{\tau}(t;\xi)(\mathcal{D}_t^{\tau})) + \ell_{\tau}(t)(1+2\ell_{\tau}(t)) (\mathcal{D}_t^{\tau})^2.   \] Similarly, we have the bound
\[ d^2_{\tau}(t;\xi) \leq d^2(\overline{X}_t^{\tau},\xi) + 2(1-\ell_{\tau}(t))(d_{\tau}(t;\xi) (\mathcal{D}_t^{\tau})) + (1-\ell_{\tau}(t))(1+2(1-\ell_{\tau}(t))) (\mathcal{D}_t^{\tau})^2. \] Adding these two bounds and maximizing over $t$, gives \eqref{lower_bound_on_d}. \end{proof}

By combining \eqref{EVI_Interpolated_Inequality} with Lemma \ref{replacement_bounds}, we obtain the following differential inequality:

% \textcolor{blue}{Why does $|\lambda|$ show up here?  It seems a little strange to me that things get worse as the strong convexity parameter gets larger.  In particular, Lemma 4.2 is a better inequality for larger values of $\lambda$ which is lost here.  Is this the crude version which then gets refined in Section 5?}

% \textcolor{red}{I think this is because Lemma 4.3 (4.6) is probably not the best lower bound. All I did from (4.3) to Theorem 4.4 was use (4.6) to say $\lambda \geq 0$. $\frac{\lambda}{4}(d^2(\xi,\underline{X}_t^{\tau})+d^2(\xi,\overline{X}_t^{\tau})) \geq \lambda/2 d^2_{\tau}(t;\xi) -2\lambda d(\overline{X}_t^{\tau}, \underline{X}_t^{\tau}) - 3\lambda^2(\overline{X}^{\tau}_t,\underline{X}_t^{\tau}). $ This puts the $|\lambda|$ on the RHS.}

\begin{theorem}[Approximate EVI] \label{differential_inequality_thm}
Under Assumption \ref{assume_rough} and the notation introduced therein, as well as the notation from Definition \ref{interpolations_usual}, the following differential inequalities holds: 
\begin{equation}
   \frac{d}{dt} d^2_{\tau}(t;\xi) + \lambda d^2_{\tau}(t;\xi) 
    \leq  2(\phi^{\#}_{\rho_0}(\xi) - \phi_{\tau}(t)) + \frac{\tau}{2} G^{\tau}_t + 2\mathcal{R}_t^{\tau} + |\lambda|  (\mathcal{D}_t^{\tau}) \cdot((\mathcal{D}_t^{\tau}) + d_{\tau}(t;\xi)).
\end{equation}
\end{theorem}

\noindent Now we want to use Theorem \ref{differential_inequality_thm} to obtain a  differential inequality for $d^2(\underline{X}^{\tau}_t,\underline{X}^{\eta}_t)$ with two different time steps $\tau,\eta>0$. To do so, we follow \cite{ambrosio2005gradient} and introduce another interpolation function

\begin{definition} \label{further_interpolated_d}
Using the notation from Definition \ref{interpolations_usual}, let $\tau, \eta > 0$ be two time steps. We define the further interpolated distance functional as
\small{\[
d^2_{\tau,\eta}(t;s) := (1-\ell_{\tau}(t)) [  (1-\ell_{\eta}(s)) \cdot d^2(\underline{X}^{\eta}_s,\underline{X}^{\tau}_t) + \ell_{\eta}(s) \cdot d^2(\overline{X}^{\eta}_s,\underline{X}^{\tau}_t) ]
 +  \ell_{\tau}(t) \left( (1-\ell_{\eta}(s)) d^2(\underline{X}^{\eta}_s,\overline{X}^{\tau}_t) + \ell_{\eta}(s) d^2(\overline{X}^{\eta}_s,\overline{X}^{\tau}_t) \right) \]}
\[ = (1-\ell_{\eta}(s)) \cdot d^2_{\tau}(t; \underline{X}^{\eta}_s) + \ell_{\eta}(s) \cdot d^2_{\tau}(t;\overline{X}^{\eta}_s). \]

\end{definition}

\begin{theorem}[Differential Inequality]  \label{differential_inequality_different_time_steps}
Fix two time steps $\tau, \eta > 0$ satisfying
\[
\frac{\lambda}{2} + \min\left\{\frac{1}{\tau},\frac{1}{\eta}\right\} > 0.
\]
Using the assumptions and notation from Theorem \ref{differential_inequality_thm} and from Definition \ref{further_interpolated_d}, we have
\[
\frac{d}{dt} d^2_{\tau,\eta}(t;t) + 2 \lambda d^2_{\tau,\eta}(t;t) 
\leq \frac{\tau}{2} G^{\tau}_t + \frac{\eta}{2} G^{\eta}_t
+ 2(\mathcal{R}^{\tau}_t + \mathcal{R}^{\eta}_t) 
+ |\lambda| \left( (\mathcal{D}_{t}^{\tau})^2 + (\mathcal{D}_{t}^{\eta})^2 \right) 
+ |\lambda| (\mathcal{D}_t^{\tau} + \mathcal{D}_t^{\eta}) d_{\tau,\eta}(t;t).
\]
\end{theorem}

\begin{proof} 
First, we note the symmetry property of $d^2_{\tau,\eta}(t;s)$:
\begin{equation}  
    d^2_{\tau,\eta}(t;s) = d^2_{\eta,\tau}(s;t).
    \label{symmetry_interpolation} 
\end{equation}  Next, by Definition \ref{further_interpolated_d}, Theorem \ref{differential_inequality_thm}, and \eqref{concave_d_interpolated}, we obtain
\[
\frac{1}{2} \frac{\partial}{\partial t} d^2_{\tau,\eta}(t;s) + \lambda d^2_{\tau,\eta}(t;s)  
\leq  \phi_{\eta}(s) - \phi_{\tau}(t) + \frac{\tau}{4} G^{\tau}_t  + \mathcal{R}^{\tau}_t + \frac{|\lambda|}{2} \mathcal{D}_t^{\tau} \left( \mathcal{D}_t^{\tau} + d_{\tau,\eta}(t;s) \right).
\] By symmetry \eqref{symmetry_interpolation}, we have $\frac{\partial}{\partial s} d^2_{\tau,\eta}(t;s) = \frac{\partial}{\partial t} d^2_{\eta,\tau}(s;t)$, which gives
\[
\frac{1}{2} \frac{\partial}{\partial s} d^2_{\tau,\eta}(t;s) + \lambda d^2_{\tau,\eta}(t;s) 
\leq \phi_{\tau}(t) - \phi_{\eta}(s) + \frac{\eta}{4} G^{\eta}_s + \mathcal{R}^{\eta}_s + \frac{|\lambda|}{2} \mathcal{D}_s^{\eta}\left( \mathcal{D}_s^{\eta} + d_{\tau,\eta}(t;s) \right).
\] Summing these two inequalities at $s=t$ implies the desired bound.
\end{proof}

\subsection{ $O(\tau + \eta)$ Convergence of $d_{\tau,\eta}(t;t)$ for $\lambda \geq 0$}
The next step for obtaining convergence rates is to obtain bounds for the remainder terms $D^{\tau}_t$ and $\mathcal{R}^{\tau}_t$ in Theorem \ref{differential_inequality_different_time_steps}.

\begin{lemma}[Bounds on $\mathcal{R}^{\tau}_t$] \label{R_bounds}
Under Assumption \ref{assume_rough} and the notation of Theorem \ref{differential_inequality_thm}, the following bound holds:
\[
\mathcal{R}_t^{\tau} \leq \left( 1-\ell_{\tau}(t) \right) \left( \frac{\tau}{4} G^{\tau}_t - \frac{\lambda}{2} (\mathcal{D}_t^{\tau})^2 \right) - \frac{1}{\tau}\left( \ell_{\tau}(t) - \frac{1}{2} \right) (\mathcal{D}_t^{\tau})^2.
\]
\end{lemma}

\begin{proof}
By \eqref{remainder_term}, it suffices to show that for any $n \in \mathbb{N}$,
\[
\phi^{\#}_{\rho_0}(X^{\tau}_n) - \phi^{\#}_{\rho_0}(X^{\tau}_{n+1}) - \frac{1}{\tau} \|X^{\tau}_{n+1} - X^{\tau}_n\|^2_{\mathbb{H}} 
\leq \frac{\tau}{4} G^{\tau}_{n \tau} - \frac{\lambda}{2} \|X^{\tau}_n - X^{\tau}_{n+1}\|^2_{\mathbb{H}}.
\]

\noindent Because $\phi^{\#}_{\rho_0}$ is $\lambda$-convex on $\mathbb{H}$, the tangent line inequality implies
\[
\phi^{\#}_{\rho_0}(X^{\tau}_n) - \phi^{\#}_{\rho_0}(X^{\tau}_{n+1}) - \frac{1}{\tau} \|X^{\tau}_{n+1} - X^{\tau}_n\|^2_{\mathbb{H}} 
\leq \langle \nabla \phi^{\#}_{\rho_0}(X^{\tau}_n), X^{\tau}_n - X^{\tau}_{n+1} \rangle_{\mathbb{H}} 
- \frac{1}{\tau} \|X^{\tau}_{n+1} - X^{\tau}_n\|^2_{\mathbb{H}} - \frac{\lambda}{2} \|X^{\tau}_n - X^{\tau}_{n+1}\|^2_{\mathbb{H}}.
\]

\noindent Using \eqref{trapezoid_implicit}, we obtain
\[
\langle \nabla \phi^{\#}_{\rho_0}(X^{\tau}_n), X^{\tau}_n - X^{\tau}_{n+1} \rangle_{\mathbb{H}} - \frac{1}{\tau} \|X^{\tau}_{n+1} - X^{\tau}_n\|^2_{\mathbb{H}}  
= \frac{\tau}{4} G^{\tau}_{n \tau}.
\]
This completes the proof.
\end{proof}

Now we have enough to obtain an error estimate on $d_{\tau,\eta}(t;t)$ when $\lambda \geq 0$.
\begin{corollary}[Error Bound on $d_{\tau,\eta}(t;t)$ for $\lambda \geq 0$] \label{positive_d_convg}
Under the assumptions and notation of Theorem \ref{differential_inequality_different_time_steps}, the following bound holds for $\lambda \geq 0$:
\[
d^2_{\tau,\eta}(t;t) \leq \|X_0^{\tau} - X_0^{\eta}\|^2_{\mathbb{H}} + \frac{7}{4} \left( \tau^2 \|\nabla \phi^{\#}_{\rho_0}(X_0^{\tau})\|^2_{\mathbb{H}}  + \eta^2 \|\nabla \phi^{\#}_{\rho_0}(X_0^{\eta})\|^2_{\mathbb{H}} \right).
\]
\end{corollary}

\begin{proof} Because $\phi$ is $\lambda$-convex with $\lambda \geq 0$, we can use Theorem \ref{differential_inequality_different_time_steps} with $\lambda = 0$ to obtain:
\begin{equation} 
\frac{1}{2} \frac{d}{dt} d^2_{\tau,\eta}(t;\xi) \leq \frac{\tau}{4} (G^{\tau}_t+G^{\eta}_t) + \mathcal{R}^{\tau}_t + \mathcal{R}^{\eta}_t.   
\label{0_bound} 
\end{equation}

For any $t \geq 0$, let $N \in \mathbb{N}$ be the smallest integer such that $t \leq N \tau$.  
Since $\lambda \geq 0$, it follows from \eqref{gradient_bounds} that $G^{\tau}_t \geq 0$.  Hence, by integrating in time, we obtain:
\begin{equation} 
\int_0^t G^{\tau}_s ds \leq \int_0^{N \tau} G^{\tau}_s ds = \tau \sum_{j=0}^{N} \left( \|\nabla \phi^{\#}_{\rho_0}(X^{\tau}_j)\|_{\mathbb{H}}^2 - \|\nabla \phi^{\#}_{\rho_0}(X^{\tau}_{j+1})\|_{\mathbb{H}}^2 \right) \leq \tau \|\nabla \phi^{\#}_{\rho_0}(X_0^{\tau})\|^2_{\mathbb{H}}.  
\label{integral_D_s_error} 
\end{equation} 
Similarly, 
\begin{equation} 
\int_0^t G^{\eta}_s ds \leq \eta \|\nabla \phi^{\#}_{\rho_0}(X_0^{\eta})\|^2_{\mathbb{H}}.
\end{equation}

Next, we handle the term $\mathcal{R}^{\tau}_t$. Observe that for any $j \in \mathbb{N}$:
\[
\int_{j \tau}^{(j+1)\tau} \left(\ell_{\tau}(s)-\frac{1}{2}\right) ds = 0, 
\quad \int_{j \tau}^{(j+1) \tau} (1-\ell_{\tau}(s)) ds = \frac{\tau}{2}.
\]
Applying Lemma \ref{R_bounds} with $\lambda \geq 0$ gives:
\[
\int_0^t \mathcal{R}^{\tau}_s ds \leq \frac{\tau^2}{8} \sum_{j=0}^{N} \left( \|\nabla \phi^{\#}_{\rho_0}(X_j^{\tau})\|^2_{\mathbb{H}} - \|\nabla \phi^{\#}_{\rho_0}(X_{j+1}^{\tau})\|^2_{\mathbb{H}} \right) -  \frac{\|\overline{X}^{\tau}_t-\underline{X}^{\tau}_t\|_{\mathbb{H}}^2}{\tau}  \int_{(N-1)\tau}^t \left(\ell_{\tau}(s)-\frac{1}{2}\right) ds.
\]
\[
 \leq \frac{\tau^2}{8} \|\nabla \phi^{\#}_{\rho_0}(X^{\tau}_{0})\|^2_{\mathbb{H}} + \frac{1}{2} \|\overline{X}^{\tau}_t-\underline{X}^{\tau}_t\|^2_{\mathbb{H}}.
\]
By applying Lemma \ref{classical_stability_lemma} along with the above display, we see that 
\begin{equation} 
\int_0^t \mathcal{R}^{\tau}_s ds \leq \frac{5\tau^2}{8} \|\nabla \phi^{\#}_{\rho_0}(X^{\tau}_{0})\|^2_{\mathbb{H}}.
\label{integral_R_bounds}   
\end{equation} 
Applying the same argument for $\eta$, we obtain:
\[
\int_0^t \mathcal{R}^{\eta}_s ds \leq \frac{5\eta^2}{8} \|\nabla \phi^{\#}_{\rho_0}(X^{\eta}_{0})\|^2_{\mathbb{H}}.
\]

Finally, integrating \eqref{0_bound} in time and applying the estimates \eqref{integral_D_s_error} and \eqref{integral_R_bounds}, we conclude.
\end{proof}

\subsection{ $O(\tau + \eta)$ Convergence of $d_{\tau,\eta}(t;t)$ for $\lambda < 0$ }

We now derive convergence rates for $\lambda < 0$. Since $\lambda_{\tau} \leq \lambda$ (see Assumption \ref{assume_rough}), any $\lambda$-convex function $\phi$ is also $\lambda_{\tau}$-convex. Applying Theorem \ref{differential_inequality_different_time_steps} with $\lambda_{\tau}$ and using Gronwall’s Inequality \cite[Lemma 4.1.8]{ambrosio2005gradient}, we obtain the bound  
for any $T > 0$:  
\begin{equation}
    d_{\tau,\eta}(T;T) \leq e^{- \lambda_{\tau} T} \left( d^2_{\tau,\eta}(0;0) + \sup_{t \in [0,T]} \int_0^T e^{2 \lambda_{\tau} s} a(s) ds \right)^{1/2} 
    + e^{- \lambda_{\tau} T}\int_0^T e^{\lambda_{\tau} s} b(s) ds.
    \label{gronwall_bound}
\end{equation}
Here, the functions $a(t)$ and $b(t)$ are given by  
\[
a(t) :=  \frac{\tau}{2}  G^{\tau}_t + \frac{\eta}{2} G^{\eta}_t  + 2( \mathcal{R}_t^{\tau} + \mathcal{R}_t^{\eta}) + |\lambda_{\tau}| \big( (\mathcal{D}_t^{\tau})^2 + (\mathcal{D}_t^{\eta})^2 \big),
\quad
b(t) := |\lambda_{\tau}| \cdot(\mathcal{D}_t^{\tau} + \mathcal{D}_t^{\eta}).
\]

To estimate the integrals in \eqref{gronwall_bound}, we will estimate each individual term in $a(t)$.

\begin{lemma} \label{estimates_lambda_negative}
Under the notation and assumptions of Theorem \ref{differential_inequality_different_time_steps}, and with the additional assumptions that $\lambda < 0$, we have the following estimates for any $T \geq 0$:
\begin{equation}
\begin{aligned}
    \int_0^T e^{2 \lambda_{\tau} t} (\mathcal{D}_t^{\tau})^2 \,dt 
    &\leq \tau^2 (T+\tau) e^{-2 \lambda_{\tau} \tau} \|\nabla \phi^{\#}_{\rho_0}(X_0^{\tau})\|^2_{\mathbb{H}},  
    \label{distance_squared_bound} \\
    \int_0^T e^{2 \lambda_{\tau} t} G_t^{\tau} \,dt 
    &\leq \tau \|\nabla \phi^{\#}_{\rho_0}(X_0^{\tau})\|^2_{\mathbb{H}}, \\
    \int_0^T e^{2 \lambda_{\tau} t} \mathcal{R}^{\tau}_t \,dt 
    &\leq \tau^2\left( \frac{1}{8} + e^{-2 \lambda_{\tau} \tau} \left( \frac{1}{2} - \frac{2}{3} \lambda_{\tau}(T+\tau) \right) \right) 
    \|\nabla \phi^{\#}_{\rho_0}(X^{\tau}_0)\|^2_{\mathbb{H}}, \\ 
    \int_0^T e^{\lambda_{\tau} t} (\mathcal{D}_t^{\tau}) \,dt 
    &\leq \tau (T+\tau) e^{-\lambda_{\tau} \tau} \|\nabla \phi^{\#}_{\rho_0}(X_0^{\tau})\|_{\mathbb{H}}.
\end{aligned}
\end{equation}

\end{lemma}

\begin{proof}

Choose the smallest integer $N$ such that $T \leq N \tau$. \smallskip 

\noindent \textbf{Step 1: Bounding } $\displaystyle \int_0^T e^{2 \lambda_{\tau} t} (\mathcal{D}_t^{\tau})^2\,dt$.

We first observe that
\begin{equation}
\int_0^T e^{2\lambda_{\tau} t} d^2(\underline{X}_t^{\tau},\overline{X}^{\tau}_t)\,dt 
\leq \int_0^{N \tau} e^{2\lambda_{\tau} t} d^2(\underline{X}_t^{\tau},\overline{X}^{\tau}_t)\,dt 
= \underbrace{\frac{e^{2\lambda_{\tau} \tau} - 1}{2 \lambda_{\tau}}}_{\geq 0} \sum_{j=0}^{N-1} e^{2 \lambda_{\tau} j \tau} d^2(X^{\tau}_j,X^{\tau}_{j+1}).
\label{first_integral_boundss}
\end{equation} 

\noindent By applying \eqref{lipschitz_bound}, we obtain
\begin{equation}
d^2(X^{\tau}_j,X^{\tau}_{j+1}) \leq e^{-2\lambda_{\tau} (j+1) \tau} \tau^2 \|\nabla \phi^{\#}_{\rho_0}(X_0^{\tau})\|^2_{\mathbb{H}}. \label{distance_bounds123}
\end{equation}
Substituting this into \eqref{first_integral_boundss} and using the inequalities $(e^{2 \lambda_{\tau} \tau} - 1) \geq 2\lambda_{\tau} \tau$ and $\lambda_{\tau} \leq 0$, we obtain the desired bound. \smallskip

\noindent \textbf{Step 2: Bounding $\int_0^T e^{2 \lambda_{\tau}t} G^{\tau}_t dt$}  \smallskip

Because $G^{\tau}_t$ is piecewise constant on $[n \tau,(n+1)\tau)$, we obtain:
\[
\int_0^T e^{2 \lambda_{\tau}t}  G^{\tau}_s ds \leq \max \left\{ \int_0^{N \tau}  e^{2 \lambda_{\tau}t}  G^{\tau}_s ds , \int_0^{(N-1)\tau}  e^{2 \lambda_{\tau}t}  G^{\tau}_s ds  \right\} 
\]  
\[
= \max \left\{ \sum_{j=0}^{N} G^{\tau}_{j \tau} \int_{j \tau}^{(j+1)\tau} e^{2 \lambda_{\tau}t}  ds, \sum_{j=0}^{N-1} G^{\tau}_{j \tau} \int_{j \tau}^{(j+1)\tau} e^{2 \lambda_{\tau}t}  ds  \right\} =  \frac{e^{2 \lambda_{\tau} \tau} - 1}{2 \lambda_{\tau}}  \left( \sum_{j=0}^{N-1} G_{j \tau}^{\tau}  e^{2 \lambda_{\tau} j \tau} + e^{2 \lambda_{\tau} N\tau}(G^{\tau}_{N\tau})_{+} \right).  
\] Where $(x)_+ := \max\{x,0\}$. To simplify our notations for this proof, we define for $n \in \N$
\begin{equation}
\alpha_n := ||\nabla \phi^{\#}_{\rho_0}(X_n^{\tau})||^2_{\mathbb{H}}, \label{alpha_def}
\end{equation}
so that from Definition \ref{interpolations_usual} $G^{\tau}_{j \tau} = \alpha_j - \alpha_{j+1}$. Since $\lambda_{\tau} < 0$, we obtain for any $m \in \N$
\[
\sum_{j=0}^{m} G^{\tau}_{j \tau} e^{2 j \lambda_{\tau} \tau} = \sum_{j=0}^{m} (\alpha_j - \alpha_{j+1}) e^{2 j \lambda_{\tau} \tau} = \alpha_0 + \sum_{j=1}^{m} \underbrace{(e^{2 j \lambda_{\tau} \tau} - e^{2 (j-1) \lambda_{\tau} \tau})}_{\leq 0} \alpha_j - e^{2m \lambda_{\tau} \tau} \alpha_{m+1} \leq \alpha_0.
\]  
Thus, we conclude from $(e^{2\lambda_{\tau} \tau} - 1)/(2 \lambda_{\tau}) >0$ that we have the bound
\[ \frac{e^{2 \lambda_{\tau} \tau} - 1}{2 \lambda_{\tau}}  \left( \sum_{j=0}^{N-1} G_{j \tau}^{\tau}  e^{2 \lambda_{\tau} j \tau} + e^{2 \lambda_{\tau} N\tau}(G^{\tau}_{N\tau})_{+} \right) \leq \frac{e^{2 \lambda_{\tau} \tau} - 1}{2 \lambda_{\tau}} ||\nabla \phi^{\#}_{\rho_0}(X_0^{\tau})||^2_{\mathbb{H}} \leq \tau ||\nabla \phi^{\#}_{\rho_0}(X_0^{\tau})||^2_{\mathbb{H}} .
\]

\noindent \textbf{Step 3: Bounding $\int_0^t e^{2 \lambda_{\tau} t} \mathcal{R}_t^{\tau} dt$}

By Lemma \ref{R_bounds}, we have that
\[ \int_0^t e^{2 \lambda_{\tau} t} \mathcal{R}^{\tau}_t dt \leq \frac{\tau}{4}\int_0^t e^{2 \lambda_{\tau} t}  (1-\ell_{\tau}(t)) \cdot G^{\tau}_t dt + \int_0^t e^{2 \lambda_{\tau} t}\left( -\frac{\lambda_{\tau}}{2}  (1-\ell_{\tau}(t)) - \frac{1}{\tau}(\ell_{\tau}(t)-\frac{1}{2}) \right)  (\mathcal{D}_{t}^{\tau})^2) dt.  \]
\[ = (I) + (II). \] We first focus on bounding $(II)$. First observe that because $\lambda_{\tau} < 0$ and $j \in \N$,
\[ \int_{j \tau}^{(j+1)\tau}  \frac{e^{2 \lambda_{\tau} t}}{\tau}(\frac{1}{2}-\ell_{\tau}(t)) dt = - \frac{ (\lambda_{\tau} \tau + e^{2 \lambda_{\tau} \tau}(\lambda_{\tau} \tau - 1) + 1)e^{2 \lambda_{\tau} j \tau} }{4\lambda_{\tau}^2 \tau^2} \leq -\frac{\lambda_{\tau} \tau}{6} e^{2 \lambda_{\tau} j \tau}.  \]  In the above inequality, we used $-(x+e^{2x}(x-1)+1) \leq -2x^3/3$. Hence, we have from \eqref{distance_bounds123}
\[ \int_0^{(N-1)\tau}  \frac{e^{2 \lambda_{\tau} t}}{\tau} (\frac{1}{2} - \ell_{\tau}(t)) \cdot d^2(\overline{X}^{\tau}_t,\underline{X}_t^{\tau}) dt \leq -\frac{\lambda_{\tau} \tau}{6} \sum_{j=0}^{N-1} e^{2 \lambda_{\tau} j \tau} \cdot d^2(X^{\tau}_j,X^{\tau}_{j+1}) \leq -\frac{\lambda_{\tau} \tau^2 }{6}   e^{-2 \lambda_{\tau} \tau} \cdot T \cdot ||\nabla \phi^{\#}_{\rho_0}(X^{\tau}_0)||^2_{\mathbb{H}}. \] We also have from \eqref{lipschitz_bound}
\[ \int_{(N-1) \tau}^t \frac{e^{2 \lambda_{\tau} t}}{\tau} (\frac{1}{2} - \ell_{\tau}(t)) \cdot (\mathcal{D}_t^{\tau})^2 dt \leq \frac{1}{2}  e^{2 \lambda_{\tau}(N-1) \tau}d^2(X^{\tau}_{N-1},X^{\tau}_{N}) \leq  \frac{\tau^2}{2} e^{-2 \lambda_{\tau} \tau} \cdot ||\nabla \phi^{\#}_{\rho_0}(X_0^{\tau})||^2_{\mathbb{H}} .     \]

Now for the first integral term in $(II)$, we have from Step 1
\[ \int_0^t e^{2\lambda_{\tau} t} \left( -\frac{\lambda_{\tau}}{2}(1-\ell_{\tau}(t)) \right) d^2(\overline{X}_{t}^{\tau},\underline{X}_t^{\tau}) \leq -\frac{\lambda_{\tau}}{2} \int_0^t e^{2\lambda_{\tau} t} d^2(\underline{X}_t^{\tau},\overline{X}^{\tau}_t) dt \leq -\frac{\lambda_{\tau}}{2} \tau^2(T+\tau)e^{-2 \lambda_{\tau} \tau} ||\nabla \phi^{\#}_{\rho_0}(X_0^{\tau})||^2_{\mathbb{H}}, \] where the final inequality is due to our Step 1 bound. \smallskip

Combining these bounds gives us
\[ (II) \leq  \tau^2e^{-2\lambda_{\tau} \tau} \left( \frac{1}{2} -  \frac{2}{3} \lambda_{\tau}(T+\tau)   \right) \cdot||\nabla \phi^{\#}_{\rho_0}(X^{\tau}_0)||^2_{\mathbb{H}}. \] 

Now we focus on $(I)$. First recall the notation $\alpha_n$ from \eqref{alpha_def}. By using that $G^{\tau}_{t}$ is piece-wise constant on $[n\tau,(n+1)\tau)$ and $(1-\ell_{\tau}(t)) \geq 0$ the arguments of Step 1 imply
\begin{equation}
(I) \leq \frac{\tau}{4} \left( \sum_{j=0}^{N-1} (\alpha_n-\alpha_{n-1}) \cdot \int_{j \tau}^{(j+1)\tau}  e^{2 \lambda_{\tau} s} (1-\ell_{\tau}(s)) ds + (G^{\tau}_{N \tau})_{+} \cdot \int_{(N-1)\tau}^{N \tau} e^{2 \lambda_{\tau} s} (1-\ell_{\tau}(s)) ds    \right). \label{bound_on_I_1}
\end{equation} By integrating the term $e^{2 \lambda_{\tau} s} (1-\ell_{\tau}(s))$, we see that for any $m \in \N$
\[ \sum_{j=1}^{m} (\alpha_n-\alpha_{n-1}) \cdot \int_{j \tau}^{(j+1)\tau}  e^{2 \lambda_{\tau} s} (1-\ell_{\tau}(s)) ds = 
\underbrace{\frac{ e^{2 \lambda_{\tau} \tau} - (1+2\lambda_{\tau} \tau)  }{4 \lambda_{\tau}^2 \tau}}_{\geq 0} \sum_{j=1}^{m} (\alpha_j-\alpha_{j-1}) e^{2 \lambda_{\tau} j \tau}.  \] By a similar argument as in summation by parts step of Step 2, we have that
\[ \frac{ e^{2 \lambda_{\tau} \tau} - (1+2\lambda_{\tau} \tau)  }{4 \lambda_{\tau}^2 \tau} \sum_{j=1}^{m} (\alpha_j-\alpha_{j-1}) e^{2 \lambda_{\tau} j \tau} \leq \frac{ e^{2 \lambda_{\tau} \tau} - (1+2\lambda_{\tau} \tau)  }{4 \lambda_{\tau}^2 \tau}  \cdot ||\nabla \phi^{\#}_{\rho_0}(X_0^{\tau})||^2_{\mathbb{H}}.  \]
Because $\lambda_{\tau} \tau < 0$, we have that $e^{2\lambda_{\tau} \tau} - (1+2\lambda_{\tau}  \tau) \leq 2\lambda_{\tau}^2 \tau^2$, so we have for any $m \in \N$
\[
\sum_{j=0}^{m} (\alpha_j-\alpha_{j+1}) \int_{j \tau}^{(j+1)\tau} e^{2 \lambda s} (1-\ell_{\tau}(s)) ds \leq \frac{\tau}{2} \cdot ||\nabla_{\mathbf{W}} \phi(\rho^{\tau}_0,X^{\tau}_0)||^2_{\mathbb{H}}.
\]  Therefore, we obtain from \eqref{bound_on_I_1} that
\[ (I) \leq \frac{\tau^2}{8} \cdot ||\nabla_{\mathbf{W}} \phi(\rho^{\tau}_0,X^{\tau}_0)||^2_{\mathbb{H}} .  \] Therefore, by combining our bounds on $(I)$ and $(II)$, we conclude the third step. \smallskip

\noindent \textbf{Step 4: Bounding $\int_0^t e^{\lambda_{\tau} t} (\mathcal{D}_t^{\tau}) dt$}
The desired bound follows from Step 1 and Holder's inequality.\end{proof}
 Now by Lemma \ref{estimates_lambda_negative} and Theorem \ref{differential_inequality_different_time_steps} along with \eqref{gronwall_bound}, we obtain the following Theorem.

\begin{theorem}[Bounds on $d_{\tau,\eta}(t;t)$ for $\lambda < 0$]  
Under the assumptions and notation of Theorem \ref{differential_inequality_different_time_steps}, if $\lambda < 0$, then we have for any $T \geq 0$
\[
d_{\tau,\eta}(T;T) \leq e^{-\lambda_{\tau} T} d(X_0^{\tau},X_0^{\eta}) 
+ e^{-\lambda_{\tau} T} \Big( \tau K(\lambda_{\tau},T,\tau) \|\nabla \phi^{\#}_{\rho_0}(X^{\tau}_0)\|_{\mathbb{H}} 
+ \eta K(\lambda,T,\eta) \|\nabla \phi^{\#}_{\rho_0}(X^{\eta}_0)\|_{\mathbb{H}} \Big),
\]
where for any $\tau' > 0$,
\begin{equation}
K(\lambda_{\tau},T,\tau') = \sqrt{ \frac{3}{4} + (1-\frac{7}{3}\lambda_{\tau'}(T+\tau')) e^{-2\lambda_{\tau} \tau'} } + |\lambda_{\tau}| (T+\tau')e^{-\lambda_{\tau} \tau'}. \label{K_constant}
\end{equation}
\label{neg_d_convg}
\end{theorem}  
 
    \subsection{$O(\tau)$ Convergence of $\underline{X}_t^{\tau}$ and Characterization of the Limit}

We now establish convergence:

\begin{theorem}[$O(\tau)$ convergence rate]  \label{linear_convergence_rate}
Under the assumptions and notation of Theorem \ref{differential_inequality_different_time_steps}, for any $T > 0$, the sequence $\underline{X}^{\tau}_t$ converges uniformly on $[0,T]$ in $\mathbb{H}$ to a limiting locally Lipschitz curve $X(t) : [0,\infty) \to \mathbb{H}$.  

\smallskip

\noindent Furthermore, defining $\rho_t := (X(t))_{\#} \rho_0$ and $\rho_t^{\tau} := (\underline{X}_t^{\tau})_{\#} \rho_0$, we have the bound
\[
W_2(\rho_t,\rho_t^{\tau}) \leq \|X(t) - \underline{X}_t^{\tau}\|_{\mathbb{H}},
\]
and the following estimate holds:
\[
\|X(t) - \underline{X}_t^{\tau}\|_{\mathbb{H}} \leq
\begin{cases}
    \sqrt{3} \|X_0^{\tau} - \textnormal{Id}\|_{\mathbb{H}} 
    + \frac{\sqrt{33}}{2}\tau \| \nabla \phi^{\#}_{\rho_0}(X^{\tau}_0)\|_{\mathbb{H}} 
    , & \text{if } \lambda \geq 0, \\[1em]
    \sqrt{3} e^{- \lambda_{\tau} t} \|X_0^{\tau} - \textnormal{Id}\|_{\mathbb{H}}  
    + \sqrt{3}   C(\lambda,t,\tau)  \cdot \tau e^{-\lambda_{\tau} t} \cdot
    \|\nabla \phi^{\#}_{\rho_0}(X_0^{\tau})\|_{\mathbb{H}}, & \text{if } \lambda < 0.
\end{cases}
\]
where the lower order term
\begin{equation}
C(\lambda_{\tau},t,\tau) :=   (1+|\lambda_{\tau}|(t+\tau))e^{-\lambda_{\tau} \tau}+ \sqrt{\frac{3}{4} + (1+\frac{7}{3}|\lambda_{\tau}|(t+\tau)e^{-2\lambda_{\tau} \tau}}. \label{C_definition}
\end{equation}

\label{streamline_convergence}
\end{theorem}

\begin{proof}
By the arguments of \cite[Theorem 4.2.2]{ambrosio2005gradient}, we have the bound
\begin{equation}
\|\underline{X}^{\tau}_t - \underline{X}^{\eta}_t\|^2_{\mathbb{H}} 
\leq 3 \left( d^2_{\tau,\eta}(t;t) 
+ \|\underline{X}^{\tau}_t - \overline{X}^{\tau}_t\|^2_{\mathbb{H}} 
+ \|\underline{X}^{\eta}_t - \overline{X}^{\eta}_t\|^2_{\mathbb{H}} \right). \label{interpolate_to_d_bound}
\end{equation}
Since Assumption \ref{assume_rough} ensures that $X_0^{\eta} \to \text{Id}$ in $\mathbb{H}$ as $\eta \to 0$, applying Theorems \ref{neg_d_convg} and \ref{positive_d_convg} along with Lemma \ref{classical_stability_lemma} implies that $\underline{X}^{\tau}_t$ converges locally uniformly to a locally Lipschitz curve $X_t \in \mathbb{H}$. 

\smallskip

\noindent The $O(\tau)$ error estimates follow from \eqref{interpolate_to_d_bound}, Theorems \ref{neg_d_convg} and \ref{positive_d_convg}, Lemma \ref{classical_stability_lemma}, and letting $\eta \downarrow 0$.
\end{proof}

Now we will use this convergence and the discrete EVI to characterize the limits $X_t$ and $\rho_t$.

\begin{theorem}[Evolution Variational Inequality Characterization]  \label{solves_EVI}
Under the assumptions and notations of Theorem \ref{streamline_convergence}, if $X(t)$ denotes the $\mathbb{H}$-limit of $\underline{X}_t^{\tau}$, then $X(t)$ satisfies the Evolution Variational Inequality (EVI):
\begin{equation} 
\frac{1}{2} \left( \|X(t) - \xi\|^2_{\mathbb{H}} - \|X(s) - \xi\|^2_{\mathbb{H}} \right) 
+ \frac{\lambda}{2} \int_s^t \|X(a) - \xi\|^2_{\mathbb{H}} \, da 
+ \int_s^t \phi^{\#}_{\rho_0}(X(a)) \, da 
\leq  (t-s)\phi^{\#}_{\rho_0}(\xi),
\label{continuous_EVI}
\end{equation} 
for all $0 \leq s \leq t$ and $\xi \in \mathbb{H}$.  
\end{theorem}

\begin{proof}  
Given $0 \leq s < t$, choose integers $n, m \in \mathbb{N}$ such that $t \in [(n+1)\tau, (n+2)\tau)$ and $s \in [m \tau, (m+1)\tau)$. Summing the discrete EVI from Lemma \ref{discrete_EVI} and multiplying by $\tau$ gives
\begin{equation} 
\frac{1}{2} \left( \|\xi - \underline{X}_t^{\tau}\|^2_{\mathbb{H}} - \|\xi - \underline{X}_s^{\tau}\|^2_{\mathbb{H}} \right) 
+ \frac{\lambda \tau }{4}  \sum_{j=m}^{n} \left( \|\xi - X^{\tau}_j\|^2_{\mathbb{H}} + \|\xi - X^{\tau}_{j+1}\|^2_{\mathbb{H}} \right) \leq \tau \sum_{j=m}^{n} \left( \phi^{\#}_{\rho_0}(\xi) - \phi^{\#}_{\rho_0}(X^{\tau}_{j+1}) \right) + \tilde{\mathcal{R}}(\tau). 
\label{integrated_EVI}
\end{equation}

Here, the remainder term is given by
\[
\tilde{\mathcal{R}}(\tau) = \frac{\tau^2}{4} \left( \|\nabla \phi^{\#}_{\rho_0}(X^{\tau}_m)\|^2_{\mathbb{H}} - \|\nabla \phi^{\#}_{\rho_0}(X^{\tau}_{n+1})\|^2_{\mathbb{H}} \right) 
- \left( \frac{1}{2} + \frac{\lambda \tau}{4} \right) \sum_{j=m}^{n} \|X^{\tau}_j - X^{\tau}_{j+1}\|^2_{\mathbb{H}}.
\]  
Observe that $\tilde{\mathcal{R}}(\tau) \to 0$ as $\tau \to 0$ by \eqref{gradient_bounds} and Lemma \ref{classical_stability}. \smallskip

By Theorem \ref{streamline_convergence} and Assumption \ref{assume_rough}, we also have for any $T > 0$:
\begin{equation} 
\lim_{\tau \to 0} \sup_{t \in [0,T]} \|X_t - \underline{X}_t^{\tau}\|_{\mathbb{H}} = 0. 
\label{strong_X_convergence} 
\end{equation}
Hence, using the continuity of $\phi^{\#}_{\rho_0}$ from Assumption \ref{assume_rough}, along with Fatou's Lemma and \eqref{strong_X_convergence}, taking the limit in \eqref{integrated_EVI} yields the claim.
\end{proof}

Because $\mathbb{H}$ is a Hilbert Space, we have an equivalent formulation of the EVI.

\begin{theorem}[Strong gradient Flow Solution]  \label{strong_gf_solution}
Under the assumptions of Theorem \ref{solves_EVI}, the limit curve $X(t)$ satisfies the solves gradient flow equation in the strong sense:
\begin{equation}
\begin{cases} \dot{X}(t) = -\nabla \phi^{\#}_{\rho_0}(X(t)), \quad \text{for a.e. } t \in (0,\infty),  \\
X(0) = \textnormal{Id}
\end{cases}
\label{strong_gf_eqn}
\end{equation}
where $\dot{X}(t)$ is the Fréchet derivative of $X: [0,\infty) \to \mathbb{H}$.  
\end{theorem}

\begin{proof}  
By Theorem \ref{solves_EVI}, $X(t)$ satisfies the Evolution Variational Inequality \eqref{continuous_EVI}, which characterizes its gradient flow. Assumption \ref{assume_rough} ensures that $\phi^{\#}_{\rho_0}$ is Fréchet differentiable, so its subdifferential reduces to the singleton $\{\nabla \phi^{\#}_{\rho_0}(X(t))\}$. Moreover, Theorem \ref{streamline_convergence} guarantees that $X(t)$ is locally absolutely continuous in $\mathbb{H}$. The conclusion follows by applying \cite[Theorem 11.14]{ambrosio2021lectures}.
\end{proof}

\begin{theorem}[Weak Solution to the Continuity Equation] \label{limiting_PDE} 
Under the assumptions and notation of Theorem \ref{strong_gf_solution}, the limiting measure  
\[
\rho_t := (X(t))_{\#} \rho_0
\]
is a weak solution to the continuity equation \eqref{gradient_flow_eqn1} with initial data $\rho_0$.
\end{theorem}

\begin{proof}  
By \cite[Corollary 3.22]{gangbo2019differentiability}, we have that $\phi$ is Wasserstein differentiable and
\[
\nabla \phi^{\#}_{\rho_0}(\xi) = \nabla_{\mathbf{W}} \phi(\xi_{\#} \rho_0, \xi).
\] Let $v_t(x) := \nabla_{\mathbf{W}} \phi(\rho_t,x)$. Then applying Theorem \ref{strong_gf_solution}, for any $\varphi \in C^{1}_c(\R^d)$, we obtain:
\[
\frac{d}{dt} \int_{\R^d} \varphi(x) d\rho_t(x)  
= \frac{d}{dt} \int_{\R^d} \varphi(X_t) d\rho_0(x)  
= -\int_{\R^d} \langle \nabla \varphi(X_t), v_t(X_t) \rangle d\rho_0(x)
= -\int_{\R^d} \langle \nabla \varphi(x), v_t(x) \rangle d\rho_t(x).
\]
Since \(X_t \to \textnormal{Id}\) in \(\mathbb{H}\) as \(t \to 0\), we also have
\[
\limsup_{t \to 0} W_2(\rho_0,\rho_t) \leq \lim_{t \to 0} \|X_t - \textnormal{Id}\|_{\mathbb{H}} = 0.
\]
Thus, $\rho_t$ is a weak solution to \eqref{gradient_flow_eqn1} with initial data $\rho_0$ (see for instance \cite[Proposition 4.2]{santambrogio2015optimal}).
\end{proof}

\section{\texorpdfstring{Monotonicity and Exponential Decay under $L$-Smoothness and $\lambda>0$}{Monotonicity and Exponential Decay under L-smoothness and lambda>0}}
\label{further_stability}

In this section, we show that when $\lambda \geq 0$ and under an additional smoothness assumption on $\phi^{\#}_{\rho_0}$, the decay properties of $\phi$ and its gradient established in Section \ref{stability_section} can be further refined. Moreover, we improve the error estimate in Theorem \ref{streamline_convergence} by proving that the error vanishes exponentially in time when $\lambda>0$. \smallskip

The key to upgrading our previous result is that when $\lambda \geq 0$ and $\nabla \phi^{\#}_{\rho_0}$ is $L$-Lipschitz, we can obtain a lower bound for $||X_n^{\tau}-X_{n+1}^{\tau}||^2_{\mathbb{H}}$ in terms of $||\nabla \phi^{\#}_{\rho_0}(X^{\tau}_{n+1})||^2_{\mathbb{H}}.$

\begin{lemma}
Under Assumption \ref{assume_rough} and its notation, suppose $\lambda \geq 0$ and there exists a constant $L \geq 0$ such that
\[
\sup_{\xi_1 \neq \xi_2 \in \mathbb{H}} \frac{\|\nabla \phi^{\#}_{\rho_0}(\xi_1) - \nabla \phi^{\#}_{\rho_0}(\xi_2)\|_{\mathbb{H}}}{\|\xi_1 - \xi_2\|_{\mathbb{H}}} \leq L.
\]
Then if $\tau \leq \frac{1}{L}$, we have the following estimate for any $n \in \mathbb{N}$:
\begin{equation} \label{distance_lower_bound} \langle \nabla \phi^{\#}_{\rho_0}(X^{\tau}_n), \nabla \phi^{\#}_{\rho_0}(X^{\tau}_{n+1}) \rangle_{\mathbb{H}} \geq  (1-L \tau) ||\nabla \phi^{\#}_{\rho_0}(X^{\tau}_{n+1})||_{\mathbb{H}}^2  \textnormal{ and }
\|X_n^{\tau} - X_{n+1}^{\tau}\|_{\mathbb{H}}^2 \geq \tau^2(1-\frac{L \tau}{2} ) \|\nabla \phi^{\#}_{\rho_0}(X_{n+1}^{\tau})\|^2_{\mathbb{H}}. \end{equation} 
\end{lemma}
\begin{proof} We observe that the inner product inequality implies the other inequality because it along with using Lemma \ref{trapezoid_well_defined} with $\lambda \geq 0$ implies
\small
\[ ||X^{\tau}_n - X^{\tau}_{n+1}||^2_{\mathbb{H}} = \frac{\tau^2}{4}  \left( || \nabla \phi^{\#}_{\rho_0}(X^{\tau}_n)||_{\mathbb{H}} + || \nabla \phi^{\#}_{\rho_0}(X^{\tau}_{n+1})||_{\mathbb{H}} + 2 \langle \nabla \phi^{\#}_{\rho_0}(X^{\tau}_n) , \nabla \phi^{\#}_{\rho_0}(X^{\tau}_{n+1})\rangle_{\mathbb{H}}  \right) \geq \tau^2(1 - \frac{L \tau}{2} ) || \nabla \phi^{\#}_{\rho_0}(X^{\tau}_{n+1})||_{\mathbb{H}}, \]
\normalsize
Therefore, it suffices to show $\langle \nabla \phi^{\#}_{\rho_0}(X^{\tau}_n), \nabla \phi^{\#}_{\rho_0}(X^{\tau}_{n+1}) \rangle_{\mathbb{H}} \geq (1-L\tau) \cdot ||\nabla \phi^{\#}_{\rho_0}(X^{\tau}_{n+1})||_{\mathbb{H}}$. We now prove this bound. First observe that 
\[
\langle \nabla \phi^{\#}_{\rho_0}(X^{\tau}_{n+1}), \nabla \phi^{\#}_{\rho_0}(X^{\tau}_{n}) \rangle_{\mathbb{H}} 
= \|\nabla \phi^{\#}_{\rho_0}(X^{\tau}_n)\|^2_{\mathbb{H}} 
+ \langle \nabla \phi^{\#}_{\rho_0}(X^{\tau}_{n+1}) - \nabla \phi^{\#}_{\rho_0}(X^{\tau}_n), \nabla \phi^{\#}_{\rho_0}(X^{\tau}_n) \rangle_{\mathbb{H}}.
\]
Applying Cauchy-Schwarz, we obtain
\begin{equation}
\geq \|\nabla \phi^{\#}_{\rho_0}(X^{\tau}_n)\|^2_{\mathbb{H}} 
- L \|X^{\tau}_{n+1} - X^{\tau}_n\|_{\mathbb{H}} \cdot \|\nabla \phi^{\#}_{\rho_0}(X^{\tau}_n)\|_{\mathbb{H}} \geq (1-L\tau) \|\nabla \phi^{\#}_{\rho_0}(X^{\tau}_n)\|_{\mathbb{H}}^2.   \label{almost_desired_boundss}
\end{equation} In the second last inequality, we used $\lambda \geq 0$ along with \eqref{gradient_bounds} to obtain that
\[
\|X^{\tau}_{n+1} - X^{\tau}_n\|^2_{\mathbb{H}} \leq \tau^2 \|\nabla \phi^{\#}_{\rho_0}(X^{\tau}_n)\|^2_{\mathbb{H}}.\] Then the inner product inequality follows from \eqref{almost_desired_boundss}, $(1-L\tau) \geq 0$, and using the gradient norms are non-increasing in $n$ (see \eqref{gradient_bounds}) because $\lambda \geq 0$.
\end{proof}

Now we will use the Discrete EVI along with our above bounds to refine our stability properties:

\begin{lemma}[Refined Decay] \label{refined_decay}
Under the notation and assumptions of Lemma \ref{distance_lower_bound}, if $\tau \leq \frac{1}{L}$, then the following hold: \smallskip

\noindent 1. The energy is non-increasing:
\[
\phi^{\#}_{\rho_0}(X^{\tau}_{n+1}) +  \frac{\lambda}{2} ||X^{\tau}_n-X^{\tau}_{n+1}||^2_{\mathbb{H}} +  \tau\left(1- \frac{L \tau}{2} \right) \cdot ||\nabla \phi^{\#}_{\rho_0}(X^{\tau}_{n+1})||^2_{\mathbb{H}} \leq \phi^{\#}_{\rho_0}(X^{\tau}_n).
\]
2. The gradient norm decays exponentially:
\[
\|\nabla \phi^{\#}_{\rho_0}(X^{\tau}_n)\|^2_{\mathbb{H}} \geq e^{2\lambda_{\tau,L} \tau} \|\nabla \phi^{\#}_{\rho_0}(X^{\tau}_{n+1})\|^2_{\mathbb{H}}.
\] where
\begin{equation} \lambda_{\tau,L} := \frac{\log(1+\lambda\tau(2 - L \tau ))}{2 \tau} \geq 0 \label{lambda_tau_L}  \end{equation}
\end{lemma}

\begin{proof} Let us first show that the energy is non-increasing along the iterates. From \eqref{original_EVI} with $\lambda \geq 0$, we obtain
\[
\left( \frac{\lambda}{2}+\frac{2}{\tau} \right) \|X^{\tau}_n - X^{\tau}_{n+1}\|^2_{\mathbb{H}} 
+ \phi^{\#}_{\rho_0}(X^{\tau}_{n+1}) 
\leq \phi^{\#}_{\rho_0}(X^{\tau}_n) 
+ \langle \nabla \phi^{\#}_{\rho_0}(X_n^{\tau}), X^{\tau}_n - X^{\tau}_{n+1} \rangle_{\mathbb{H}}.
\]
Applying \eqref{trapezoid_implicit}, this simplifies to
\[
\phi^{\#}_{\rho_0}(X^{\tau}_{n+1}) 
+ \frac{\lambda}{2} ||X^{\tau}_n-X^{\tau}_{n+1}||^2_{\mathbb{H}} + \frac{\tau}{2} \left( \|\nabla \phi^{\#}_{\rho_0}(X^{\tau}_{n+1})\|^2_{\mathbb{H}} 
+ \langle \nabla \phi^{\#}_{\rho_0}(X^{\tau}_{n+1}), \nabla \phi^{\#}_{\rho_0}(X^{\tau}_{n}) \rangle_{\mathbb{H}} \right) 
\leq \phi^{\#}_{\rho_0}(X^{\tau}_n).
\]
The result then follows from Lemma \ref{distance_lower_bound}. \smallskip

For the exponential decay of the gradient, we use $\lambda$-convexity to see
\[
\langle \nabla \phi^{\#}_{\rho_0}(X^{\tau}_n) - \nabla \phi^{\#}_{\rho_0}(X^{\tau}_{n+1}), X^{\tau}_n - X^{\tau}_{n+1} \rangle_{\mathbb{H}} 
\geq \lambda \|X^{\tau}_n - X^{\tau}_{n+1}\|^2_{\mathbb{H}}.
\]
Applying \eqref{trapezoid_implicit}, it follows that
\[ \|\nabla \phi^{\#}_{\rho_0}(X^{\tau}_n)\|^2_{\mathbb{H}} 
\geq \|\nabla \phi^{\#}_{\rho_0}(X^{\tau}_{n+1})\|^2_{\mathbb{H}} 
+ \frac{2\lambda}{\tau} \|X^{\tau}_n - X^{\tau}_{n+1}\|^2_{\mathbb{H}}.
\]
Then, using Lemma \ref{distance_lower_bound}, along with $\lambda \geq 0$ and the above display, we conclude.
\end{proof}

\begin{lemma}[Refined Stability]  \label{refined_stability} Under the assumptions and notation of Lemma \ref{refined_decay},  we have
\begin{equation}  ||X^{\tau}_{j+1}-X^{\tau}_j||_{\mathbb{H}} \leq  \tau e^{-j  \lambda_{\tau,L} \cdot \tau} ||\nabla \phi^{\#}_{\rho_0}(X_0^{\tau})||_{\mathbb{H}} \label{distance_decay}.
\end{equation} In addition, if $N \tau \leq T$, we have the stability bound
\begin{equation}
\sum_{j=0}^{N} e^{2j \lambda_{\tau,L} \cdot \tau} \frac{W_2^2(\rho^{\tau}_{j+1},\rho^{\tau}_j)}{\tau} \leq \sum_{j=0}^{N} e^{2j \lambda_{\tau,L} \cdot \tau} \frac{\|X^{\tau}_{j+1} - X^{\tau}_j\|^2_{\mathbb{H}}}{\tau} \leq  T \|\nabla \phi^{\#}_{\rho_0}(X^{\tau}_0)\|^2_{\mathbb{H}}   \label{stability_with_exp} \end{equation} 
\end{lemma}

\begin{proof} 
Applying \eqref{trapezoid_implicit} along with Lemma \ref{refined_decay}, we obtain \eqref{distance_decay}. To derive \eqref{stability_with_exp}, we square \eqref{distance_decay} and sum over all terms and use Lemma \ref{same_measure_bound}. \end{proof}

To conclude this section, we refine the constants in Theorem \ref{streamline_convergence} for $ \lambda \geq 0 $. Since $ \lambda_{\tau,L} \leq \lambda $, $\phi$ remains $ \lambda_{\tau,L} $-convex. Applying Theorem \ref{differential_inequality_different_time_steps} with $ \lambda_{\tau,L} \geq 0 $, we can estimate $ d_{\tau,\eta}(t;t) $ using Gronwall’s Inequality \cite[Lemma 4.1.8]{ambrosio2005gradient}. Indeed, we have the bound
\begin{equation} \label{positive_d_interpolated_bound} d_{\tau,\eta}(T;T) \leq e^{-\lambda_{\tau} T} \left( d^2_{\tau,\eta}(0;0) + \sup_{t \in [0,T]} \int_0^T e^{2 \lambda_{\tau,L} s} \tilde{a}(s) ds \right)^{1/2} + e^{-\lambda_{\tau} T} \int_0^T e^{\lambda_{\tau} s} \tilde{b}(s) ds, \end{equation} where
\[
\tilde{a}(t) := \frac{\tau}{2} G^{\tau}_t + \frac{\eta}{2} G^{\eta}_t + 2(\mathcal{R}_t^{\tau} + \mathcal{R}_t^{\eta}) + \lambda_{\tau,L}( (\mathcal{D}_t^{\tau})^2 + (\mathcal{D}_t^{\eta})^2), \quad \text{and} \quad \tilde{b}(t):= \lambda_{\tau,L}(\mathcal{D}_t^{\tau}+\mathcal{D}_t^{\eta}).
\]
Now let us estimate the integral terms:

\begin{lemma}
Under the notation and assumptions of Lemma \ref{refined_stability}, then the following estimates hold for any $T \geq 0$:
\[
\begin{aligned}
    \int_0^T e^{2\lambda_{\tau,L} t} G^{\tau}_t \, dt 
    &\leq \tau \cdot \frac{\lambda}{\lambda_{\tau,L}} \cdot(1+2\lambda(T+\tau)) \cdot ||\nabla \phi^{\#}_{\rho_0}(X_0^{\tau})||^2_{\mathbb{H}} , \\
    \int_0^T e^{2\lambda_{\tau,L} t} \mathcal{R}^{\tau}_t \, dt 
    &\leq \tau^2 \left( \frac{\lambda}{\lambda_{\tau,L}} \cdot \left[ \frac{1}{4} + \frac{\lambda}{2} (T+\tau) \right]+ \frac{1}{2} e^{2 \lambda_{\tau, L} \tau} \right) \cdot ||\nabla \phi^{\#}_{\rho_0}(X_0^{\tau})||^2_{\mathbb{H}}.
\end{aligned}
\]
\[
\begin{aligned}
    \int_0^T e^{2\lambda_{\tau,L} t} (\mathcal{D}_t^{\tau})^2 \, dt 
    &\leq \frac{\lambda}{\lambda_{\tau,L}} \tau^2(T+\tau) ||\nabla \phi^{\#}_{\rho_0}(X_0^{\tau})||^2_{\mathbb{H}}, \\
    \int_0^T e^{\lambda_{\tau,L} t} (\mathcal{D}_t^{\tau}) \, dt 
    &\leq (T+\tau) \tau \sqrt{\frac{\lambda}{\lambda_{\tau,L}} } ||\nabla \phi^{\#}_{\rho_0}(X_0^{\tau})||_{\mathbb{H}}.
\end{aligned}
\]

\label{error_term_bounds}
\end{lemma}

\begin{proof} Choose $N \in \N$ to be the smallest integer such that $T \leq N \tau$.\smallskip

\noindent \textbf{Step 1: } $\int_0^T e^{2\lambda_{\tau,L} t} (\mathcal{D}_t^{\tau})^2dt$. \smallskip

Observe that 
\[ \int_0^T e^{2\lambda_{\tau,L} t} (\mathcal{D}_t^{\tau})^2dt \leq \int_0^{N \tau}  e^{2\lambda_{\tau,L} t} d^2(\overline{X}_t^{\tau} \underline{X}_t^{\tau}) dt =\frac{e^{2\lambda_{\tau,L} \tau}-1}{2\lambda_{\tau,L}} \sum_{j=0}^{N-1} ||X^{\tau}_j - X^{\tau}_{j+1}||^2_{\mathbb{H}} \cdot e^{2j \lambda_{\tau,L} \tau}. \] In addition, the definition $\lambda_{\tau,L}$ implies that $e^{2 \lambda_{\tau,L} \tau}-1 \leq 2 \lambda \tau$, so using the above bound with \eqref{stability_with_exp} implies the desired bound. \smallskip

\noindent \textbf{Step 2:} $\int_0^T e^{\lambda_{\tau,L} t} (\mathcal{D}_t^{\tau}) dt$. \smallskip

The bound follows directly from Holder's inequality and the bound from Step 1.

\smallskip
\noindent \textbf{Step 3:} $\int_0^T e^{2 \lambda_{\tau,L} t} G^{\tau}_t \, dt$ \smallskip

We will use the notation $\alpha_j$ from \eqref{alpha_def}. Then as $G^{\tau}_t \geq 0$ because $\lambda \geq 0$ from \eqref{gradient_bounds}, we see that
\begin{equation} \int_0^T e^{2 \lambda_{\tau,L} t} G^{\tau}_t dt \leq \sum_{j=0}^{N-1} (\alpha_{j} - \alpha_{j+1}) \int_{j \tau}^{(j+1) \tau} e^{2 \lambda_{\tau,L} t} dt =\frac{e^{2\lambda_{\tau,L} \tau} - 1 }{2\lambda_{\tau,L}} \sum_{j=0}^{N-1} (\alpha_j - \alpha_{j+1}) e^{2j \lambda_{\tau,L} \tau}. \label{bounds12312} \end{equation} To control the sum, we observe that
\[  \sum_{j=0}^{N-1} (\alpha_j - \alpha_{j+1}) e^{2j \lambda_{\tau,L} \tau} =  \left( \alpha_0 - \alpha_N e^{2N \lambda_{\tau,L} \tau} \right) + \sum_{j=0}^{N-1} \alpha_{j+1}( e^{2(j+1)\lambda_{\tau,L} \tau} - e^{2j \lambda_{\tau,L} \tau}) \leq \alpha_0  + \alpha_0 (e^{2 \lambda_{\tau,L} \tau}-1) N .     \]  In the last inequality we used Lemma \ref{refined_decay}. Then again the definition of $\lambda_{\tau,L}$ implies that $e^{2 \lambda_{\tau,L} \tau}-1 \leq 2\lambda \tau$, so we see that from \eqref{bounds12312}
\[ \int_0^T e^{2\lambda_{\tau,L} t} G_t^{\tau} dt \leq  \tau \cdot \frac{\lambda}{\lambda_{\tau,L}} \cdot (1+2\lambda(T+\tau)) \cdot ||\nabla \phi^{\#}_{\rho_0}(X_0^{\tau})||^2_{\mathbb{H}}. \]

\noindent \textbf{Step 4:} $\int_0^t e^{\lambda t} \mathcal{R}_t^{\tau} dt$ \smallskip

By using Lemma \ref{R_bounds} and that $\lambda \geq 0$, we have that
\[ \int_0^T e^{2\lambda_{\tau,L} t}\mathcal{R}_t^{\tau} \leq  \frac{\tau}{4} \int_0^T e^{2\lambda_{\tau,L}t } G_t^{\tau} dt -  \frac{1}{\tau}\int_0^T e^{2\lambda_{\tau,L}t } (\ell_{\tau}(t)-\frac{1}{2})(\mathcal{D}_t^{\tau})^2 dt = (I) + (II). \]  We can control $(I)$ by using the error bound derived in Step 3. To control $(II)$, observe that
\[ \int_{j \tau}^{(j+1) \tau} \left( \frac{1}{2} - \ell_{\tau}(t) \right) e^{2 \lambda_{\tau,L} t} dt\leq 0. \] Hence, we see from Lemma \ref{refined_stability}
\[ (II) \leq \frac{1}{2 \tau} \cdot d^2(X_{N-1}^{\tau},X_{N}^{\tau}) \cdot \int_{(N-1)\tau}^T e^{2 \lambda_{\tau,L} t} \leq \frac{\tau^2}{2} ||\nabla \phi(X^{\tau}_N)||^2_{\mathbb{H}} \cdot e^{2 \lambda_{\tau,L} (N+1)\tau}  \]
\[ \leq \frac{\tau^2}{2} e^{2 \lambda_{\tau,L} \tau} \cdot ||\nabla \phi(X_0^{\tau})||^2_{\mathbb{H}}. \] The bounds on $(I)$ and $(II)$ allows us to conclude.

 \end{proof}

By using \eqref{positive_d_interpolated_bound}, Lemma \ref{refined_stability}, and Lemma \ref{error_term_bounds} along with a similar argument as in Theorem \ref{streamline_convergence}, we obtain

\begin{theorem}[Refined Error Estimates] \label{refined_convergence} Under the notation and assumptions of Lemma \ref{refined_decay} and Theorem \ref{streamline_convergence}, if $\lambda \geq 0$ and $\tau \leq 1/L$, we have that
\[ W_2(\rho_t,\rho_t^{\tau}) \leq ||X(t)-\underline{X}_t^{\tau}||_{\mathbb{H}} \leq \sqrt{3}e^{-\lambda_{\tau,L} t} (||\textnormal{Id} - X_0^{\tau}||_{\mathbb{H}} + \tilde{C}(\lambda,t,\tau) \cdot \tau ||\nabla \phi^{\#}_{\rho_0}(X_0^{\tau})||_{\mathbb{H}}),  \] where we have
\[ \tilde{C}(\lambda,t,\tau) := e^{\lambda_{\tau,L} \cdot \tau} + \sqrt{\frac{\lambda}{  \lambda_{\tau,L}}} \cdot (t+\tau) + \sqrt{ \frac{\lambda}{\lambda_{\tau,L}} \left(1+2\lambda(t+\tau) \right) + \lambda (t+\tau)  + e^{2 \lambda_{\tau,L} \tau}   }     \]
\end{theorem}

\section{Numerical Experiments} \label{sec:experiments}

In this section, we present numerical experiments for our trapezoidal time discretization scheme applied to the the energy functional $U$ from Example~\ref{C2_functionals}. We will take $\chi_{\sigma}$ to be the density of a $\mathcal{N}(0,\sigma^2\textnormal{Id})$ random variable.  When $\rho = \frac{1}{N} \sum_{j=1}^{N} \delta_{X_j}$ is an empirical measure, the computations from Example~\ref{C2_functionals} implies
\begin{equation}
\nabla_{\mathbf{W}} U(\rho, X_i)
= \nabla_{\mathbf{W}} \mathcal{F}(\rho, X_i)
+ \nabla V(X_i)
+ \frac{1}{N} \sum_{j=1}^{N} \nabla W(X_i - X_j).
\label{grad_E_formula_particle}
\end{equation} where
\[
\nabla_{\mathbf{W}} \mathcal{F}(\rho, X_i)
= \frac{1}{N} \sum_{j=1}^{N} 
\Big( f'\big( (\rho \ast \chi_{\sigma})(X_i) \big)
+ f'\big( (\rho \ast \chi_{\sigma})(X_j) \big) \Big)
\, \nabla \chi_{\sigma}(X_i - X_j).
\]

This formulation is particularly convenient for numerical implementation, 
as both the energy $U(\rho)$ and its Wasserstein gradient 
$\nabla_{\mathbf{W}} \mathcal{E}(\rho)$ can be evaluated explicitly from the particle positions 
$\{X_i\}_{i=1}^N$.  In practice, this allows for efficient computation of the energy and its gradients without any need for grid-based interpolation or density estimation.

\smallskip
Moreover, push-forwards are especially simple in this particle representation: 
if $\rho = \frac{1}{N}\sum_{i=1}^{N} \delta_{X_i}$ and $T:\mathbb{R}^d \to \mathbb{R}^d$ is a map, then
\[
T_{\#}\rho = \frac{1}{N}\sum_{i=1}^{N} \delta_{T(X_i)}.
\]
Hence both the evaluation of $\mathcal{E}$ and the application of transport maps 
can be implemented directly at the level of the particle positions. \smallskip

\begin{figure}[!htbp]
  \centering
  \resizebox{0.8\linewidth}{!}{%
  \begin{minipage}{\linewidth}
    \centering

    % --- Row 1 ---
    \begin{minipage}{0.40\linewidth}
      \centering
      \includegraphics[width=0.9\linewidth]{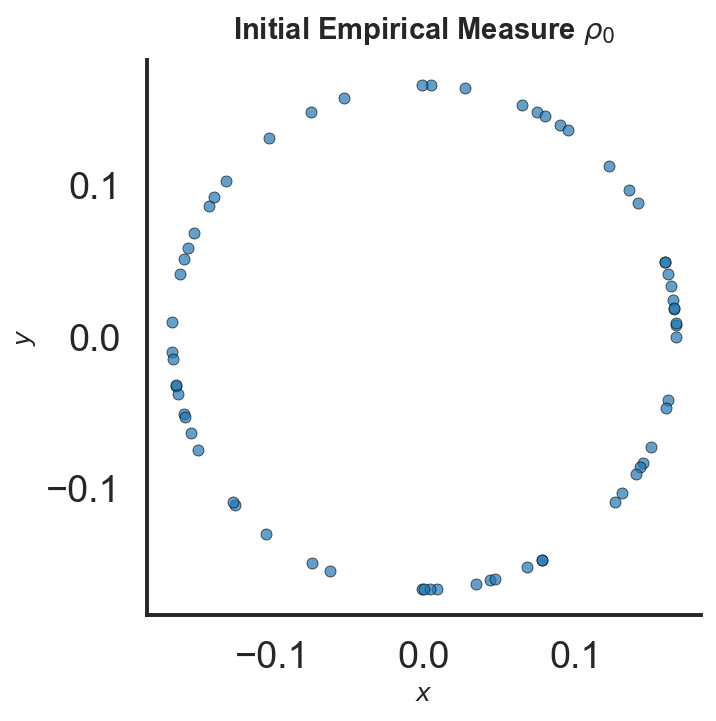}
      \captionof*{figure}{\scriptsize (a) Initial empirical measure $\rho_0$}
    \end{minipage}\hfill
    \begin{minipage}{0.4\linewidth}
      \centering
      \includegraphics[width=\linewidth]{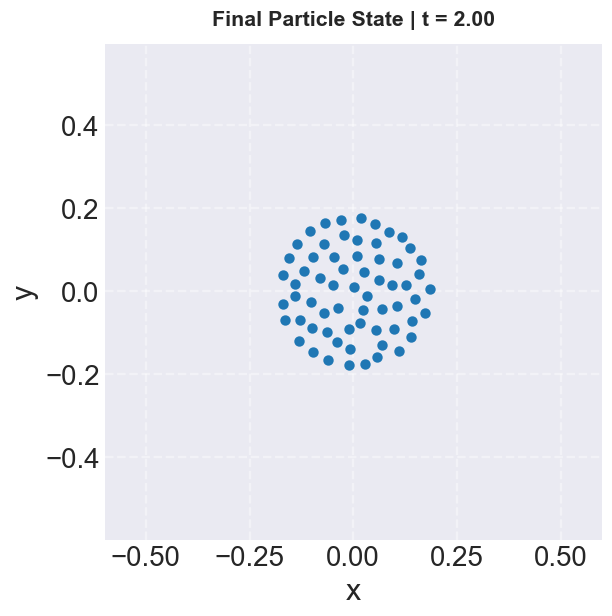}
      \captionof*{figure}{\scriptsize (b) Final particle locations at $t=2$}
    \end{minipage}

    \vspace{0.75em}

    % --- Row 2 + 3 combined ---
    \begin{minipage}{0.62\linewidth}
      \centering
      \includegraphics[width=\linewidth]{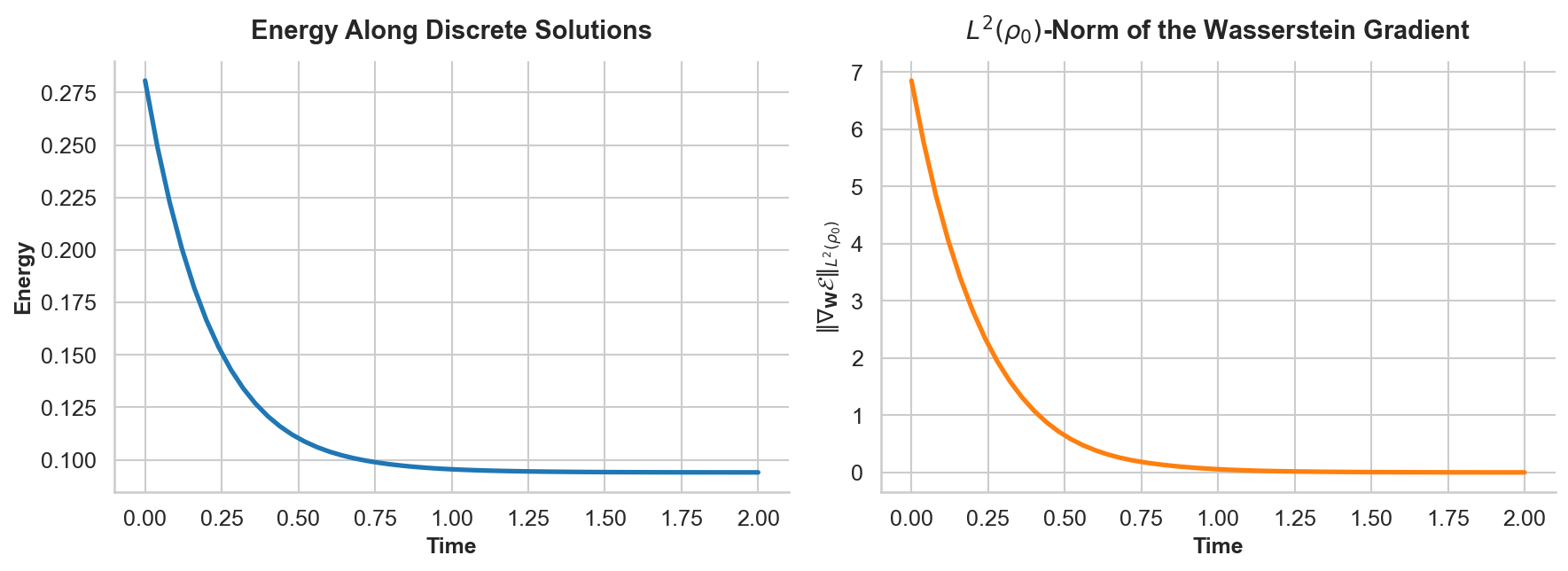}
      \captionof*{figure}{\scriptsize (c) Energy $\mathcal{E}$ and 
      $\|\nabla_{\mathbf{W}}\mathcal{E}\|_{L^2(\rho_0)}$ versus time}
    \end{minipage}\hfill
    \begin{minipage}{0.34\linewidth}
      \centering
      \includegraphics[width=\linewidth]{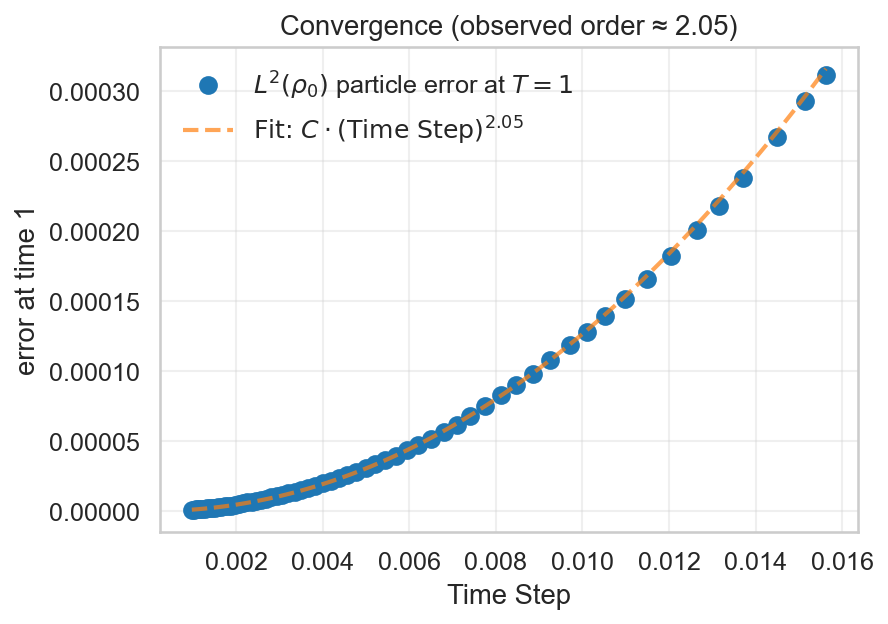}
      \captionof*{figure}{\scriptsize (d) $L^2(\rho_0)$ particle error}
    \end{minipage}

  \end{minipage}%
  } % end resizebox

  \caption{\small Numerical results for the trapezoidal scheme with potentials 
  $(f,V,W) = (0, |x|^2, -\tfrac{1}{4\pi}\log(\e^2 + |x|^2))$,  
  where $\e = 10^{-2}$, $N = 64$, and $\tau = 1/25$.  
  (a)–(b) show the particle configurations at the beginning and end of the simulation;  
  (c) displays the energy and Wasserstein gradient norm over time;  
  (d) reports the estimated time–step convergence rate where 
  $\tau_{\text{Ref}} = 1/4096$, while the other $\tau \in [1/1024,1/64]$.}
  \label{fig:vw_full}
\end{figure}

\begin{figure}[!htbp]
  \centering
  \resizebox{0.8\linewidth}{!}{%
  \begin{minipage}{\linewidth}
    \centering

    % --- Row 1: initial vs final ---
    \begin{minipage}{0.40\linewidth}
      \centering
      \includegraphics[width=.9\linewidth]{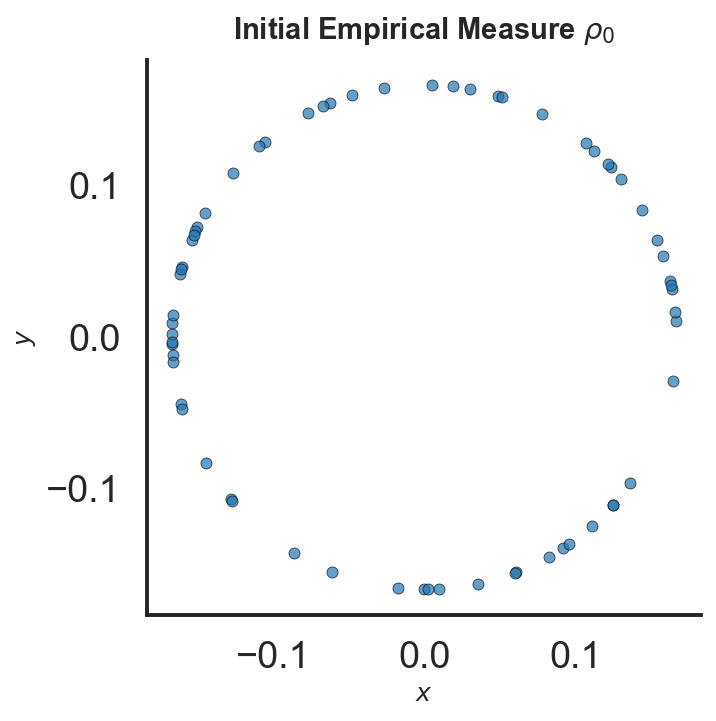}
      \captionof*{figure}{\scriptsize (a) Initial empirical measure $\rho_0$}
    \end{minipage}\hfill
    \begin{minipage}{0.40\linewidth}
      \centering
      \includegraphics[width=\linewidth]{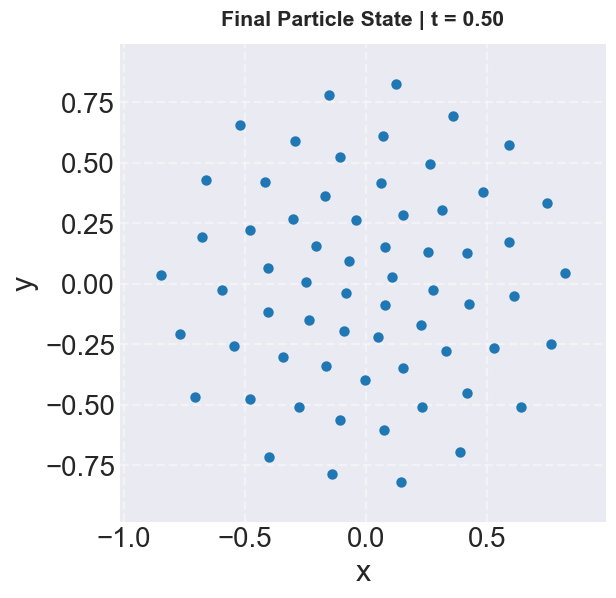}
      \captionof*{figure}{\scriptsize (b) Final particle locations at $t=0.1$}
    \end{minipage}

    \vspace{1.0em}

    % --- Row 2 + 3 combined ---
    \begin{minipage}{0.62\linewidth}
      \centering
      \includegraphics[width=\linewidth]{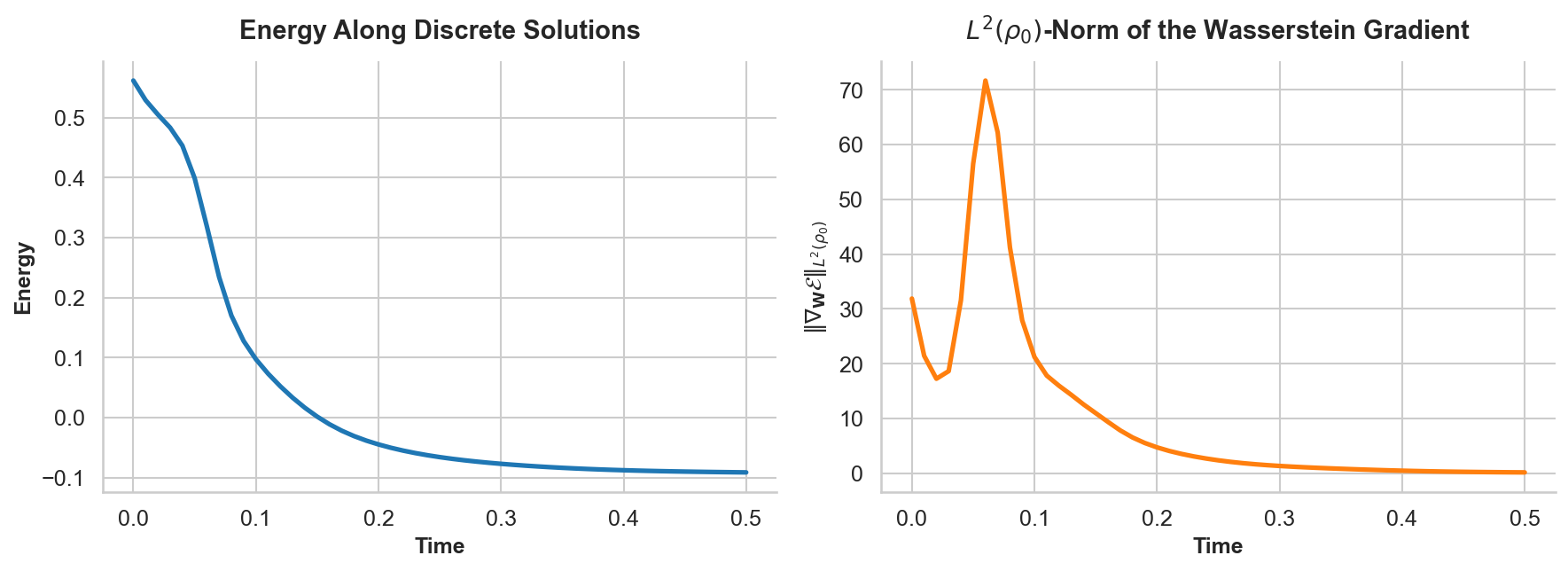}
      \captionof*{figure}{\scriptsize (c) Energy $\mathcal{E}$ and 
      $\|\nabla_{\mathbf{W}}\mathcal{E}\|_{L^2(\rho_0)}$ versus time}
    \end{minipage}\hfill
    \begin{minipage}{0.34\linewidth}
      \centering
      \includegraphics[width=\linewidth]{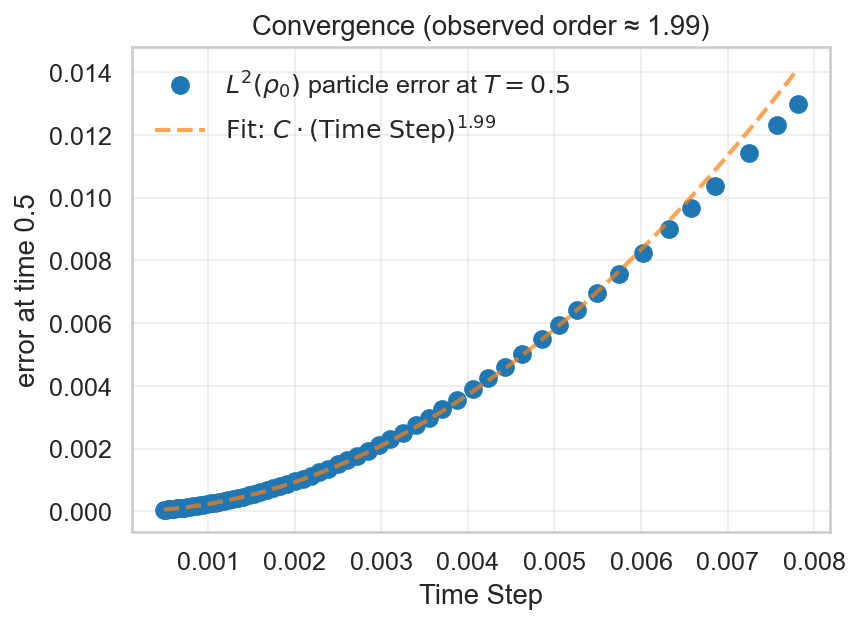}
      \captionof*{figure}{\scriptsize (d) $L^2(\rho_0)$ particle error}
    \end{minipage}

  \end{minipage}%
  } % end resizebox

  \caption{\small Numerical results for the trapezoidal scheme with potentials 
  $(f,V,W) = (\tfrac{1}{2}\log(x+\e), |x|^2, 0)$,  
  where $\e = 10^{-2}$, $\sigma = 1/10$, $N = 64$, and $\tau = 1/100$.  
  (a)–(b) show the particle configurations at the beginning and end of the simulation;  
  (c) displays the energy and Wasserstein gradient norm over time;  
  (d) reports the estimated time–step convergence rate where 
  $\tau_{\text{Ref}} = 0.5/4096$, while the other $\tau \in [0.5/1024,0.5/64]$.}
  \label{fig:fvw_full}
\end{figure}

Numerically, to obtain $X^{\tau}_{n+1}$ we employed the variational formulation 
from~\eqref{trapezoid_var_problem}. 
The minimization of~\eqref{trapezoid_var_problem} was carried out via gradient descent 
in the particle positions, using the explicit expression for the gradient 
given in~\eqref{grad_E_formula_particle}. 
In practice, the descent was initialized by an explicit Euler step 
$X^{\tau}_{n} - \tau \nabla_{\mathbf{W}} \mathcal{E}(\rho^{\tau}_n, X^{\tau}_n)$ 
and the gradient descent was iterated until the $L^2(\R^d;\rho_0)$ norm of successive iterates fell below a prescribed tolerance. \smallskip

To estimate the numerical convergence rate, we compared the discrete particle configurations obtained with time steps $\tau$ and $\tau_{\mathrm{Ref}}$ at the same final time $T>0$, satisfying $N\tau = M\tau_{\mathrm{Ref}} = T$. 
The corresponding error is given by $\| X^{\tau}_{N} - X^{\tau_{\mathrm{ref}}}_{M} \|_{L^2(\mathbb{R}^d;\rho_0)}.$ These experiments also numerically confirm Theorem \ref{higher_order_convergence_theorem}, exhibiting the predicted $O(\tau^2)$ convergence rate.

\bibliographystyle{alpha}
\bibliography{bib}

\end{document}